\documentclass{amsart}
\usepackage{myarticle-macrosSharedArxiv}
\usepackage{graphicx,subfig,enumerate,bm}
\usepackage{color}

\usepackage{amsmath,amssymb,graphicx,amsthm,bm}
\makeatletter
\DeclareRobustCommand*\cal{\@fontswitch\relax\mathcal}
\makeatother

\newcommand{\totalSize}{N}
\newcommand{\tildeM}{\widetilde{M}}
\newcommand{\tildeGM}{\widetilde{GM}}
\newcommand{\blockE}{\mathsf{E}}
\newtheorem{nekDef}{Definition}

\newcommand{\dataset}{\mathcal{X}}
\newcommand{\outputs}{\mathcal{B}}
\newcommand{\parameter}{\mathcal{P}}
\newcommand{\nuisance}{\mathcal{N}}

\newcommand{\graphV}{V}
\newcommand{\graphE}{E}
\newcommand{\graphG}{G}
\newcommand{\vdmS}{S}
\newcommand{\vdmD}{D}
\newcommand{\vdmC}{C}
\newcommand{\vdmI}{I}

\newcommand{\NN}{\mathbb{N}}

\newcommand{\RR}{\mathbb{R}}
\newcommand{\CC}{\mathbb{C}}
\newcommand{\EE}{\mathbb{E}}

\renewcommand{\sgn}{\mbox{sgn}}

\newcommand{\ud}{\textup{d}}

\renewcommand{\argmin}{\operatornamewithlimits{argmin}}

\title{Graph connection Laplacian and random matrices with random blocks}

\author[M.~El Karoui]{Noureddine El Karoui}
\address{Department of Statistics, UC Berkeley, Berkeley, CA}
\email{nkaroui@berkeley.edu}
\author[H.-T.~Wu]{Hau-Tieng Wu}
\address{Department of Mathematics, University of Toronto, Toronto, Ontario, Canada M5S 2E4}
\email{hauwu@stanford.edu}

\begin{document}

\maketitle

\begin{abstract}
Graph connection Laplacian (GCL) is a modern data analysis technique that is starting to be applied for the analysis of high dimensional and massive datasets.
	Motivated by this technique, we study matrices that are akin to the ones appearing in the null case of GCL, i.e the case where there is no structure in the dataset under investigation. Developing this understanding is important in making sense of the output of the algorithms based on GCL. 
	We hence develop a theory explaining the behavior of the spectral distribution of a large class of random matrices, in particular random matrices with random block entries of fixed size. Part of the theory covers the case where there is significant dependence between the blocks. Numerical work shows that the agreement between our theoretical predictions and numerical simulations is generally very good.
\end{abstract}

\section{Introduction}

Graph connection Laplacian (henceforth GCL) \cite{singer_wu:2012,Bandeira_Singer_Spielman:2013,Chung_Zhao_Kempton:2013} is a new and promising data analysis framework for high dimensional and massive datasets. GCL and its variants are currently being used for the analysis of the cryo-Electron-microscope (cryoEM) problem \cite{singer_zhao_shkolnisky_hadani:2011,Hadani_Singer:2011a,singer_wu:2012,Tzeneva:2011}, dynamical systems analysis \cite{Sonday_Singer_Kevrekidis:2013}, sensor network localization \cite{Cucuringu_Singer_Cowburn:2012}, multi-view reconstruction \cite{Martinec_Pajdla:2007,Wang_Singer:2013}, vectorized PageRank \cite{Chung_Zhao_Kempton:2013}, ptychographic imaging problem \cite{Alexeev_Bandeira_Fickus_Mixon:2014,Marchesini_Tu_Wu:2014} and other problems. GCL is a conceptual and practical generalization of graph Laplacian (GL) methods, which are now fairly commonly applied in statistical and machine learning. The idea underlying these methods is that the data to be analyzed -- though high-dimensional in the form given to us (think of a high-resolution picture/image as a point in the high-dimensional Euclidean space) has in fact a relatively low-dimensional structure. An idealized model is that the data points actually live on a low-dimensional geometric object, for example, a manifold, embedded in a high-dimensional Euclidean space. This model can be understood as a generalization of the model considered in principal components analysis, where the data points are locally assumed to - approximately -  live on a low-dimensional affine space embedded in a high-dimensional Euclidean space. 

Under this low dimensional assumption, GL works by doing variants of kernel principal components analysis on data points. When the low-dimensional geometric object is a manifold, GL gives a way to estimate spectral properties of the heat kernel of its Laplace-Beltrami operator. It can be shown that a particular algorithm based on GL, the diffusion map (DM), is theoretically capable \cite{berard_besson_gallot:1994,berard2,coifman_lafon:2006,singer_wu:2013} to recover the geometrical and topological structure of this manifold. 

GCL works on more complicated objects compared with GL -- an extra group relationship between the data points is assumed in addition to the low dimensional geometric structure. 
Take the image data for example. Depending on the problem, in that setting, two rotated versions of the same image may be considered to be different or the same objects. In fact, while they appear very different in data-analytic methods simply operating on data point, such as GL, we might view them as a single object by taking the rotation into account. In other words, we ``group'' the dataset into subsets so that images in each subset are the same up to rotation. In GCL, the subsets are viewed as a new point cloud and the group relationship among the images are included in the analysis. One direct benefit of taking the rotation into account is dimension reduction of the dataset since the rotation information is taken into account. 
When the point cloud can be parametrized by a manifold and the group relationship between the data points encodes the parallel transport information of a chosen bundle, GCL-based methods allow us to estimate the heat kernel of the connection Laplacian associated with the chosen bundle (the natural and relevant differential-geometric object - see \cite[Chapter 1]{Berline_Getzler_Vergne:2004}) on the manifold. One particular algorithm based on GCL is  vector diffusion maps (VDM), which is a generalization of popular algorithms like Laplacian Eigenmap \cite{belkin_niyogi:2003,belkin_niyogi:2005,belkin_niyogi:2007}, DM \cite{coifman_lafon:2006,singer:2006}, etc,
and provides tools to understand the local geodesic distance on the manifold. Practically, GCL-based ideas can be algorithmically implemented efficiently. We give more numerical details on GCL later in the introduction. 

Though the motivation of GCL is linked to particular datasets \cite{singer_zhao_shkolnisky_hadani:2011,Hadani_Singer:2011a,singer_wu:2012}, GCL is naturally also interesting as a tool addressing problems arising in many modern aspects of statistical learning, applied mathematics and what is increasingly called data-science. A common target of analysis for these fields is big datasets, which are more and more prevalent. In addition to its size, the explosively increasing dimensionality of the data and the inevitable noise inside the data are two important features of modern datasets. To handle the high dimensionality of the dataset, it is common to assume the existence of low dimensional structure or sparsity inside the data, and design the analysis based on these assumptions. To deal with the noise in this {\it large $p$} (i.e many measurements per observation), {\it large $n$} (i.e many observations) setup, we have to take its peculiar and sometime counterintuitive behavior into account. It is therefore natural to seek to understand the impact of ``noise''  - broadly defined - on the behavior of our algorithms. As readers familiar with random matrix theory will know, the impact of noise in high-dimension can be dramatic (see e.g \cite{nekKernels} and \cite{nekInfoPlusNoiseKernelMatrices10}). This important issue is the focus of the current paper. As we will see, GCL gives rise to a specific kind of random matrices. We study generalizations of this kind of matrices and show that they have sometimes surprising properties. 

\subsection{On a framework leading to GCL}
In this subsection, we discuss a specific applied framework leading to the GCL to better motivate GCL-based algorithms. Estimating the intrinsic parameters from an {\it observation dataset} $\dataset$ is a main task in data analysis. We call the set of intrinsic parameters the {\it parameter space} $\parameter$. As discussed above, it is commonly believed that the parameter space has a lower dimensional structure. In many cases, this low dimensional structure attenuates various difficulties, for example reducing the impact of what is sometime called ``the curse of dimensionality". The {\it model space} associated with the parameter space, denoted as $\outputs$, is commonly assumed to be the range of a transformation from $\parameter$. In some cases, $\dataset$ is the same as $\outputs$, and we may be interested in inferring $\parameter$ from $\dataset$ for the sake of extracting more understanding about the system.
However, in some other cases there might be a gap between $\dataset$ and $\outputs$. Indeed, in addition to noise, $\outputs$ might be different from $\dataset$ due to the deformation introduced by the way we observe the system or other natural processes. We call the deformation {\it the nuisance parameter} and denote it as $\nuisance$. 

Consider the following example. Take a density function describing the finger and wrist of a child containing growth plates. We assume that the growth plates are parameterized by a set of parameters $v\in\parameter\subset\RR^d$, where $d\geq 1$. The same child is examined at several (say $T$) points in time. $\{v_i\}_{i=1}^T$ describe his/her growth plates at these different observations/experiment times. We denote the density function as $\psi_{v}:\RR^3\to \RR^+$ emphasizing the dependence on ${v}$. At the different time stamps, we take X-ray images of $\psi_{v}$ from a fixed rotational position $R_0\in SO(3)$ by the X-ray transformation, denoted as 
\begin{align}\label{definition:Xray_transform}
T_{\psi_{v}}(R_0)(x,y):=\int_{-\infty}^\infty \psi_v(xR_0^1+yR_0^2+tR_0^3)\ud t,\,\,\mbox{ where }R_0=\left[\begin{array}{ccc}
| & | & |\\
R_0^1&R_0^2&R_0^3\\
| & | & |
\end{array}
\right].
\end{align}
$(x,y)\in\RR^2$ and we call the unit vector $R_0^3$ the {\it projection direction}. We would like to study how the growth plates are parametrized by $\parameter$. In this problem, the model space is $\outputs=\{T_{\psi_{v}}(R_0);\,v\in \parameter\}$. However, the observation dataset $\dataset$ might be different from $\outputs$ since the child's hand might vary from time to time, that is, $\dataset=\{T_{\psi_{v}}(R(v)R_0);\,v\in\parameter,\,R(v)\in SO(3)\}$, where $R(v)$ is a random sample of $SO(3)$. In other words, the model space depends only on $\parameter$, while the observation dataset depend on not only on $\parameter$ but also on $SO(3)$. The extra parameters are the nuisance parameters describing how the patient rotates his hand, that is, $\nuisance=SO(3)$. 

In general, we may formulate the above framework through the concept of  {\it group action}. Consider a metric space $Y$ equipped with a metric $d$, and a group $G$ with the identity element $e$. We call $Y$ the total space and $G$ the structure group. The left group action of $G$ on $Y$ is a map from $G\times Y$ onto $Y$
\begin{align}
G\times Y\to Y,\quad (g,x)\mapsto g\circ x
\end{align}
so that $(gh)\circ x=g\circ (h\circ x)$ is satisfied for all $g,h\in G$ and $x\in Y$ and $e\circ x=x$ for all $x\in Y$. 
The right group action can be defined in the same way and can be constructed by composing the left group action with the inverse group operation, so it is sufficient to discuss left actions. Take the parameter space $\parameter$ and the model space $\outputs$. Suppose the observation dataset $\dataset$ is located in $Y$, the nuisance parameter is $G$ which acts on $Y$, and $\outputs=\dataset/G$. In other words, $\dataset$ is not only parameterized by $\parameter$, but also by $G$. Note that $\dataset=G\circ \outputs$ is a special case. In general, by the nature of the setup, the group action may be non-isometric, which corresponds to non-rigid deformations in the image registration literature.
From the data analysis viewpoint, removing these nuisance parameters is generally helpful, for example for dimension reduction and so on.

Besides the nuisance parameter, the underlying structure of $\parameter$ is important. In fact, even if $\nuisance=\emptyset$, or if we managed to remove $\nuisance$ from $\dataset$, the underlying structure of $\parameter$ might be informative. For example, in the X-ray transform (\ref{definition:Xray_transform}), the projection directions of all possible projection images are parametrized by the 2-dimensional sphere $S^2$, which contains rich geometric and topological structures. To take these non-trivial structures into account, {\it spectral methods} such as Laplacian Eigenmap or DM are commonly applied, and lots of successes have been reported. See, for example, \cite{belkin_niyogi:2003,belkin_niyogi:2005,coifman_lafon:2006,belkin_niyogi:2007} and the references therein. An additional benefit of spectral methods is that they are generally based on computationally efficient algorithms.

The importance of $\parameter$ and $\nuisance$ were discussed separately above. In some situations, the combination of $\parameter$ and $\nuisance$ might lead to further structural information about $\parameter$. One particular example is the {\it class averaging algorithm} aiming to improve the signal to noise ratio of the images collected from the cryoEM \cite{singer_zhao_shkolnisky_hadani:2011,Hadani_Singer:2011a,singer_wu:2012} so that the 3-dimensional structure of the molecule can be better reconstructed. To be more precise, the projection images of the X-ray transform are parameterized by $SO(3)$, while the parameter space of the class averaging algorithm is the projection direction $R_0^3 \in S^2$ embedded in $\mathbb{R}^3$ (see Equation \eqref{definition:Xray_transform}). Thus, under the above framework, if there is no (rotational) symmetry in the molecule described by $\psi$, $\dataset=T_\psi(SO(3))\cong SO(3)$, the nuisance parameter is $\nuisance=SO(2)$, the model space is $\outputs=T_\psi(SO(3))/SO(2)\cong S^2$ and $\parameter=S^2$ embedded in $\mathbb{R}^3$. Note that the information derived from the nontrivial combination of $\parameter$ and $\nuisance$ (rather than $\parameter$ alone), that is, $T_\psi(SO(3))$, allows us to obtain statistics describing non-trivial aspects of the geometric and topological structure of $\parameter$ by constructing the connection Laplacian of the tangent bundle of $\parameter=S^2$. As has been shown in \cite{singer_wu:2012}, the connection Laplacian is approximated by the GCL constructed from the dataset. The algorithm based on GCL which leads to the solution of the denoising problem in cryoEM - the class averaging algorithm - is VDM \cite{singer_wu:2012,wu:2012,singer_wu:2013}. 

\subsection{GCL: terminology and notations} \label{Section:Introduction:GCL}
We summarize the GCL algorithm considered in \cite{singer_wu:2012,singer_wu:2013} under the above framework, which motivates the block random matrix theory in this study. Suppose $\nuisance=O(m)$, where $m\in\NN$. Take a set of $n>0$ random samples from $\dataset$, denoted as $\dataset_n$, which corresponds to the finite random samples of the model space, denoted as $\outputs_n=\dataset_n/O(m)$. Construct a graph $\graphG=(\graphV,\graphE)$, where $\graphV$ represents $\outputs_n$, and build up an {\it affinity function} $w:\graphE\to \RR^+$ from the distance between pairs in $\outputs_n$. In addition to the affinity function $w$, build up a group-valued function $g:\graphE\to O(m)$, referred to as the {\it connection function}, quantifying the nuisance parameters among data on the vertex. Then, build up the $n\times n$ block matrix $S$ with $m\times m$ block as the weighted matrix, where the $(i,j)$-th entry of $S$ is:
\begin{align}\label{definition:vdmS}
\vdmS_{ij}=\left\{
\begin{array}{ll}
w(i,j)g(i,j) & \mbox{when }(i,j)\in E\\
0&\mbox{otherwise}
\end{array}\right.
\end{align}
and a $n\times n$ block diagonal matrix $D$ with the $i$-th diagonal block 
\begin{align}\label{definition:vdmD}
\vdmD_{ii}=\sum_{j\neq i}w(i,j)I_m.
\end{align}
The (normalized) GCL is defined as 
\begin{equation}\label{definition:vdmC}
	\vdmC:=\vdmI-\vdmD^{-1}\vdmS\;. 
\end{equation}
Analyzing the eigen-structure of the GCL leads to statistics describing $\parameter$, like the VDM and vector diffusion distance.

\subsection{Partial motivation for the paper: impact of noise on GCL}
Up to now, the discussion is based on the assumption that the observation dataset is noise free. When noise exists, we seek to understand how the noise influences the GCL, in particular in the large $p$, large $n$ setup. In general, if $Z_1,\ldots,Z_n\in\RR^p$ are i.i.d. random vectors, the problem we would study is formulated as a {\it kernel random matrix with random block structure} in the following way. For a group valued function $\mathsf{f}:\RR^p\times\RR^p\to {\mathcal G}$, where ${\mathcal G}$ is a matrix group, so that $\mathsf{f}(Z_i,Z_j)=\mathsf{f}(Z_j,Z_i)^*$, build up a symmetric matrix whose $(i,j)$-th block is
\[
S_{ij}=\mathsf{f}(Z_i,Z_j),
\]
where the statistical property of $Z_i$ and $\mathsf{f}$ depend on the application.
Note that when $\mathsf{f}$ has  range $\RR$ instead of ${\mathcal G}$ and $\mathsf{f}(Z_i,Z_j)=\mathsf{f}(Z_i^T Z_j)$, the problem turns into a kernel random matrix problem, which has been studied in \cite{nekKernels,Cheng_Singer:2013,Do_Vu:2013}. In the current paper, we are interested in how the noise influences the output of the GCL(-like) algorithm and consider the ``no signal'' situation as a ``null'' case situation and as an approximation of the very high noise situation. In a subsequent recent paper, we have considered the ``information plus noise" situation, see \cite{NEKHautieng2014CGLRobust}.

One particular motivation and application of the current work is the class averaging algorithm in the cryoEM problem. Due to the high noise nature of the problem, it is important to know how much confidence we have on the result by studying the null hypothesis that there is no signal in the data and only noise.
Although we do not focus on fully answering this question, in Section \ref{Section:ConsequenceOfMaster} and Section \ref{app:subsec:vdmComps}, the careful analysis of the GCL matrix motivated by this problem when all the signals are purely independent noise provides a clue. We mention that the GCL built up in this way is a block random matrix with additional dependent structure among the blocks introduced by the underlying low dimensional structure assumption and the way we prepare the data, and hence the more general random matrix theory is needed.  

The above motivating problem opens the following general question -- when a random matrix follows additional structures, does its empirical spectral distribution still converge to the semi-circle law? We are particularly interested in the case where the random matrix is a block matrix, such that each block is randomly sampled from a matrix group and there are some dependence relationship among blocks.

\subsection{Our contribution -- Random matrices with random blocks}

In light of the structure of the matrix $\vdmS$ appearing in GCL (see Equations \eqref{definition:vdmS}), it is natural to study random matrices whose entries are random blocks. A natural question is to understand the limiting spectral distribution (LSD) of $S$ and $D$ in that context. This is naturally a way for us to understand what the limiting spectrum of $\vdmC$ (see Equation \eqref{definition:vdmC}) should look like when our dataset is basically ``pure noise".

We first show that under a quite general condition, the limiting distribution of certain random matrices with random block entries is asymptotically deterministic (the central result in this direction is Theorem \ref{thm:veryGeneralMasterThm}). Indeed, we allow the block entries to have a significant amount of dependence. Furthermore, Theorem \ref{thm:veryGeneralMasterThm} applies more generally to non-block random matrices.

Next, we would like to quantify the deterministic limiting distribution. As a first approximation of this problem, we show that under the ``strip independent condition'', it is enough to understand a Gaussian counterpart to the matrix we are studying - this is the content of Theorem \ref{thm:replacementByGaussian}. 
As a second approximation of this problem, we develop a theory that characterizes the limiting distribution of random matrices with independent random blocks. As an application of this result, we get Theorem \ref{thm:cvToWignerUnderSimpleStructure}, which shows convergence to the Wigner semi-circle law for a very broad class of random block matrices. We \cite{NEKHolgerShrinkage11,nekCorrEllipD} and other researchers \cite{PajorPasturPub09,GoetzeTikhomirov05,ChatterjeeSimpleInvariance05,ChatterjeeGeneralizationLindeberg06} have used similar ideas or rather subset of these ideas in the past (at a high-level at least). Moreover, our results allow a fair amount of dependence between the block-entries of the random matrices we consider, in contrast to e.g \cite{GirkoMatrixEquationForBlockRandomMatrices95} which requires independent blocks.

In the GCL discussed above, the blocks are quite dependent since the random block matrix we consider has a ``kernel''-like structure. We give some more details on the ``null'' case - i.e $Z_i$'s are pure noise and contain no-signal -  when it is applied to the class averaging algorithm in the Appendix, where we show for instance that the marginal distribution of $g_{i,j}$ is the Haar distribution on $SO(2)$.

Of course, a large amount of work is still needed to tackle the particular problems we care about here. Tackling the spectral distribution of the matrix $C$  appearing in GCL for the simulations we considered requires a number of specialized computations - some of which depend on the specifics of certain algorithms etc and are not broadly informative. Hence, we do not carry out all these computations here and we study instead a broad class of related models. 
From our numerical work, it is clear that the results we get are relevant to the issues encountered in GCL and inform our thinking about this class of algorithms.

\subsection{Organization of the paper} The paper is organized as follows. In Section \ref{sec:theory}, we develop a theory to explain the behavior of the LSD of many random matrices, including random matrices with random block entries - a fair amount of dependence being allowed. We also present a number of situations 
that are intuitively close to the null case (i.e $Z_i$'s are pure noise) where our theory applies.  The analyses of the LSD under the independent strip structure condition and independent block condition are discussed in Sections 3 and 4. Section 4 also provides detailed examples and sufficient conditions for our results of Sections 2 and 3 to go through. In Section \ref{sec:numericalWork}, we present numerical work to investigate the agreement between our theoretical results and the results of numerical simulations. The Appendix contains a number of needed reminders, results and proofs. 

\textbf{Notations} We denote by $\opnorm{M}$ the largest singular value of the matrix $M$. We use the sign $\equalInLaw$ to denote equality in law.

\section{Theory}\label{sec:theory}
We consider $\totalSize\times \totalSize$ matrices that are Hermitian with above diagonal  ``block-rows" (or ``strips") of height bounded by a constant $d$. An example are 
matrices with i.i.d block entries but the theory we develop here applies more generally. We apply later this general theory to block matrices. There has been work on block matrices with various patterns, applying mostly to situations where the blocks are large, i.e their size is going to infinity asymptotically \cite{BollaRandomBlockMatricesLargeBlock04,OrabyHermitianRandomBlockMatrices07,slowFadingMIMOSytems08}, or to specific patterns (\cite{DetteRandomBlockMatrices} for a tridiagonal example). Our work extends and generalizes to much more involved dependence structures some results of Girko \cite{GirkoMatrixEquationForBlockRandomMatrices95}. 

Our analysis is based on Stieltjes transforms. Throughout, we call the Stieltjes transform of a $\totalSize\times \totalSize$ matrix $M$
$$
m_M(z):=\frac{1}{\totalSize}\trace{\left(M-z\id_{\totalSize}\right)^{-1}}\;,
$$
where $z\in\CC$ so that $\imag{z}=v>0$ and $\id_{\totalSize}$ is the $\totalSize\times \totalSize$ identity matrix. In much of our analysis, $d$ is held fixed. We will let $\totalSize$ grow to infinity.

\subsection{Asymptotically deterministic character of LSD}

We have the following ``master" theorem.
\begin{theorem}\label{thm:veryGeneralMasterThm}
Suppose that the $\totalSize\times \totalSize$ Hermitian matrix $M$ is such that, for independent random variables $\{Z_i\}_{i=1}^n$ and a matrix valued function $f$,
$$
M=f(Z_1,\ldots,Z_n)\;.
$$
Suppose further that for all $1\leq i \leq n$, there exists a matrix $N_i$ such that 
$$
N_i=f_i(Z_1,\ldots,Z_{i-1},Z_{i+1},\ldots,Z_n)
$$
and $\rank{M-N_i}\leq d_i$. Let $z\in \mathbb{C}^+$ and $\imag{z}=v>0$.
Then, for any $t>0$, 
\begin{equation}\label{eq:controlStieltjesTVeryGalMasterTheorem}
P\left(\left|m_M(z)-\Exp{m_M(z)}\right|>t\right)\leq C \exp\left(-c\frac{\totalSize^2 v^2 t^2}{\sum_{i=1}^n d_i^2}\right)\;,
\end{equation}
where $C$ and $c$ are two constants that do not depend on $n$ nor $d_i$'s.
\end{theorem}
The previous theorem is a McDiarmid-style result for Stieltjes transforms - based on rank approximations. The conceptual approach is similar to the one we used in \cite{nekCorrEllipD}.

\begin{proof}
Let us call ${\cal F}_i=\sigma\left\{Z_k\right\}_{1\leq k \leq i}$ (i.e the $\sigma$-field generated by the random variables $Z_k$'s for $k\leq i$) and ${\cal F}_0=\left\{\emptyset\right\}$.
Of course, 
$$
m_M(z)-\Exp{m_M(z)}=\sum_{i=1}^n \Big[\Exp{m_M(z)|{\cal F}_{n-i+1}}-\Exp{m_M(z)|{\cal F}_{n-i}}\Big]\;.
$$
This is clearly a sum of martingale differences, by construction. Let us call $m_M^{(i)}(z)=\frac{1}{N}\trace{(N_{i}-z\id)^{-1}}$. Note that under our assumptions, $N_{i}$ is independent of $Z_i$, since it involves only $\{Z_k\}_{k\neq i}$. Therefore,
$$
\Exp{m_M^{(i)}(z)|{\cal F}_i}=\Exp{m_M^{(i)}(z)|{\cal F}_{i-1}}\;.
$$
Hence, 
\begin{align*}
\Exp{m_M(z)|{\cal F}_{n-i+1}}-\Exp{m_M(z)|{\cal F}_{n-i}}&=\Exp{m_M(z)-m_M^{(n-i+1)}(z)|{\cal F}_{n-i+1}}\\
&-\Exp{m_M(z)-m_M^{(n-i+1)}(z)|{\cal F}_{n-i}}\;.
\end{align*}

Our assumptions also guarantee that $M_{(i)}=M-N_i$ is of rank at most $d_i$. 
Lemma \ref{lemma:controlDiffStieltjesFiniteRankPerturb} in the Appendix gives 
\begin{align*}
\left|\Exp{m_M(z)-m_M^{(n-i+1)}(z)|{\cal F}_{n-i+1}}\right| \leq \frac{d_i}{\totalSize v} \;, \text{ and }
\left|\Exp{m_M(z)-m_M^{(n-i+1)}(z)|{\cal F}_{n-i}}\right| \leq \frac{d_i}{\totalSize v}\;.
\end{align*}
Therefore, 
$$
\left|\Exp{m_M(z)|{\cal F}_{n-i+1}}-\Exp{m_M(z)|{\cal F}_{n-i}}\right|\leq 2\frac{d_i}{\totalSize v}\;.
$$
Hence, $m_M(z)-\Exp{m_M(z)}$ is a sum of bounded martingale differences. Applying the Azuma-Hoeffding inequality (\cite{ledoux2001} and \cite{nekCorrEllipD} to deal with the details we have to handle here), we get that, for any $t>0$
$$
P\left(\left|m_M(z)-\Exp{m_M(z)}\right|>t\right)\leq C \exp\left(-c\frac{\totalSize^2 v^2 t^2}{\sum_{i=1}^n d_i^2}\right)\;,
$$
as announced in the Theorem. 
\end{proof}

Theorem \ref{thm:veryGeneralMasterThm} yields a simple proof of the following result that plays a central role in our work. 

\begin{theorem}\label{thm:controlStieltjesTMasterTheorem}
Suppose the $N\times N$ Hermitian matrix $M$ can be written
$$
M=\sum_{1\leq i,j \leq n} \Theta_{i,j}\;,
$$
where $\Theta_{i,j}=f_{i,j}(Z_i,Z_j)$ is a $N\times N$ matrix and the random variables $\{Z_i\}_{i=1}^n$ are independent. ($f_{i,j}$'s are simply matrix valued functions of our random variables.)  
Let $M_i$ be the Hermitian matrix 
$$
M_i=\Theta_{i,i}+\sum_{j\neq i} \left(\Theta_{i,j}+\Theta_{j,i}\right)\;.
$$
Assume that $\rank{M_i}\leq d_i$.  Let $z\in \mathbb{C}^+$ and $\imag{z}=v>0$.
Then, for any $t>0$, 
\begin{equation}\label{eq:controlStieltjesTMasterTheorem}
P\left(\left|m_M(z)-\Exp{m_M(z)}\right|>t\right)\leq C \exp\left(-c\frac{\totalSize^2 v^2 t^2}{\sum_{i=1}^n d_i^2}\right)\;,
\end{equation}
where $C$ and $c$ are two constants that do not depend on $\totalSize$, $n$ nor $d_i$'s. 
\end{theorem}
\begin{proof}
Let us call 
$$
N_i=M-M_i\;.
$$
It is clear that $N_i=f_i(Z_1,\ldots,Z_{i-1},Z_{i+1},\ldots,Z_n)$. In other words, $N_i$ does not depend on $Z_i$. By our assumption on $\rank{M_i}=\rank{M-N_i}$, we see that the hypotheses made in Theorem \ref{thm:veryGeneralMasterThm} are satisfied in the context of Theorem \ref{thm:controlStieltjesTMasterTheorem}. Therefore, the conclusions of Theorem \ref{thm:veryGeneralMasterThm} apply here, too, and Theorem \ref{thm:controlStieltjesTMasterTheorem} is shown. 
\end{proof}

\subsubsection{Consequences of Theorem \ref{thm:controlStieltjesTMasterTheorem}}\label{Section:ConsequenceOfMaster}
The following consequences of Theorem \ref{thm:controlStieltjesTMasterTheorem} are tailored towards our applications to ``random-strip" matrices and GCL-like matrices.
\begin{corollary}\label{coro:controlStieltjesTSumOfIndepMatrices}
Suppose the $N\times N$ Hermitian matrix $M$ can be written
$$
M=\sum_{i=1}^n M_i\;,
$$
where $M_i$ are independent with $\rank{M_i}\leq d_i$. Let $z\in \mathbb{C}^+$ and $\imag{z}=v>0$.
Then, for any $t>0$, 
\begin{equation}\label{eq:controlStieltjesTSumOfIndepMatrices}
P\left(\left|m_M(z)-\Exp{m_M(z)}\right|>t\right)\leq C \exp\left(-c\frac{N^2 v^2 t^2}{\sum_{i=1}^n d_i^2}\right)\;,
\end{equation}
where $C$ and $c$ are two constants that do not depend on $n$ nor $d_i$'s. 
\end{corollary}

\begin{proof}
The corollary is a simple consequence of Theorem \ref{thm:controlStieltjesTMasterTheorem}. Indeed, we can apply Theorem \ref{thm:controlStieltjesTMasterTheorem} with $M_i=\Theta_{i,i}$ and $\Theta_{i,j}=\Theta_{j,i}=0$ if $i\neq j$ to get Corollary \ref{coro:controlStieltjesTSumOfIndepMatrices}. The ``latent variable" $Z_i$ is simply in this case the vector of elements of $\Theta_{i,i}$. 
\end{proof}

As a simple consequence of the previous corollary, we have the following result which is important for the rest of the paper. 

\begin{theorem}\label{thm:concentrationOfStieltjesTransforms}
Suppose the Hermitian matrix $M$ has a ``strip'' structure, i.e it is composed of $n$ strips of size $d\times \totalSize$, where $\totalSize=nd$, and the portions of strips that are above the ($d\times d$ block-) diagonal are independent. Let $z\in \mathbb{C}^+$ with $\imag{z}=v>0$. 
Then, for constants $C$ and $c$ that do not depend on $n$, $d$ or our model, we have 
$$
\forall t>0\;, \; \; P(|m_M(z)-\Exp{m_M(z)}|>t)\leq C \exp(-c nv^2t^2)\;.
$$
In the case where $\Exp{m_M(z)}$ has a limit, the convergence of $m_M(z)$ (and hence the spectral distribution of $M$) is in the sense of a.s convergence.
\end{theorem}

\begin{proof}
This theorem is a simple consequence of Corollary \ref{coro:controlStieltjesTSumOfIndepMatrices} where the matrix $M_i$ correspond to the element of the $i$-th strip that is above the diagonal and to the corresponding Hermitian transpose. Here $d_i\leq 2d$ for all $i$. The a.s. convergence of the spectral distribution is an immediate consequence the Borel-Cantelli lemma. See \cite{nekCorrEllipD} for details. 
\end{proof}

In some situations that are more complicated (for instance the GCL), we will need to be able to handle more dependent structures within the matrix $M$. The following corollary is relevant to those cases. It is targeted towards kernel-like structures. 

\begin{corollary}\label{coro:controlStieltjesLatentVariablesCaseBlockMatrices}
Suppose the Hermitian matrix $M$ is a $n\times n$ block matrix with square $\tilde{d}_i\times \tilde{d}_i$ blocks which can be written as
$$
M[i,j]=\mathsf{f}_{ij}(Z_i,Z_j)\;,
$$
where $Z_i$'s are independent random variables and $\mathsf{f}_{ij}$ are deterministic functions, but could depend on $i$ and $j$. Then Theorem \ref{thm:controlStieltjesTMasterTheorem} applies with $d_i=2\tilde{d}_i$. 
\end{corollary}

This type of kernel-like matrices is of particular interest to us - as GCL and its building blocks naturally give rise to such matrices.

\begin{proof}
In this situation, the $N\times N$ matrix $\Theta_{i,j}$ of Theorem \ref{thm:controlStieltjesTMasterTheorem} is simply the matrix consisting of 0's except in its $(i,j)$ block (corresponding to the matrix $M$'s $(i,j)$ block) where it is equal to $M[i,j]$.

Define $M_i$ as in Theorem \ref{thm:controlStieltjesTMasterTheorem}. All we have to do to show the validity of the corollary is therefore to verify that $\rank{M-M_{i}}\leq d_i=2\tilde{d}_i$. It is clear that $M-M_{i}$ is a matrix that contains only 0's except on its $i$-th block row and column.
Of course, if $A$ is a $N\times N$ Hermitian matrix of the form
$$
A=\begin{pmatrix}
A_{11} & A_{12}\\
A_{21} & 0_{(N-d)\times (N-d)}
\end{pmatrix}\;,
$$
then $\rank{A}\leq 2d$. Indeed, $A$ can be written by using at most 2$d$ vectors (and their transposed version). So any vector $v$ orthogonal to these 2$d$ vectors is such that $Av=0_N$. Hence, $\rank{M-M_{i}}\leq 2\tilde{d}_i=d_i$. We can therefore apply Theorem \ref{thm:controlStieltjesTMasterTheorem} under the hypotheses stated in our corollary. 
\end{proof}

\begin{remark}
Kernel random matrix analyses in \cite{nekInfoPlusNoiseKernelMatrices10,nekKernels,Cheng_Singer:2013,Do_Vu:2013} are special cases of Corollary \ref{coro:controlStieltjesLatentVariablesCaseBlockMatrices}. Indeed, the entries of the kernel random matrix $A$ of size $n\times n$ are
\[
A(i,j)=f(Z_i^TZ_j),
\]
where $f$ is a real-valued function. The extra freedom considered in Corollary \ref{coro:controlStieltjesLatentVariablesCaseBlockMatrices}, that is, the fact that the group-valued function $\mathsf{f}_{ij}$ depends in an arbitrary manner on $Z_i$ and $Z_j$ and not only on $Z_i^TZ_j$, means that the computation of the LSD (if it exists) can be very complicated.  However, our Corollary \ref{coro:controlStieltjesLatentVariablesCaseBlockMatrices} shows the deterministic character of the spectral distribution in the large $n$ limit in great generality.
\end{remark}

\begin{remark}
In the class averaging algorithm in the cryo-EM problem, one needs to work with matrices with block-entries defined through
$$
g_{i,j}=\argmin_{g\in SO(2)}\norm{Z_i-g\circ Z_j}_2 \quad\text{and}\quad d_{i,j}^2=\min_{g\in SO(2)}\norm{Z_i-g\circ Z_j}_2^2\;.
$$
Corollary 2.2 clearly applies to matrices with block entries of the form $M[i,j]=f(d_{i,j}^2)g_{i,j}$, for $f$ a function from $\mathbb{R}$ to $\mathbb{R}$. Hence, it is a useful tool for developing an understanding of certain aspects of the class averaging algorithm.
\end{remark}

\begin{remark}
Let us call $M[i,i]$ the blocks on the diagonal of the block-diagonal of the matrix $M$. Let us call $M^{(0)}$ the matrix obtained by replacing the block diagonal entries of $M$ by $0_{d\times d}$ and leaving the other elements of $M$ intact. 
We first note that when $M$ is such that $\sup_{1\leq i\leq n} \opnorm{M[i,i]}=\lo_P(1)$, Weyl's inequality gives 
$\opnorm{M-M^{(0)}}=\lo_P(1)$. The spectral distributions of $M$ and $M^{(0)}$ are therefore the same in the large $N$ limit. So we will often assume that the block-diagonal of $M$ is 0 - and effectively work with $M^{(0)}$ - keeping in mind that this assumption can be removed at very low cost provided the block diagonal entries of $M$ do not grow too fast. ($M$ will eventually take the form $M=\mathsf{M}/\sqrt{N}$, where $\mathsf{M}$ has independent strips with distributions independent of $n$ (except for the size of the strips). So assuming that $\opnorm{M[i,i]}=\lo_P(1)$ will turn out to be rather minimal.)
\end{remark}

\section{Characterizing the limit: dependence on low-order moments}

In this section, we analyze the LSD of a given random block matrix under the ``independent strip structure'' condition. Note that it can be viewed as an approximation of the more general random block matrices, for example, the one considered in Corollary \ref{coro:controlStieltjesLatentVariablesCaseBlockMatrices}. 

Our main result in this section, Theorem \ref{thm:replacementByGaussian}, says that the general random matrices we consider can be understood by simply considering random matrices with Gaussian entries that have a covariance structure that match the low-order moments of the random matrices we consider.

Let $M_n$ be a $N\times N$ matrix, where $N=nd$. We write 
$$
M_n=\begin{pmatrix}
M_n(1)\\
M_n(2)\\
\vdots\\
M_n(n)
\end{pmatrix}\;,
\text{ where } M_n(i) \in \mathbb{R}^{d\times \totalSize}\;. 
$$
We refer to the matrices $M_n(i)'s$ as block rows or {\it strips}.
We further write the block row/strip $M_n(i)$ as 
$$
M_n(i)=\big(\underbrace{{\cal R}_n(i)}_{i\times d} \;\; \underbrace{{\cal M}_n(i)}_{(n-i)\times d}\big)\;, \text{ where }{\cal R}_n(i) \in \mathbb{R}^{d\times (id)} \;.
$$ 
${\cal M}_n(i)$ is the $d\times (n-i)d$ strip that is on the $i$-th block row above the block diagonal of $M_n$. 
We call 
$$
\tildeM_n(i)=\big(0_{d\times (id)}\;\;{\cal M}_n(i)\big)\;,
$$ 
which is a $d\times \totalSize$ matrix.

\paragraph{Assumption B1} 
Let $z\in\mathbb{C}^+$ and $v=\imag{z}>0$.
If $\Gamma$ is a real, symmetric, deterministic matrix, we assume that  
\begin{equation}\label{eq:assumptionB11}
\frac{1}{nd}\Exp{\opnorm{\tildeM_n(i) (\Gamma-z\id)^{-1} \tildeM_n(i)\trsp-\Exp{\tildeM_n(i) (\Gamma-z\id)^{-1} \tildeM_n(i)\trsp}}}\leq \frac{R_i}{v}\;,
\tag{Assumption-B1.1}
\end{equation}
where $R_i\in \mathbb{R}_+$ is independent of $\Gamma$.  
We also assume that $\tildeM_n(i)$ is such that there exists a function $K_i$ of $z$
such that 
\begin{equation}\label{eq:assumptionB12}
\opnorm{\frac{1}{nd}\Exp{\tildeM_n(i)(\Gamma -z\id)^{-2}\tildeM_n(i)\trsp}}\leq K_i(z)\;.
\tag{Assumption-B1.2}
\end{equation}
We assume that $K_i$ is bounded in $i$.  
We finally assume that 
\begin{equation}\label{eq:assumptionB13}
\frac{1}{nd}\Exp{\opnorm{\tildeM_n(i) (\Gamma-z\id)^{-2} \tildeM_n(i)\trsp-\Exp{\tildeM_n(i) (\Gamma-z\id)^{-2} \tildeM_n(i)\trsp}}}\leq \frac{R_i}{v^2}\;,
\tag{Assumption-B1.3}
\end{equation}
where $R_i\in \mathbb{R}_+$ is independent of $\Gamma$. \\

\paragraph{About \ref{eq:assumptionB11}} Note that the matrix $\tildeM_n(i) (\Gamma-z\id)^{-1} \tildeM_n(i)\trsp$ is $d\times d$ and $d$ is assumed to be fixed in our analysis. So if we call $v_{k,j}(i)$ the $(k,j)$ entry of 
$$
\frac{1}{nd}\left(\tildeM_n(i) (\Gamma-z\id)^{-1} \tildeM_n(i)\trsp-\Exp{\tildeM_n(i) (\Gamma-z\id)^{-1} \tildeM_n(i)\trsp}\right)\;,
$$ 
a simple way to check that \eqref{eq:assumptionB11} holds for models under consideration is to verify that
$\sup_{1\leq k,j\leq d}\Exp{|v_{k,j}(i)|}\leq R_i/v$; in this case \eqref{eq:assumptionB11} holds with $R_i$ replaced by $d R_i$, since the operator norm of a symmetric matrix is smaller than the maximum $l_1$ norm of its rows (\cite{hj}, p.313). 

\paragraph{About \ref{eq:assumptionB12}} We note that the assumption about $K_i$ is easily satisfied: for instance, if we assume that there exists a constant $C_i$ such that for any deterministic unit vector $u$, $\opnorm{\Exp{\tildeM_n(i)uu\trsp \tildeM_n(i)\trsp}}\leq C_i$, then after diagonalizing $\Gamma$, we see that, if $\imag{z}=v$,
$K_i(z)=C_i/v^2$ is a valid choice. 

We are now in position of stating our main theorem.
\begin{theorem}\label{thm:replacementByGaussian}
Let $M_n$ be an $n\times n$ Hermitian block matrix of size $d\times d$, with random block-rows/strips satisfying Assumption B1. Assume that $\Exp{M_n}=0$ and that its block diagonal is 0. Call $\mathsf{m}_n(z):=m_{M_n/\sqrt{nd}}(z)$, the Stieltjes transform of $M_n/\sqrt{nd}$. 

Let $GM_n$ be a block matrix with Gaussian random blocks, with mean 0. Call $\mathsf{gm}_n(z):=m_{GM_n/\sqrt{nd}}(z)$ and suppose that Assumption B1 is satisfied for it, too.

Call ${\cal M}_n(i)$ the $d\times (n-i)d$ random matrix corresponding to the $i$-th block row of $M_n$ above the diagonal. Call $\tildeM_n(i)=[0_{d\times (id)} {\cal M}_n(i)]$. Call ${\cal GM}_{n}(i)$ and $\tildeGM_n(i)$ the corresponding matrices for the matrix $GM_n$. Assume that the block rows/strips of $M_n$ and $GM_n$ above the diagonal (i.e ${\cal M}_n(i)$'s and ${\cal GM}_n(i)$'s in our notation) are independent.

Assume furthermore that for all $1\leq i \leq n$ and for any deterministic (unit) vector $u$, 
\begin{equation}\label{eq:keyRequirementLindeberg}
\Exp{\tildeM_n(i)uu\trsp \tildeM_n(i)\trsp}=\Exp{\tildeGM_n(i)uu\trsp \tildeGM_n(i)\trsp}\;.
\end{equation}
Then 
\begin{equation}\label{eq:approxMeanStieltjesT}
\left|\Exp{\mathsf{m}_n(z)-\mathsf{gm}_n(z)}\right|\leq \frac{1}{nd}\sum_{i=1}^n R_i g(z,K_i)\;,
\end{equation}
where $g(z,K_i)=2d(2+K_i(z))\frac{1}{v^3}$. 
In particular, if $\sum_i R_i/n\tendsto 0$ as $n\tendsto \infty$, the LSD of $M_n$ is the same as that of $GM_n$. 
\end{theorem}
\begin{proof}
We use the Lindeberg method, where we replace the $i$-th block row and column by a Gaussian version satisfying Equation \eqref{eq:keyRequirementLindeberg}.  
We call $\blockE_i$ the $d\times N$ matrix with $\blockE_i(j,k)=\delta_{k,(i-1)d+j}$. Recall that the block diagonal of $M_n$ is 0. We note that 
$$
M_n=\sum_{i=1}^n \big[\blockE_i\trsp \tildeM_n(i)+\tildeM_n(i)\trsp \blockE_i\big]\;.
$$
We call, for $1\leq k \leq n-1$,  
$$
I_n(k)=\sum_{i=1}^k \left[\blockE_i\trsp \tildeM_n(i)+\tildeM_n(i)\trsp \blockE_i\right]+\sum_{i=k+1}^n \left[\blockE_i\trsp \tildeGM_n(i)+\tildeGM_n(i)\trsp \blockE_i\right]\;,
$$
and extend the definition for $k=0$ and $k=n$ with $I_n(0)=GM_n$ and $I_n(n)=M_n$. 
Clearly, 
\begin{gather*}
\trace{\left(\frac{GM_n}{\sqrt{\totalSize}}-z\id\right)^{-1}}-\trace{\left(\frac{M_n}{\sqrt{\totalSize}}-z\id\right)^{-1}}\\
=\sum_{k=0}^{n-1} \Big[\trace{\left(\frac{I_n(k)}{\sqrt{\totalSize}}-z\id\right)^{-1}}-\trace{\left(\frac{I_n(k+1)}{\sqrt{\totalSize}}-z\id\right)^{-1}}\Big]\;.
\end{gather*}
Therefore, 
$$
\left|\Exp{\mathsf{gm}_n(z)}-\Exp{\mathsf{m}_n(z)}\right|\leq \frac{1}{nd}\sum_{k=0}^{n-1}\left|\Exp{\trace{\left(\frac{I_n(k)}{\sqrt{\totalSize}}-z\id\right)^{-1}}-\trace{\left(\frac{I_n(k+1)}{\sqrt{\totalSize}}-z\id\right)^{-1}}}\right|.
$$
Now the conditions of Theorem \ref{thm:replaceFirstBlockRow} are satisfied, so Equation \eqref{eq:keyApproxCoro} applies to 
$$
\left|\Exp{\trace{\left(\frac{I_n(k)}{\sqrt{\totalSize}}-z\id\right)^{-1}}-\trace{\left(\frac{I_n(k+1)}{\sqrt{\totalSize}}-z\id\right)^{-1}}}\right|\;.
$$ 
Thus we have established Equation \eqref{eq:approxMeanStieltjesT}.
\end{proof}
\section{Application to block random matrices with independent block entries}
To show that our theory applies, we just need to verify that Assumptions B1 is satisfied. Let us translate it, in the context of block matrices, to easier-to-verify assumptions about the block matrices constituting the block entries. 
We remind the reader that we assume that $\Exp{M_n}=0$ and hence the same is true for the random block matrices we are dealing with. 

\subsection{On \ref{eq:assumptionB12}}
In the case of block random matrices with independent block entries, we write 
$$
\forall\, i\,, \, M_n(i)=\big(\underbrace{M_n[i,1]}_{d}\;\; \underbrace{M_n[i,2]}_d \ldots \underbrace{M_n[i,n]}_d\big)\;,
$$
where $M_n[i,k]$ is the $k$-th $d\times d$ block matrix on the $i$-th strip of $M_n$.

We now present an easy-to-verify condition to make sure that \ref{eq:assumptionB12} is satisfied in certain models of interest. (The notation $\tildeM_n(i)$ that appears below is introduced for instance in Theorem \ref{thm:replacementByGaussian} on p. \pageref{thm:replacementByGaussian}.)
\begin{lemma}
Suppose the matrix $M_n$ is constituted of $d\times d$ independent blocks and $\Exp{M_n}=0$. Call ${\cal S}^{i}_m[k,j]$ the (cross-) covariance matrix of the $j$-th row and $k$-th row of the $m$-th block matrix on the $i$-th strip of $M_n$ (i.e $M_n[i,m]$). 
If there exists $C$ such that $\opnorm{{\cal S}^i_m[j,k]}\leq C$, then for any real, symmetric, deterministic matrix $\Gamma$, 
$$
\opnorm{\frac{1}{nd}\Exp{\tildeM_n(i)(\Gamma -z\id)^{-2}\tildeM_n(i)\trsp}}\leq \frac{Cd}{v^2}\;.
$$
In other words, \ref{eq:assumptionB12} is satisfied with $K_i(z)=Cd/v^2$. 
\end{lemma}
\begin{proof}
	The matrix $\tildeM_n(i)$ is constituted of $n$ independent $d\times d$ matrices, $i$ of them being $0_{d\times d}$. 
	Let us call $A$ a generic $d\times N$ matrix, constituted of $d\times d$ independent blocks with $\Exp{A}=0$. We call $A[i]$ its $i$-th $d\times d$ block i.e 
	$$
	A=\big(\underbrace{A[1]}_{d}\;\; \underbrace{A[2]}_d \ldots \underbrace{A[n]}_d\big)\;.
	$$ 

	If $T$ is a deterministic $N\times N$ matrix, we call $T[i,j]$ its $(i,j)$-th $d\times d$ block.  We have 
	$$
	A T A\trsp=\sum_{1\leq i,j\leq n} A[i] T[i,j] A[j]\trsp\;.
	$$
	By independence of $A[i]$ and $A[j]$ when $i\neq j$, $\Exp{A[i] T[i,j] A[j]\trsp}=0_{d\times d}$ when $i\neq j$.  Hence,
	$$
	\Exp{A T A\trsp}=\sum_{1\leq i\leq n} \Exp{A[i] T[i,i] A[i]\trsp}\;.
	$$
	Note that if we can bound uniformly $\Exp{\opnorm{A[i] T[i,i] A[i]\trsp}}$ by $K(z)$, then we have 
	$$
	\frac{1}{nd}\opnorm{\Exp{A T A\trsp}}\leq \frac{K(z)}{d}\;.
	$$
	So let us focus on the $d\times d$ matrix  $Q[i]=\Exp{A[i] T[i,i] A[i]\trsp}$. Let us call $r_k$ the $k$-th row of $A[i]$. The $k,j$ entry of $Q[i]$ is just 
	$$
	Q[i](k,j)=\Exp{r_k T[i,i]r_j\trsp}=\trace{T[i,i]\Exp{r_j\trsp r_k}}=\trace{T[i,i]{\cal S}_i[k,j]}\;,
	$$
	where ${\cal S}_{i}[k,j]=\Exp{r_j\trsp r_k}$ is the $d\times d$ cross-covariance matrix between the $j$-th and the $k$-th rows of $A[i]$. 

	Suppose that $\opnorm{T[i,i]}\leq \frac{1}{v}$ and $\opnorm{{\cal S}_{i}[k,j]}\leq C$, where $C$ is a constant independent of $i,j,k$. Then, 
	$$
	\left|Q[i](j,k)\right|\leq \frac{Cd}{v}\;, \forall (j,k)\;.
	$$
	Therefore, 
	$$
	\opnorm{Q[i]}\leq \frac{Cd^2}{v}\;.
	$$
	We note that if $T=(\Gamma-z\id)^{-2}$, where $\Gamma$ is real symmetric, then $\opnorm{T}\leq 1/v^2$, if $v=\imag{z}$. 
The Lemma is shown.
	
\end{proof}
\subsection{Concentration of quadratic forms in block rows}
We now give sufficient conditions for \ref{eq:assumptionB11} and \ref{eq:assumptionB13} to be satisfied.

\begin{lemma}\label{lemma:controlVariance}
Let us call ${\cal Q}=A T A\trsp$, where $A$ is a $d\times (nd)$ real random matrix composed of independent $d\times d$ blocks, denoted by $A[i]$. We assume that $\Exp{A}=0$.
$T$ is a $nd\times nd$ symmetric matrix with complex entries with $\opnorm{T}\leq \frac{1}{v}$.

Denote by ${\cal S}_i[k,k]$ the covariance matrix of the $k$-th row of the matrix $A[i]$. Assume that there exists $C>0$ such that 
$$
\sup_{1\leq i \leq n}\sup_{1\leq k \leq d} \opnorm{{\cal S}_i[k,k]}\leq C\;.
$$
Assume further that, for some $\eps>0$, the rows of $A[i]$ have uniformly bounded $2+2\eps$-th moments, for all $1\leq i \leq n$. 
Then, when $d$ is fixed, we have 
$$
\opnorm{\Exp{\left|{\cal Q}-\Exp{{\cal Q}}\right|}}=\frac{\gO(n^{1/(1+\eps)}\wedge n^{1/2})}{v}\;.
$$
\end{lemma}

\begin{proof}
Let us call ${\cal Q}(j,k)$ the $(j,k)$ entry of ${\cal Q}$.
Since $d$ is held fixed, to show the result, it is enough to show that 
$$
\forall 1\leq j,k\leq d \;, \Exp{|{\cal Q}(j,k)-\Exp{{\cal Q}(j,k)}|}=\frac{\gO(n^{1/(1+\eps)}\wedge n^{1/2})}{v}\;.
$$
Let us call $r_j$ the $j$-th row of $A$. Clearly, 
$$
{\cal Q}(j,k)=r_j T r_k\trsp=\frac{1}{4}\left((r_j+r_k) T (r_j+r_k)\trsp-(r_j-r_k) T (r_j-r_k)\right)\;.
$$
Hence, to understand ${\cal Q}$, we simply need to understand forms of the type 
$$
f(\mathsf{r})= \mathsf{r}\trsp T \mathsf{r}
$$
where $\mathsf{r} \in \mathbb{R}^{nd}$ is a random vector composed of independent blocks of size $d$. Indeed, given the structure we have assumed for $A$, it is clear that both $r_j+r_k$ and $r_j-r_k$ are vectors composed of independent blocks of length $d$. (Our assumptions about ${\cal S}_i[k,k]$ implies that the same assumptions are true for all the $\mathsf{r}$'s we will be looking at, with a upper bound less than $2C$.) In other words, 
$$
\mathsf{r}=\begin{pmatrix}
\mathsf{r}_1\\\mathsf{r}_2\\ \vdots \\ \mathsf{r}_n
\end{pmatrix}\;,
$$
where $\mathsf{r}_i \in \mathbb{R}^d$ are independent of each other. We call $\Sigma[i,i]$ the covariance matrix of $\mathsf{r}_i$. 

We have of course, if $T[i,j]$ denotes the $(i,j)-$th $d\times d$ block of $T,$
$$
f(\mathsf{r})=\sum_{i,j} \mathsf{r}_i\trsp T[i,j] \mathsf{r}_j\triangleq \sum_i \mathsf{r}_i\trsp T[i,i] \mathsf{r}_i +{\mathcal R}\;.
$$

$\bullet$ \textbf{On $\bm{\var{{\mathcal R}}}$}
By definition, 
$$
{\mathcal R}=\sum_{i\neq j} \mathsf{r}_i\trsp T[i,j] \mathsf{r}_j\;.
$$
Since $\mathsf{r}_i$ and $\mathsf{r}_j$ are independent when $i\neq j$, we see that $\Exp{{\mathcal R}}=0$. So 
$$
\var{{\mathcal R}}=\Exp{{\mathcal R}{\mathcal R}^*}=\sum_{(i\neq j), (k\neq l)} \Exp{\mathsf{r}_i\trsp T[i,j] \mathsf{r}_j\mathsf{r}_k\trsp T^*[k,l] \mathsf{r}_l}\;.
$$
If one of the indices $(i,j,k,l)$ appears exactly once, $\Exp{\mathsf{r}_i\trsp T[i,j] \mathsf{r}_j\mathsf{r}_k\trsp T^*[k,l] \mathsf{r}_l}=0$, by independence of the $\mathsf{r}_j$'s and the fact that they all have mean 0. Now since each index appears at most once in each pair, we see that each index can appear at most twice among the four indices. It is therefore clear that 
$$
\var{{\mathcal R}}=\sum_{i\neq j}\Exp{\mathsf{r}_i\trsp T[i,j] \mathsf{r}_j\mathsf{r}_i\trsp T^*[i,j] \mathsf{r}_j}+\Exp{\mathsf{r}_i\trsp T[i,j] \mathsf{r}_j\mathsf{r}_j\trsp T^*[j,i] \mathsf{r}_i}\;.
$$
Of course, by independence, 
$$
\Exp{\mathsf{r}_i\trsp T[i,j] \mathsf{r}_j\mathsf{r}_j\trsp T^*[j,i] \mathsf{r}_i}=\Exp{\mathsf{r}_i\trsp T[i,j]\Sigma[j,j]T^*[j,i] \mathsf{r}_i}=\trace{T[i,j]\Sigma[j,j]T^*[j,i]\Sigma[i,i]}\;.
$$
Consider the matrix $D_\Sigma$ which is block-diagonal with $i$-th diagonal block $\Sigma[i,i]$. We note that 
$$
\trace{TD_\Sigma T^* D_\Sigma}=\sum_{i,j}\trace{T[i,j]\Sigma[j,j]T^*[j,i]\Sigma[i,i]}\;.
$$
Note further that $\trace{T[i,i]\Sigma[i,i]T^*[i,i]\Sigma[i,i]}\geq 0$. So we conclude that 
$$
\sum_{i\neq j}\Exp{\mathsf{r}_i\trsp T[i,j] \mathsf{r}_j\mathsf{r}_j\trsp T^*[j,i] \mathsf{r}_i}\leq \trace{TD_\Sigma T^* D_\Sigma}\leq 4N\frac{C^2}{v^2}\;.
$$
The same argument works for $\sum_{i\neq j}\Exp{\mathsf{r}_i\trsp T[i,j] \mathsf{r}_j\mathsf{r}_i\trsp T^*[i,j] \mathsf{r}_j}$ and we conclude that 
$$
\var{{\mathcal R}}\leq 8N \frac{C^2}{v^2}\;.
$$
This naturally implies that 
$$
\Exp{|{\mathcal R}-\Exp{{\mathcal R}}|}\leq \frac{2\sqrt{2}\sqrt{N}C}{v}\;.
$$
Note that this bound works under the assumption that $\mathsf{r}_i$'s have uniformly bounded covariances, i.e only 2 moments. 

$\bullet$ \textbf{On the convergence of $D_1(\mathsf{r})=\sum_i \mathsf{r}_i\trsp T[i,i]\mathsf{r}_i$}
Note that 
$$
D_1(\mathsf{r})-\Exp{D_1(\mathsf{r})}=\sum_i X_i\;,
$$
where $X_i$ are independent and mean 0 random variables in $L_{1+\eps}$. Using the Marcienkiewicz-Zygmund inequality (\cite{ChowTeicherBook97}, p. 386), we see that for any $\eps>0$ there exists $B_{1+\eps}$ such that 
$$
\Exp{\left|D_1(\mathsf{r})-\Exp{D_1(\mathsf{r})}\right|^{1+\eps}}\leq B_{1+\eps}\Exp{\left[\sum_i X_i^2\right]^{(1+\eps)/2}}\;.
$$
We have $(\sum_i X_i^2)^{p/2}=(\sum_i (|X_i|^p)^{2/p})^{p/2}=\norm{Y}_{2/p}$, where $Y_i=|X_i|^p$. For $p\in [1,2]$, we have $2/p\geq 1$, so 
$\norm{Y}_{2/p}\leq \norm{Y}_1$. Therefore, when $p\in[1,2]$, 
$$
\left(\sum_i X_i^2\right)^{p/2}\leq \sum_{i}|X_i|^p\;.
$$
Hence, 
$$
\Exp{|D_1(\mathsf{r})-\Exp{D_1(\mathsf{r})}|^{1+\eps}}\leq B_{1+\eps} \Exp{\sum_i |X_i|^{1+\eps}}\;.
$$
We conclude that when $\mathsf{r}_i$'s have $2+2\eps$ moments with $0<\eps\leq 1$, we have 
$$
\Exp{\left|D_1(\mathsf{r})-\Exp{D_1(\mathsf{r})}\right|}\leq \left[\Exp{\left|D_1(\mathsf{r})-\Exp{D_1(\mathsf{r})}\right|^{1+\eps}}\right]^{1/(1+\eps)}\leq C_\eps \frac{n^{1/(1+\eps)}}{v}\;.
$$
$\bullet$\textbf{Conclusion}

We can finally conclude that, for all $(j,k)$
$$
\Exp{\left|{\cal Q}(j,k)-\Exp{{\cal Q}(j,k)}\right|}=\frac{\gO(n^{1/(1+\eps)}\wedge n^{1/2})}{v}\;.
$$
The result announced in the Lemma follows immediately since $d$ is assumed to be fixed. 
\end{proof}

\begin{corollary}\label{coro:suffCondiB11General}
Suppose the symmetric matrix $M_n$ is such that its $i$-th ($d$-high) block row/strip, $M_n(i)$, is composed of independent $d\times d$ matrices.
Denote by ${\cal S}_{m}^{i}[k,k]$ the covariance matrix of the $k$-th row of $M_n[i,m]$. Assume that there exists $C>0$ such that 
$$
\sup_{1\leq i \leq n}\sup_{i\leq m \leq n} \sup_{1\leq k \leq d} \opnorm{{\cal S}_{m}^{i}[k,k]}\leq C\;.
$$
Assume further that the rows of all the $d\times d$ block matrices above the block diagonal of $M_n$ have uniformly bounded $(2+2\eps)$-th moments ($\eps>0$) and that $\Exp{M_n}=0$.
Then \ref{eq:assumptionB11} and \ref{eq:assumptionB13} hold with $R_i=\gO(n^{-\eps/(1+\eps)}\wedge n^{-1/2})$.
\end{corollary}
The proof is an immediate application of Lemma \ref{lemma:controlVariance}. Note that ``padding'' a block row with 0 block matrices does not change anything to our analysis: just consider the 0 block as a random variables with 0-covariance.

\begin{corollary}\label{coro:suffCondiB11GaussOrBounded}
Suppose the symmetric matrix $M_n$ is such that $M_n(i)$ is composed of independent $d\times d$ matrices. Suppose the entries of $M_n$ are either bounded or Gaussian. 
Denote by ${\cal S}_{m}^{i}[k,k]$ the covariance matrix of the $k$-th row of $M_n[i,m]$. Assume that there exists $C>0$ such that 
$$
\sup_{1\leq i \leq n}\sup_{i\leq m \leq n} \sup_{1\leq k \leq d} \opnorm{{\cal S}_m^{i}[k,k]}\leq C\;.
$$
Then \ref{eq:assumptionB11} and \ref{eq:assumptionB13} hold with $R_i=\gO(n^{-1/2})$.
\end{corollary}
When the entries of the matrix are Gaussian or bounded, the $2+2\eps$-th moment condition is automatically satisfied when our condition on covariance matrices is satisfied. 
\subsection{On $\Exp{\tildeM_n(i)uu\trsp \tildeM_n(i)\trsp}$}
The following fact will be helpful in establishing equivalence between models from a LSD point of view. 
\begin{fact}\label{fact:allThatMattersIsSymmPartOfCrossCovariances}
Let $\tildeM_n^{(1)}$ and $\tildeM_n^{(2)}$ be two random $d\times N$ strips with mean 0. Let us call  $C^{(1)}[j,k]$ the cross-covariance between the $j$-th and the $k$-th row of $\tildeM_n^{(1)}$ and $C^{(2)}[j,k]$ the cross-covariance between the $j-$th and the $k-$th row of $\tildeM_n^{(2)}$. 
Suppose that 
$$
\forall (j,k) \;\;, C^{(1)}[j,k]+(C^{(1)}[j,k])\trsp=C^{(2)}[j,k]+(C^{(2)}[j,k])\trsp\;.
$$
Then, if $u$ is any deterministic vector,
$$
\Exp{\tildeM_n^{(1)} uu\trsp \tildeM_n^{(1)'} }=\Exp{\tildeM_n^{(2)} uu\trsp \tildeM_n^{(2)'} }\;.
$$
\end{fact}
\begin{proof}
If $\mathsf{r}_j$ denotes the $j$-th row of $\tildeM_n$, we have, for the $(k,j)$-th entry of the matrix $\mathsf{EQ}=\Exp{\tildeM_nuu\trsp \tildeM_n\trsp}$,
$$
\mathsf{EQ}(k,j)=\Exp{\mathsf{r}_k uu\trsp \mathsf{r}_j\trsp}=\trace{uu\trsp \Exp{\mathsf{r}_j\trsp \mathsf{r}_k}}\;.
$$
Let us call $C[j,k]$ the cross-covariance matrix $C[j,k]=\Exp{\mathsf{r}_j\trsp \mathsf{r}_k}$. Since $\trace{AB}=\trace{BA}$ and $\trace{A}=\trace{A\trsp}$, we have 
$$
\mathsf{EQ}(k,j)=\trace{uu\trsp C[j,k]}=\trace{C[j,k]uu\trsp}=\trace{uu\trsp C\trsp[j,k]}=\trace{uu\trsp \frac{C[j,k]+C\trsp[j,k]}{2}}\;.
$$
The result is established.
\end{proof}

\paragraph{A remark on the case of anti-symmetric cross-covariances.}
We now assume that if $j\neq k$, the cross-covariance matrix $C[j,k]=\Exp{\mathsf{r}_j\trsp \mathsf{r}_k}$ is anti-symmetric (see the proof of Fact \ref{fact:allThatMattersIsSymmPartOfCrossCovariances} for the definition of $\mathsf{r}_j$'s). In this case, 
$$
C[j,k]+C\trsp[j,k]=0\;.
$$
This means in particular that if $\tildeM_n^{(1)}$ is such that its rows have anti-symmetric cross-covariance, we can create a ``good" $\tildeM_n^{(2)}$ by picking independent vectors matching the covariance of each row of $\tildeM_n^{(1)}$. This way $\tildeM_n^{(2)}$ clearly has anti-symmetric cross-covariance between its rows (indeed the cross-covariance is 0 for all pairs of distinct rows); but each row of $\tildeM_n^{(2)}$ has by construction the same covariance as the corresponding row of $\tildeM_n^{(1)}$.  So we have 
$$
\Exp{\tildeM_n^{(1)} uu\trsp \tildeM_n^{(1)'}}=\Exp{\tildeM_n^{(2)}uu\trsp \tildeM_n^{(2)'}}\;.
$$

\paragraph{The case of block matrices with mean 0.} We now assume the $d\times \totalSize$ matrix $\tildeM_n$ is made of $n$ $d\times d$ independent blocks.
In that case,  $\mathsf{r}_j$ and $\mathsf{r}_k$, its $j$-th and $k$-th rows, are composed of independent blocks, so $\Exp{\mathsf{r}_j\trsp \mathsf{r}_k}$ is block diagonal. The $l$-th $d\times d$ block on the diagonal is just the cross covariance between the $j$-th and $k$-th row of the $l$-th $d\times d$ block matrix in ${\tildeM_n}$, since $\Exp{\tildeM_n}=0$. So our assumptions about the cross-covariance of the rows of ${\tildeM_n}$ in Fact \ref{fact:allThatMattersIsSymmPartOfCrossCovariances} can be replaced by assumptions concerning the cross-covariance of the rows of the block matrices making up  ${\tildeM_n}$ and the same result holds.

\subsection{Applications and examples}
We now give some examples to show the applicability of our results. A source of motivation came from the examples discussed in Subsubsection \ref{subsubsec:classAveragingAlgo} below. A number of the examples we study here are idealized or simplified versions of those.

We start by defining a broad class of matrices for which we will show that our results apply and the LSD turn out to be the well-known semi-circle law.

\begin{nekDef}[$\sigma$-Simple Structure]
Let $B$ be a $d\times d$ random matrix. Call $\{r_i\}_{i=1}^d$ its rows. 
We say that the random matrix $B$ has \textit{$\sigma$-simple structure} if and only if
\begin{enumerate}
\item the entries of $B$ have $2+\eps$ moments for some $\eps>0$.
\item if $j\neq k$, $\Exp{r_j\trsp r_k}$ is anti-symmetric with $\opnorm{\Exp{r_j\trsp r_k}}\leq C$, $C>0$.
\item for all $j$, $\Exp{r_j\trsp r_j}=\sigma^2 \id_d$. 
\end{enumerate}
\end{nekDef}

The following theorem explains the spectral distributions we see in a number of our numerical investigations in Section 3. 

\begin{theorem}\label{thm:cvToWignerUnderSimpleStructure}
Suppose $M_n$ is a Hermitian $N\times N$ matrix with independent $d\times d$ block entries above the block diagonal. Let $DM_n$ be the block-diagonal of $M_n$. Suppose that $\opnorm{DM_n}/\sqrt{N}=\lo_P(1)$ and that $\rank{\Exp{M_n-DM_n}}=\lo(N)$.  
Suppose the off-diagonal blocks of $M_n$ have $\sigma$-simple structure with $2+\eps$ moments and $\sigma=1$. Suppose further that the cross-covariance between the rows of these off-diagonal blocks is uniformly bounded (i.e independently of $n$). 
Then the LSD of $M_n/\sqrt{\totalSize}$ is the Wigner semi-circle law. The convergence happens a.s.  
\end{theorem}

Since we are talking about sequences of random variables, it is important to specify how they are built. 
For our theorem to hold, we assume that the sequence of matrices $M_n$ is constructed by bordering the matrix $M_{n-1}$ with an independent block matrix satisfying our assumptions.   

\begin{proof}
A.s convergence of the spectral distribution is an immediate consequence of Theorem \ref{thm:concentrationOfStieltjesTransforms} and the Borel-Cantelli lemma. See \cite{nekCorrEllipD} for details. 

Define $M_n^{(0)} := M_n-DM_n $. The fact that $\opnorm{DM_n}/\sqrt{N}=\lo_P(1)$ guarantees that spectrally, $M_n/\sqrt{N}$ and $M_n^{(0)}/\sqrt{N}$ are asymptotically equivalent as we discussed earlier. 

Since $\rank{\Exp{M_n^{(0)}}}=\lo(N)$, we see by Lemma \ref{lemma:controlDiffStieltjesFiniteRankPerturb} that $M_n^{(0)}/\sqrt{N}$ and $(M_n^{(0)}-\Exp{M_n^{(0)}})/\sqrt{N}$ are asymptotically spectrally equivalent. Indeed, the modulus of the difference of their Stieltjes transform at $z$ is less than $\rank{\Exp{M_n^{(0)}}}/(\totalSize v)=\lo(1/v)$. 
So the theorem holds for $M_n$ provided we can prove it for $(M_n^{(0)}-\Exp{M_n^{(0)}})$; this latter matrix is still a matrix of independent blocks. However it has mean 0 and its block diagonal is zero. Our preliminary results have been obtained for matrices of this type. 
 Our assumptions guarantee that Assumption B1 is met for $(M_n^{(0)}-\Exp{M_n^{(0)}})$. Therefore we can apply Theorem \ref{thm:replacementByGaussian}. 

Since the rows of the block matrices composing $M_n$ have anti-symmetric cross-covariance, Fact \ref{fact:allThatMattersIsSymmPartOfCrossCovariances} and the discussion that follows it show that in the ``matching step'' of Theorem \ref{thm:replacementByGaussian}, we can use Gaussian matrices with independent rows. 

We note that a $d\times d$ Gaussian matrix with i.i.d ${\cal N}(0,1)$ entries has the same covariance for its rows as our initial model, under our assumptions, does. Assumption B1 is trivially met for a random block matrix with this Gaussian distribution on the blocks. Let us call the corresponding $N\times N$ matrix $GM_n$. 
Theorem \ref{thm:replacementByGaussian} guarantees that this Gaussian equivalent model has the same LSD as $(M_n^{(0)}-\Exp{M_n^{(0)}})$ and therefore the same is true for our initial sequence of matrices $M_n$.

Note that $GM_n/\sqrt{nd}$ is simply a scaled $nd \times nd$ GOE (Gaussian orthogonal ensemble) matrix with $d\times d$ block matrices on the diagonal removed. Call $BD_n$ the block diagonal matrix of a random matrix drawn according to $\totalSize \times \totalSize$ GOE. The norm of this block diagonal matrix $BD_n$ is simply the maximum of the norms of the $d\times d$ matrices on the block diagonal. For each such matrix, the operator norm is bounded by the largest row norm and hence by $d$ times the largest absolute value of the elements of the matrix. Hence, the operator norm of the block diagonal matrix is less than $d$ times the largest  entry (in absolute value) of all these matrices. There are $n d^2$ such elements (corresponding to $n d(d+1)/2$ independent elements), with variance at most $2/(nd)$. 
Hence, using well-known properties of the maximum of independent Gaussian random variables (see e.g \cite{deHaanFerreiraExtremeValueBook06} or \cite{TalagrandSpinGlassesBook03}, p.9), we have
$$
\frac{\opnorm{BD}}{\sqrt{nd}}\leq \sqrt{2/(nd)}\sqrt{2\log(nd^2)} \text{ a.s }\;.
$$
Clearly, the upper bound goes to 0. 
Therefore, $GM_n/\sqrt{nd}$ has the same LSD as $GOE_{nd}/\sqrt{nd}$, where $GOE_{nd}$ is a $N\times N$ random matrix sampled from GOE. Since the LSD of a GOE matrix is the Wigner semi-circle law, we have established the result. 
\end{proof}

\subsubsection{Case of $O(d)$ and $SO(d)$ sub-blocks}
We consider in this subsubsection the case of random matrices with independent random sub-blocks drawn at random uniformly (i.e according to Haar measure) from $O(d)$ and $SO(d)$. In this simple case, some of the results could be obtained by applying work of Girko \cite{GirkoMatrixEquationForBlockRandomMatrices95}. We present the results to illustrate the fact that our conditions are very easy to check.

\paragraph{$\bm{O(d)}$ case}
\begin{fact}
Matrices drawn according to Haar measure on $O(d)$ have  $\sigma$-simple structure with $\sigma=\frac{1}{\sqrt{d}}$.

Therefore, if $M_n$ is a Hermitian $\totalSize\times \totalSize$ block random matrix, with $\totalSize=nd$ and the blocks above the diagonal are drawn i.i.d according to Haar measure on $O(d)$, the LSD of $M_n/\sqrt{\totalSize}$ is the Wigner semi-circle law, scaled by $d^{-1/2}$ - provided the operator norm of the block diagonal of $M_n$ is $o_P(\totalSize^{1/2})$.
\end{fact}

\begin{proof}
We denote by ${\cal O}$ a random matrix drawn from $O(d)$ and by $r_k$ its $k$-th row. 
Since for all $k$, $\norm{r_k}=1$, the entries of ${\mathcal O}$ have infinitely many moments. 
When ${\cal O}$ is drawn according to Haar measure, we have by definition, for any given orthogonal $\mathsf{O}$, 
$$
{\cal O}\mathsf{O}\equalInLaw\mathsf{O}{\cal O}\equalInLaw {\cal O}\;.
$$
Taking $\mathsf{O}$ to be a permutation matrix, we see that the columns and rows of ${\cal O}$ are exchangeable. Taking $\mathsf{O}_j=\id_d-2 e_j e_j\trsp$ (where $e_j$ is the $j$-th canonical basis vector), we see that, 
$$
\forall k\neq j\;, (r_k,r_j)\equalInLaw (r_k,-r_j)\;.
$$
By Lemma \ref{lemma:invarianceAndMomentCsq}, this naturally implies that 
$$
\text{if }k\neq j\;,\Exp{r_j r_k\trsp}=0\;.
$$
Since $\norm{r_j}^2=1$ and the columns of ${\cal O}$ - and hence the entries of $r_j$ - are exchangeable, we see that, if $r_j(l)$ is the $l$-th entry of $r_j$,
$$
\forall 1\leq l \leq d\;, \; \; \;\Exp{r_j(l)^2}=\frac{1}{d}\Exp{\norm{r_j}^2}=\frac{1}{d}\;.
$$
Therefore, by Lemma \ref{lemma:invarianceAndMomentCsq}, 
$$
\forall j\;, \scov{r_j}=\frac{1}{d}\id_d\;.
$$
\end{proof}

\paragraph{$\bm{SO(d)}$ case}

\begin{fact}
Matrices drawn according to Haar measure on $SO(d)$ have $\sigma$-simple structure with $\sigma=\frac{1}{\sqrt{d}}$.

Therefore, if $M_n$ is a Hermitian $\totalSize\times \totalSize$ block random matrix, with $\totalSize=nd$ and the blocks above the diagonal are drawn i.i.d according to Haar measure on $SO(d)$, the LSD of $M_n/\sqrt{\totalSize}$ is the Wigner semi-circle law, scaled by $d^{-1/2}$ - provided the operator norm of the block diagonal of $M_n$ is $o_P(\totalSize^{1/2})$.
\end{fact}

\begin{proof}
	When ${\cal O}$ is drawn according to Haar measure on $SO(d)$, we have by definition, for any given $\mathsf{O} \in SO(d)$, 
	$$
	{\cal O}\mathsf{O}\equalInLaw\mathsf{O}{\cal O}\equalInLaw {\cal O}\;.
	$$	
	
$\bullet$ \textbf{Case  }$\mathbf{d\geq 3}$ Take $\mathsf{O}_{j,k,l}$ to encode the permutation $(j,k,l)\rightarrow (l,j,k)$. Clearly $\mathsf{O}_{j,k,l}$ is in $SO(d)$. This shows that the columns and the rows of ${\cal O}$ are exchangeable. 
Let $D_{j,k}$ be a diagonal matrix such that 
$$
D_{j,k}(i,i)=\begin{cases}
-1 \text{ if } i=j \text{ or } k\\
1 \text{ otherwise. }
\end{cases}
$$
Clearly $D_{j,k} \in SO(d)$. 
Since $D_{j,k}{\cal O}\equalInLaw {\cal O}$, we have 
$$
\text{ if } l\neq j \text{ and } l\neq k, (r_j,r_l)\equalInLaw (-r_j,r_l)\;.
$$
This shows that 
$$
\Exp{r_j r_l\trsp}=0 \text{ if } l\neq j\;.
$$
The fact that $\scov{r_j}=\id_d/d$ is proven as in the $O(d)$ case.

$\bullet$ \textbf{Case  }$\mathbf{d=2}$

It is clear geometrically that a matrix from $SO(2)$, which is simply a planar rotation, can be written 
$$
\mathcal{O}_{\theta}=\begin{pmatrix}
\cos(\theta) & -\sin(\theta)\\
\sin(\theta) & \cos(\theta)
\end{pmatrix}\;.
$$
When drawn according to Haar measure, $\theta$ is uniform on $[0,2\pi]$. (Geometrically, $\theta$ simply represents the angle by which the first canonical basis vector is rotated.) Hence,
$$
r_1\trsp r_2 =
\begin{pmatrix}
\cos(\theta) \sin(\theta) & \cos^2(\theta)\\
-\sin^2(\theta)&  -\cos(\theta) \sin(\theta)\;.
\end{pmatrix}
$$
So when ${\mathcal O}$ is drawn according to Haar measure, $\Exp{r_1\trsp r_2}$ is anti-symmetric.  
The fact that $\scov{r_i}=\frac{1}{2}\id_2$, $i=1,2$, is proven similarly.

\end{proof}

\subsubsection{Measures on $Gl(d,\RR)$ and $Sl(d,\RR)$}
We call ${\cal D}_m$ the set of diagonal matrices with $D_{ii}=1$ except for exactly $m$ indices for which $D_{ii}=-1$. 
\begin{fact}\label{fact:propertiesCrossCovarianceMatricesUnderGenericAssumptions}
Suppose that the $d\times d$ random matrix $B$ is such that it has the singular value decomposition 
$$
B=UDV\trsp\;,
$$
where $U$, $D$ and $V$ are independent. Suppose further that $U$ and $V$ (which are of course orthonormal) have laws that are invariant under the action of any permutations and any diagonal matrix in ${\cal D}_m$, $1\leq m \leq 2$. 
Then, if the entries of $D$ have $2+\eps$ moments, $B$ has $\sigma$-simple structure with 
$$
\sigma^2=\Exp{\trace{B\trsp B}}/d^2=\Exp{\trace{D^2}}/d^2\;,
$$
if $m=1$. If $m=2$, the same statement is true provided $d\geq 3$.
\end{fact}

\begin{proof}
	Our assumptions on the entries of $D$ guarantee that the entries of $B$ have $2+\eps$ moments. 
Let $P$ be a permutation. Since $PU\equalInLaw U$, it is clear that 
$$
PB\equalInLaw B\;.
$$
Hence the rows of $B$ are exchangeable. By a similar argument applied to $BP$, we see that the columns of $B$ are exchangeable. 

Suppose $m=1$. Let $\mathsf{D}_j$ be in ${\cal D}_1$ with $\mathsf{D}_j(j,j)=-1$. Since $D_jB\equalInLaw B$, we see that for any $k\neq j$, $(r_j,r_k)\equalInLaw (-r_j,r_k)$. 
When $m=2$ and $d=3$, we arrive at the same conclusion by using a matrix $\mathsf{D}$ in ${\cal D}_2$ such that $\mathsf{D}(j,j)=-1$, $\mathsf{D}(k,k)=1$ and $\mathsf{D}(l,l)=-1$ for $l\neq k$ nor $j$. Such a matrix exists by assumption. 
This implies that under our assumptions (see Lemma \ref{lemma:invarianceAndMomentCsq} for details)
$$
\Exp{r_j r_k\trsp}=0\;, \text{ when } j\neq k\;, \text{ and }\Exp{r_j}=0\;, \text{ for all } j\;.
$$

Now by exchangeability of the columns of $B$, we see that the diagonal of $\scov{r_j}$ is proportional to $\id_p$, with proportionality constant $\sigma^2=\Exp{\norm{r_j}^2}/d=\Exp{\trace{B B\trsp}}/d^2$, the latter equality coming from exchangeability of the rows of $B$. 

Suppose $m=1$.  Let $\mathsf{D}_j$ be in ${\cal D}_1$ with $\mathsf{D}_j(j,j)=-1$. Since $B \mathsf{D}_j \equalInLaw B$, we see that for any $k\neq j$, if $c_j$ denotes a generic column of $B$, $(c_j,c_k)\equalInLaw (-c_j,c_k)$. This implies that the off-diagonal elements of $\scov{r_k}$ are equal to 0. The case of $m=2$ is treated as above and we have shown the lemma. 
\end{proof}

We have the following corollaries. The proof of this corollary is immediate - owing to elementary facts about Wishart matrices for instance (see \cite{EatonBookReprint83} or \cite{anderson03}). 

\begin{corollary}
Suppose that $B$ is $d\times d$ with i.i.d {\cal N}(0,1) entries. Then $B\in Gl(d,\RR)$ with probability 1. Furthermore, it satisfies the assumptions of Fact \ref{fact:propertiesCrossCovarianceMatricesUnderGenericAssumptions} with $m=1$.
\end{corollary}

Next we discuss the case of $Sl(d,\RR)$.
\begin{corollary}\label{coro:SLdCase}
Let $G$ be a $d\times d$ matrix with i.i.d ${\cal N}(0,1)$ entries. Suppose 
$$
B=\frac{\widetilde{G}}{|\det(G)|^{1/d}}\;,
$$
where $\widetilde{G}=G/\sgn(\det(G))$ if $d$ is odd and $\widetilde{G}=G$ except that one column of $G$ - picked uniformly at random - is replaced by its opposite when $d$ is even. 
Then $B \in Sl(d,\RR)$ for any $d\geq 1$ almost surely. Furthermore, for $d\geq 3$, $B$ satisfies the assumptions of Fact  \ref{fact:propertiesCrossCovarianceMatricesUnderGenericAssumptions}.

Therefore, if $M_n$ is a Hermitian $\totalSize\times \totalSize$ block random matrix, with $\totalSize=nd$ and the blocks above the diagonal are drawn i.i.d with the same law as $B$ and $d\geq 3$, the LSD of $M_n/\sqrt{\totalSize}$ is a scaled Wigner semi-circle law - provided the operator norm of the block diagonal of $M_n$ is $o_P(\totalSize^{1/2})$. 
\end{corollary}

To get a finer understanding of this problem - especially for $d=2$ - we compute the law of the squared singular values of $B$ in Section \ref{subsec:CompsSLdAppendix}.
It turns out that in the case of $Sl(2,\RR)$, the largest eigenvalue of $B\trsp B$ has Cauchy-like tail. Therefore, the matrix $D$ of Fact \ref{fact:propertiesCrossCovarianceMatricesUnderGenericAssumptions} does not have 2 moments. The situation is therefore structurally different from the other problems we have investigated in this paper. In particular, in Figure \ref{fig:Sld2-1} and Figure \ref{fig:SldQQplot} in Section \ref{Section:NumericalSldR}, the spectral behavior of $Sl(2,\RR)$ is dramatically different from the others. We do not undertake here a specific characterization of the limiting distribution in this case as this issue is quite tangential to the main aims of the paper.

\begin{proof}
Let us write the singular value decompositions of $G$ and $B$ as $G=U(G)D(G)V(G)\trsp$ and $B=U(B)D(B)V(G)\trsp$.
Since $d\geq 3$, we have seen that $G$ satisfies the assumptions on $U(G)$ and $V(G)$ we made in Fact \ref{fact:propertiesCrossCovarianceMatricesUnderGenericAssumptions}. However, the $U(B)$ and $V(B)$ - though very closely related to $U(G)$ and $V(G)$ -  are not independent anymore, since $B\in Sl(d,\RR)$ implies that $\det(U(B)V(B))=1$. Because $d\geq 2$, our arguments involving matrices in ${\cal D}_2$ are still valid (matrices in ${\cal D}_2$ have determinant 1). Our argument involving permutation now require permutation matrices containing a cycle - as we did in the case of $SO(d)$. So our exchangeability arguments actually apply here 
and the only question we have to grapple with is that of the number of moments of the entries of $B$. 

We recall that by using Bartlett's decomposition \cite{muirhead82}, we see that, for independent $\chi_i^2$ random variables, 
$$
(\det(G))^2=\prod_{i=1}^{d} \chi_i^2\;.
$$
Recall that the density $f_1$ of $\chi_1^2$ is such that $f_1(x)\sim x^{-1/2}$ at 0 and $f_p$ the density of $\chi^2_p$ is such that $f_p(x)\sim x^{p/2-1}$ at 0. So we see that 
$$
\Exp{\frac{1}{|\det(G)|^{p/d}}}<\infty
$$
provided $\Exp{(\chi_1^2)^{-p/(2d)}}<\infty$ i.e $1/2+p/2d<1$ or $d>p$. 
By Holder's inequality, if $p\geq 1$ and $q=p/(p-1)$,
$$
\Exp{|B_{i,j}|^{k}}\leq \Exp{|G_{i,j}|^{kq}}^{1/q}\Exp{|\det(G)|^{-kp/d}}^{1/p} \;.
$$
So the entries of $B$ have $k$ moments provided $d>kp$ for some $p>1$. In other words, if $k<d$, the entries of $B$ have $k$ moments. 
We conclude that when $d\geq 3$, the assumptions of Fact  \ref{fact:propertiesCrossCovarianceMatricesUnderGenericAssumptions} are satisfied. 

\end{proof}

\subsubsection{Class averaging algorithm}\label{subsubsec:classAveragingAlgo}
We discuss in details in Appendix \ref{app:subsec:vdmComps} various properties of the elements of block matrices arising in the class averaging algorithm in the null case considered in this paper. For the non-null case, we refer the reader to the subsequent paper \cite{NEKHautieng2014CGLRobust}. 
One quantity of interest in this algorithm is 
$$
g_{ij}=\argmin_{g \in SO(2)}\|Z_i-g\circ Z_j\|_2^2\;,
$$
where $Z_i$'s is the data, viewed as a real-valued function on $\mathbb{R}^2$. 
Here is a summary of our results in the null case, where $Z_i$'s are pure noise:
\begin{lemma}
\begin{enumerate}
\item Suppose that $Z_i$ and $Z_j$ are independent. Suppose that each random variable has a distribution that is invariant under the action of $SO(2)$. Then $g_{ij}$ and $Z_i$ are independent and so are $g_{ij}$ and $Z_j$. Furthermore, $g_{ij}$ is uniformly distributed on $SO(2)$. 
\item Suppose that $Z_i$, $Z_j$ and $Z_k$ are independent, each random variable having a distribution that is invariant under the action of $SO(2)$. Then $g_{ij}$ and $g_{ik}$ are independent, and so are $g_{ij}$ and $g_{jk}$. Furthermore, the random variables $\{g_{ij}\}_{j=1}^n$ are jointly independent.
\end{enumerate}	
\end{lemma}

These results are shown in Appendix \ref{app:subsec:vdmComps}, specifically in the proofs of Lemmas \ref{lemma:distgijNullCase} and \ref{lemma:independencegijandgik}.

\section{Numerical Experiments}\label{sec:numericalWork}
We now present some numerical work to investigate the agreement between our theoretical results and simulations in ``reasonable'' dimensions. 
We show three simulations for the block random matrix and two simulations related to the class averaging algorithm. 

\subsection{Random block matrix with independent blocks}
We start from showing how the GCL behaves in the setup with independent blocks.
\subsubsection{Random orthogonal group $SO(d)$ and $O(d)$ with Haar measure}
Consider a $n\times n$ symmetric block matrix $R_{SO(d),n,\text{Haar}}$ with $d\times d$ entries so that its $(i,j)$-th entry, $i<j$, for all $i,j=1,\ldots,n$,
is uniformly sampled according to the Haar measure on $SO(d)$. We mention that the QR decomposition, at least as implemented in Matlab, leads to random matrices which are not distributed according to Haar measure, and needs to be corrected in order to numerically obtain the uniform samples on $SO(d)$ \cite{Mezzadri:2006}. 

The histogram of $R_{SO(d),n,\text{Haar}}$'s spectrum is shown in Figure \ref{fig:SOd}, when $n=1000$ and $d=2,3$. We also show the QQplot of the eigenvalues of $R_{SO(d),n,\text{Haar}}$ versus the eigenvalues of symmetric Gaussian random matrix of size $dn\times dn$, a good approximation to the Wigner semi-circle law. It is clear that the spectral distribution of $R_{SO(d),n,\text{Haar}}$ is a scaled semi-circle law, as predicted by our theory.

\begin{figure}[h]	
\begin{center}
\includegraphics[width=0.23\textwidth]{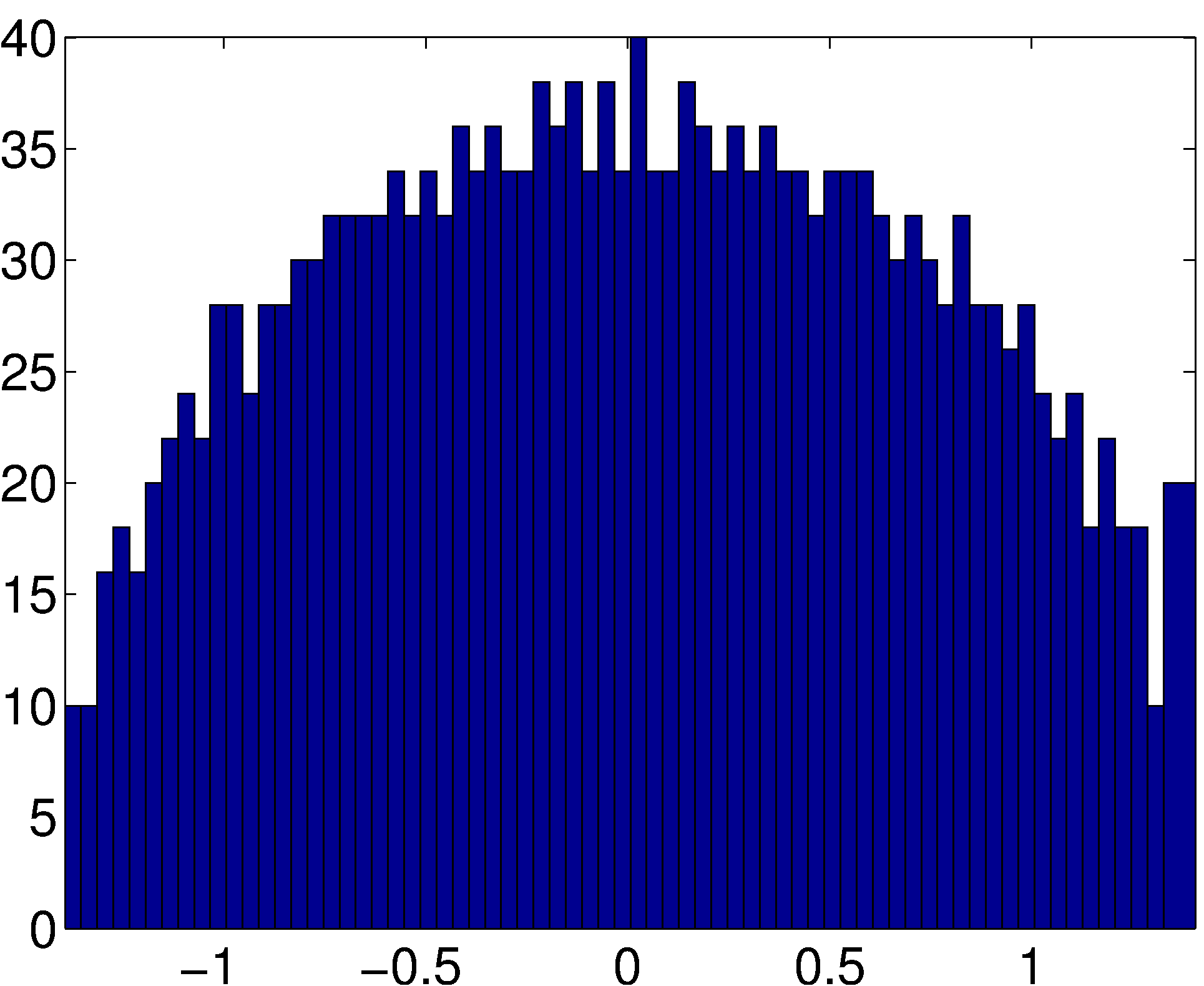}
\includegraphics[width=0.23\textwidth]{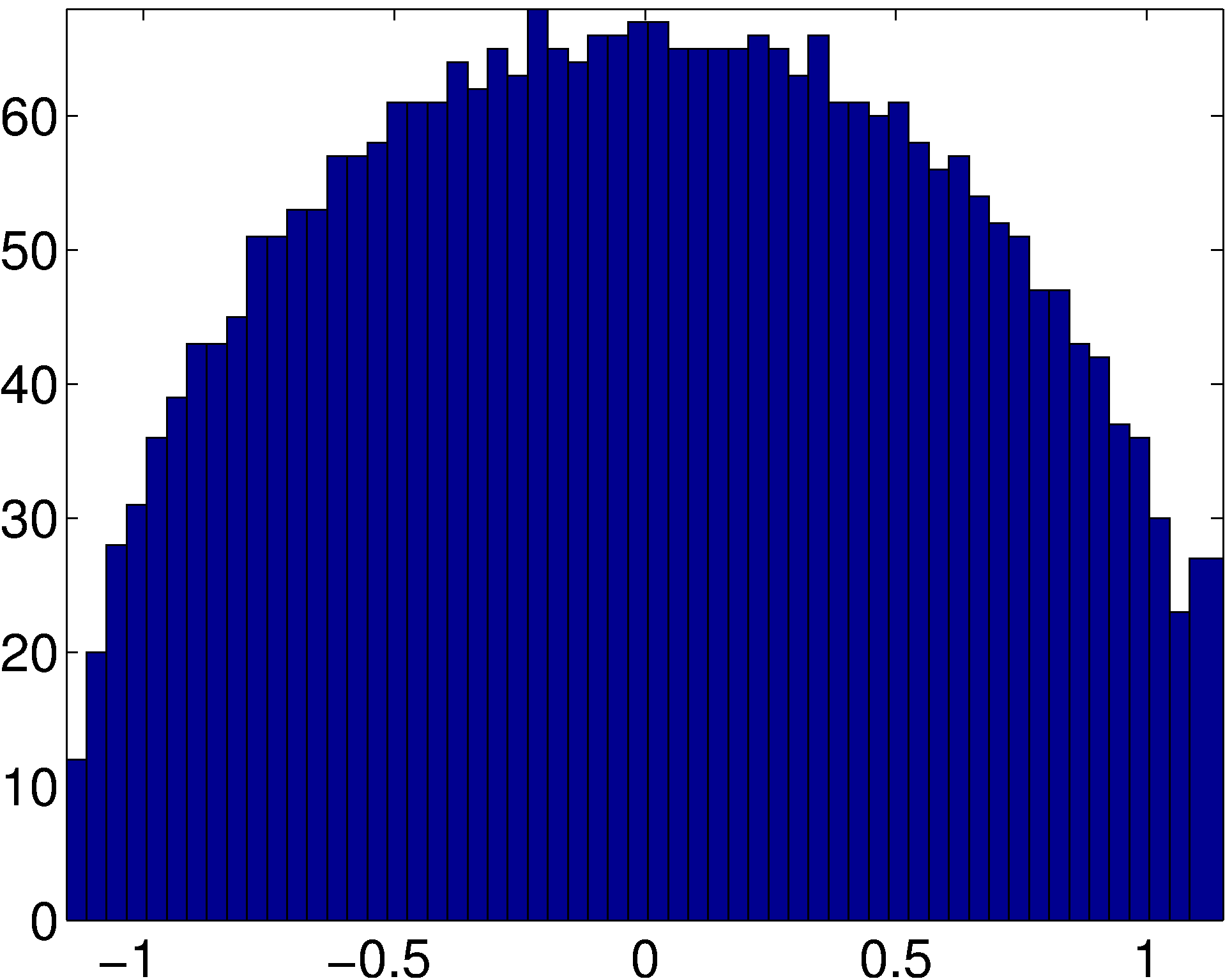}
\includegraphics[width=0.23\textwidth]{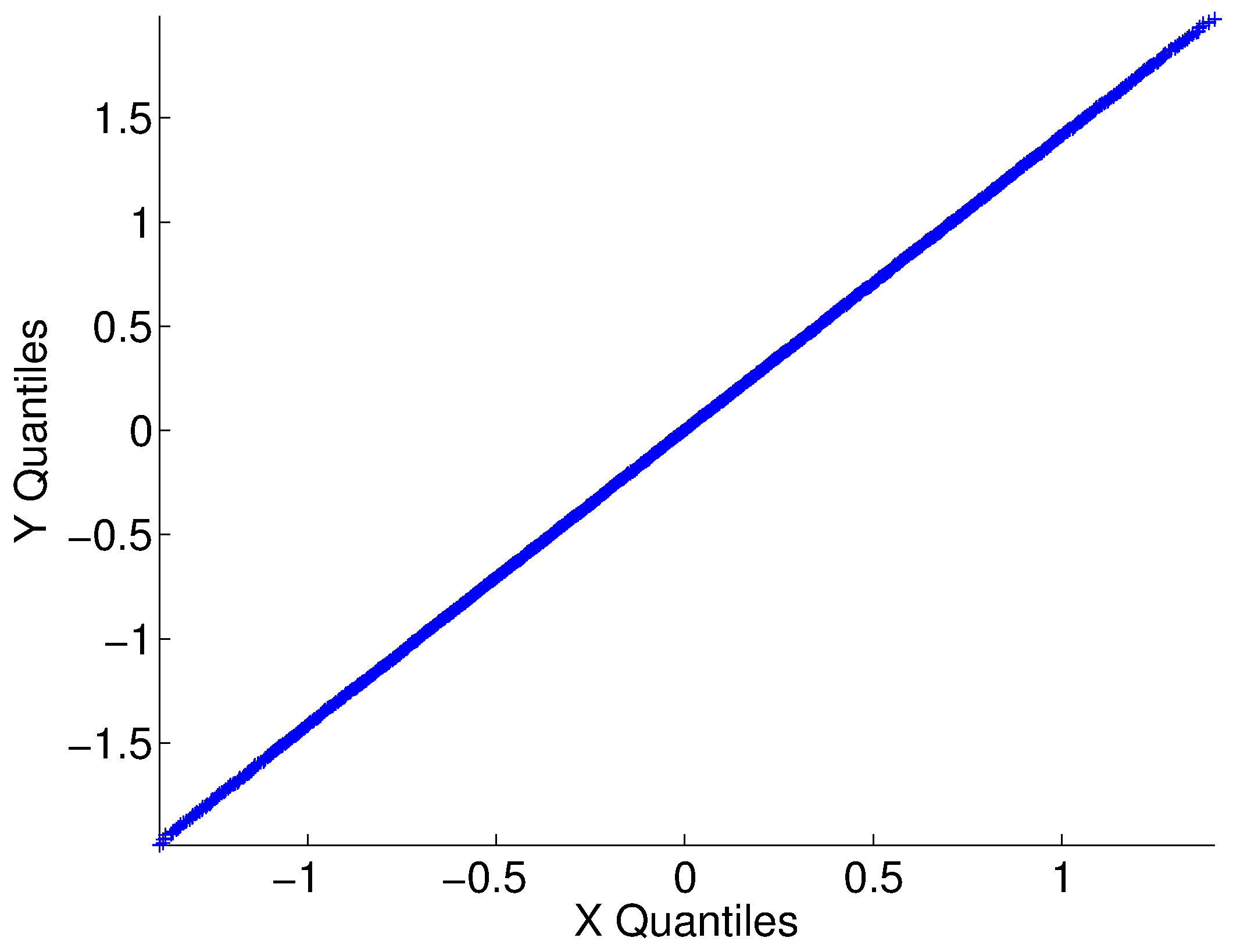} 
\includegraphics[width=0.23\textwidth]{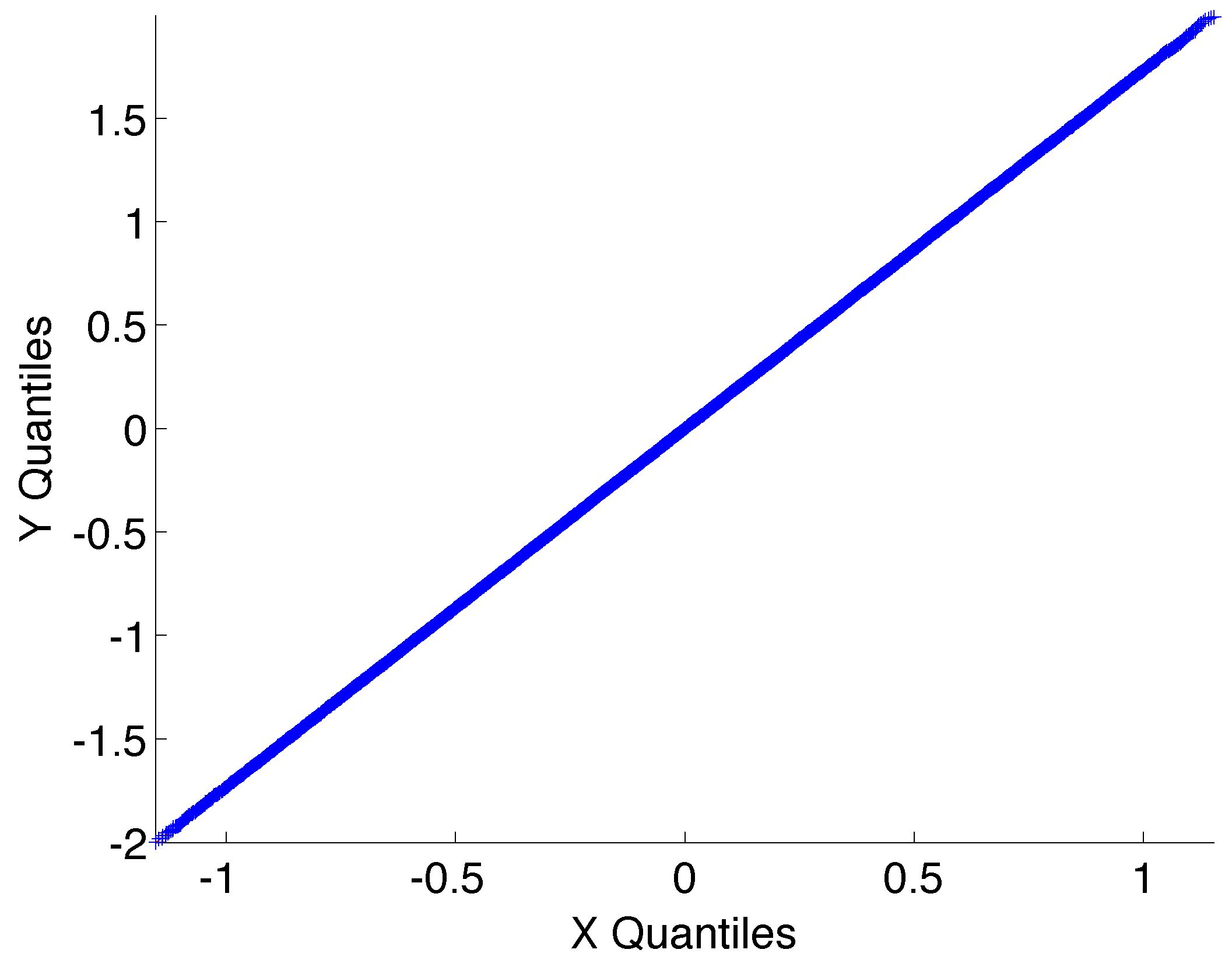}
\end{center}
\caption{\emph{Histogram of the eigenvalues of $R_{SO(d),n,\text{Haar}}$ with entries sampled from $SO(d)$ when $(n,d)=(1000,2)$ (left) and $(1000,3)$ (left middle) and QQplot of the eigenvalues of $R_{SO(d),n,\text{Haar}}$ versus the eigenvalues of symmetric Gaussian random matrix of size $dn\times dn$ when $(n,d)=(1000,2)$ (right middle) and $(1000,3)$ (right).}}
\label{fig:SOd}
\end{figure}

Next, consider a $n\times n$ symmetric block matrix $R_{O(d),n,\text{Haar}}$ with $d\times d$ entries so that its $(i,j)$-th entry, $i<j$, for all $i,j=1,\ldots,n$, is uniformly sampled according to the Haar measure on $O(d)$ \cite{Mezzadri:2006}. The histogram of $R_{O(d),n,\text{Haar}}$'s spectrum is shown in Figure \ref{fig:Od}, when $n=1000$ and $d=2,3$. We also show the QQplot of the eigenvalues of $R_{O(d),n,\text{Haar}}$ versus the eigenvalues of symmetric Gaussian random matrix of size $dn\times dn$. Again, it is clear that we obtain the semi-circle law.

\begin{figure}[h]	
\begin{center}
\includegraphics[width=0.23\textwidth]{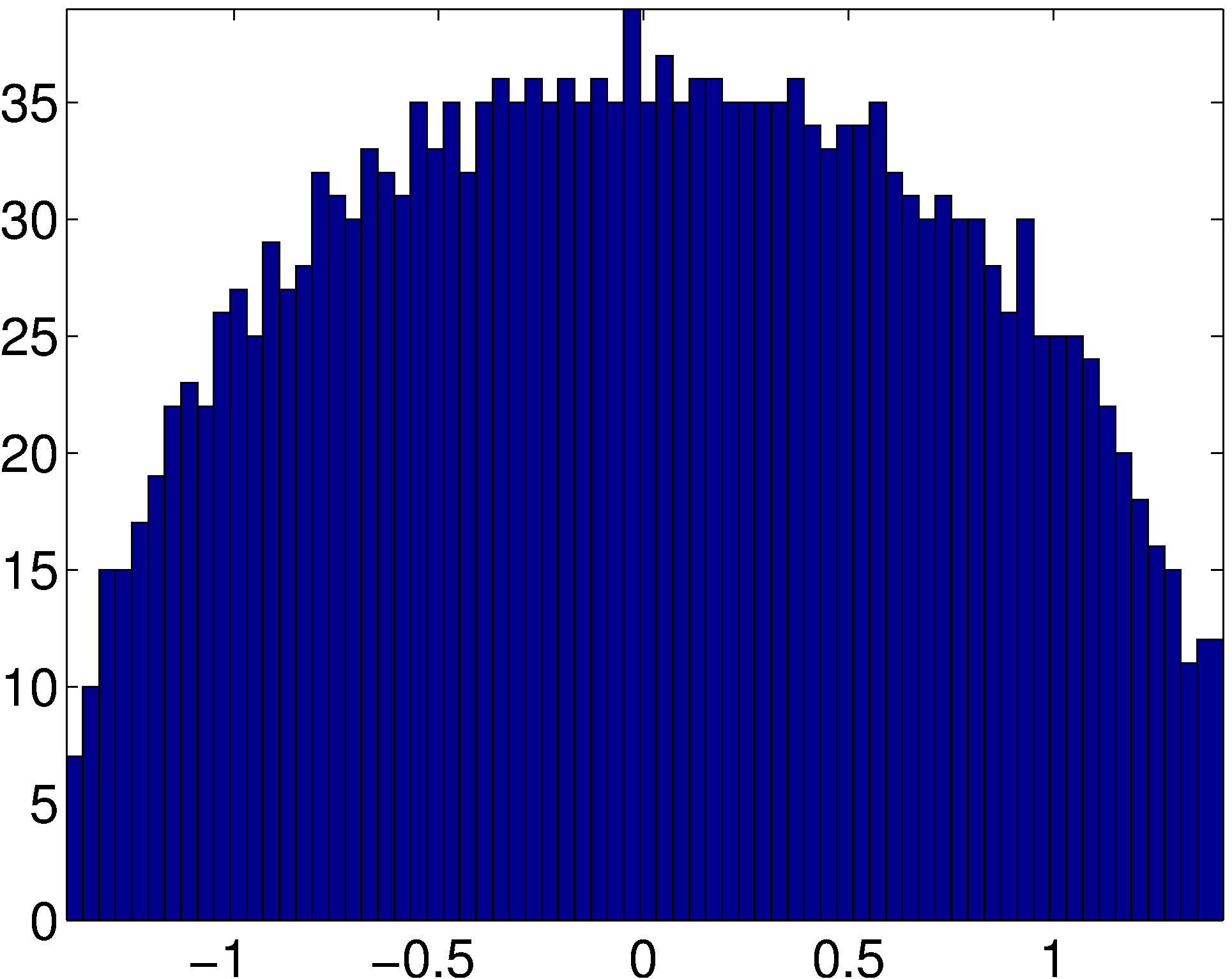}
\includegraphics[width=0.23\textwidth]{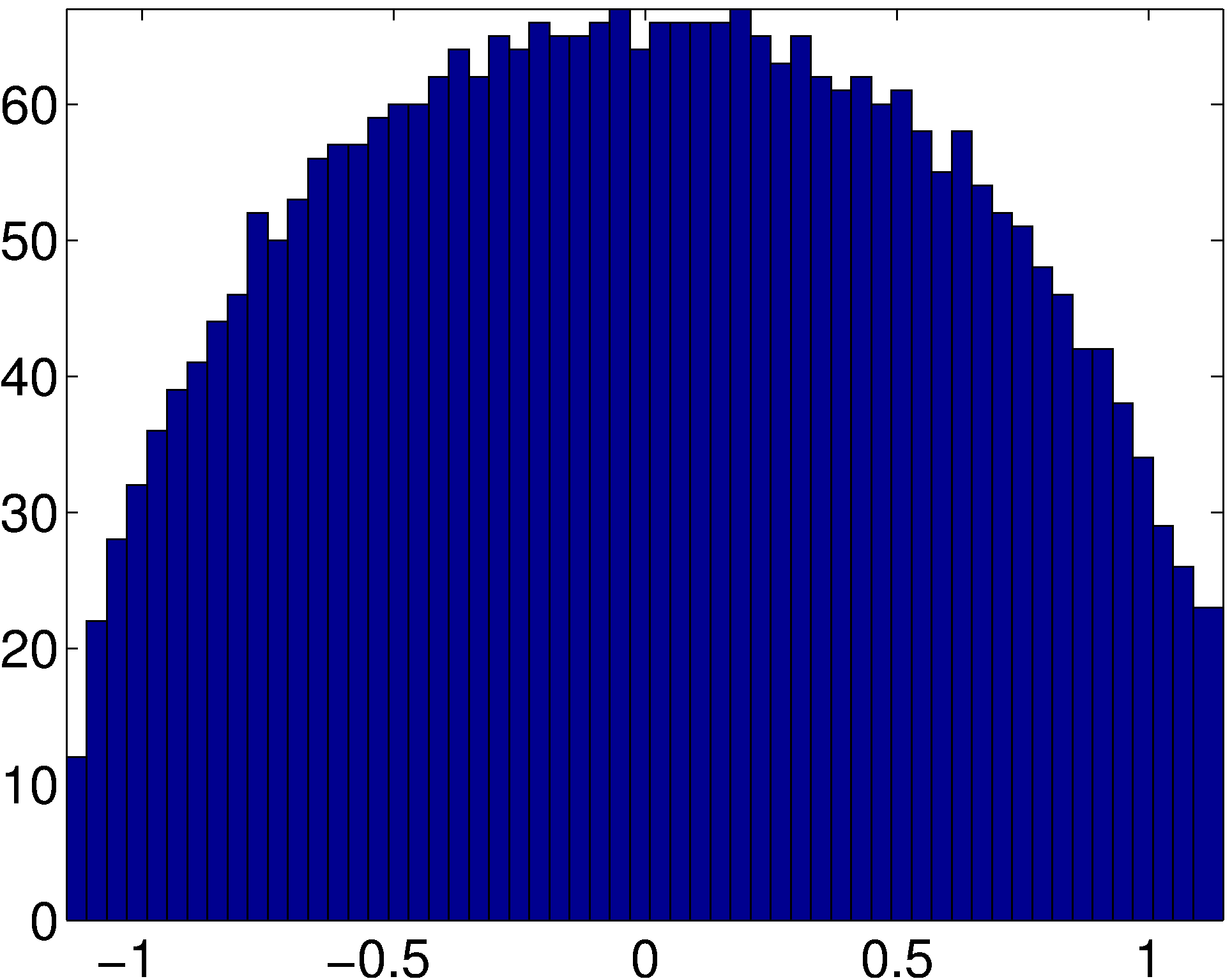}
\includegraphics[width=0.23\textwidth]{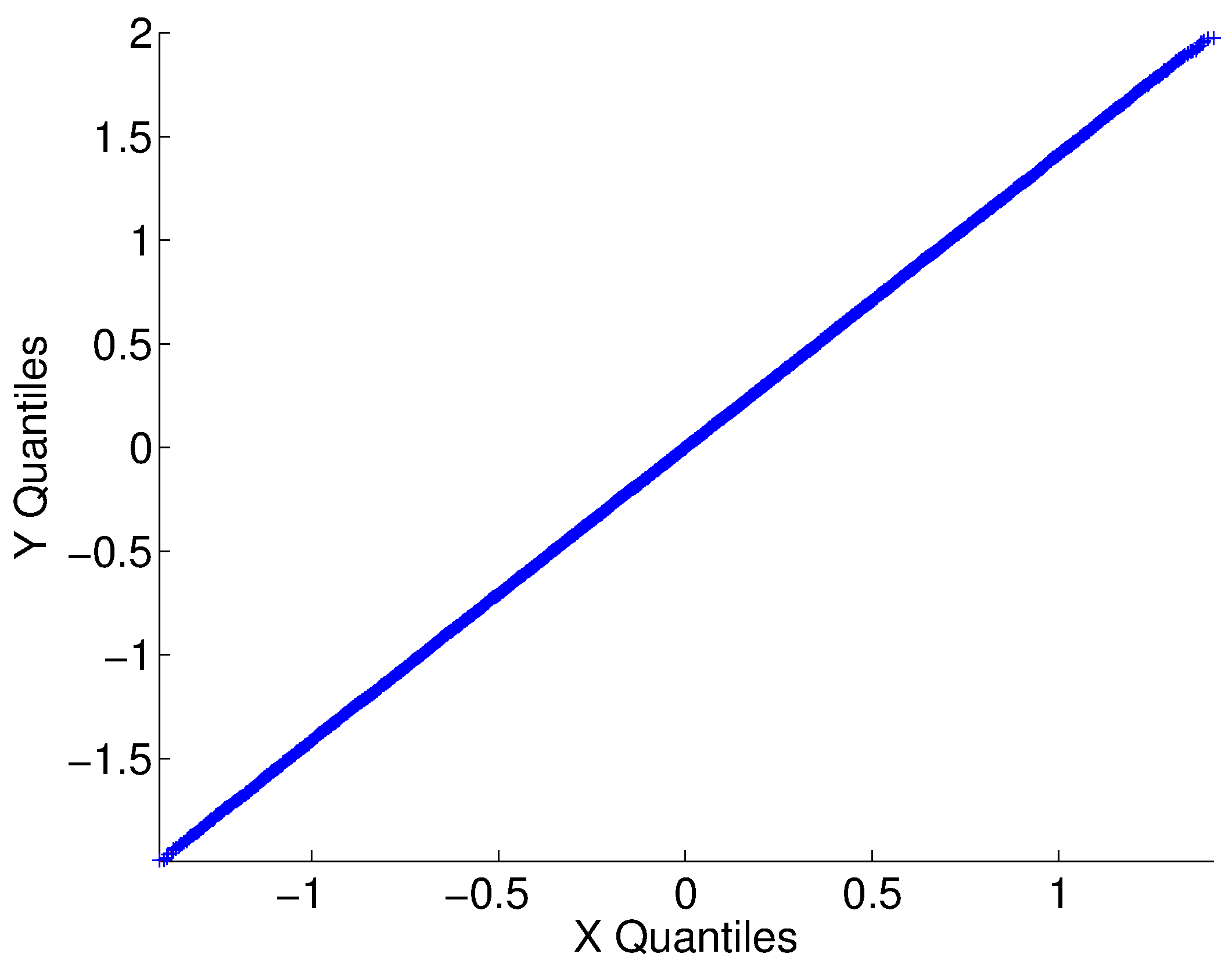} 
\includegraphics[width=0.23\textwidth]{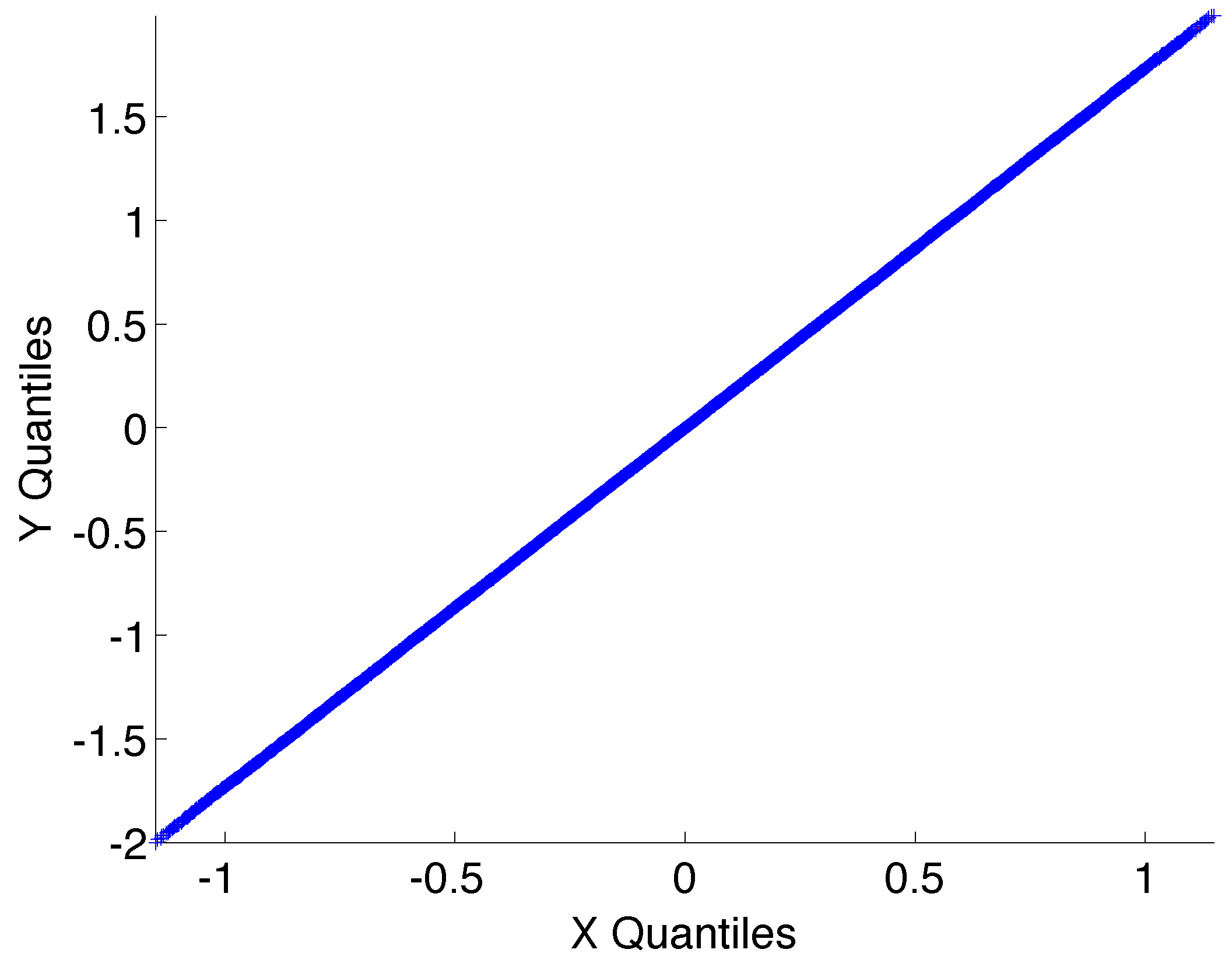}
\end{center}
\caption{\emph{Histogram of the eigenvalues of $R_{O(d),n,\text{Haar}}$ with entries sampled from $O(d)$ when $(n,d)=(1000,2)$ (left) and $(1000,3)$ (left middle) and QQplot of the eigenvalues of $R_{O(d),n,\text{Haar}}$ versus the eigenvalues of symmetric Gaussian random matrix of size $dn\times dn$ when $(n,d)=(1000,2)$ (right middle) and $(1000,3)$ (right).}}
\label{fig:Od}
\end{figure}

\subsubsection{Random orthogonal group $O(d)$ with non-Haar measure}
Consider a $n\times n$ symmetric block matrix $R_{O(d),n,\text{nonHaar}}$ with $d\times d$ entries so that its $(i,j)$-th entry is the orthogonal matrix in the QR decomposition of a random $d\times d$ matrix. It has been studied in \cite{Mezzadri:2006} that this sampling scheme on $O(d)$ is non-uniform.

\begin{figure}[h]	
\begin{center}
\includegraphics[width=0.23\textwidth]{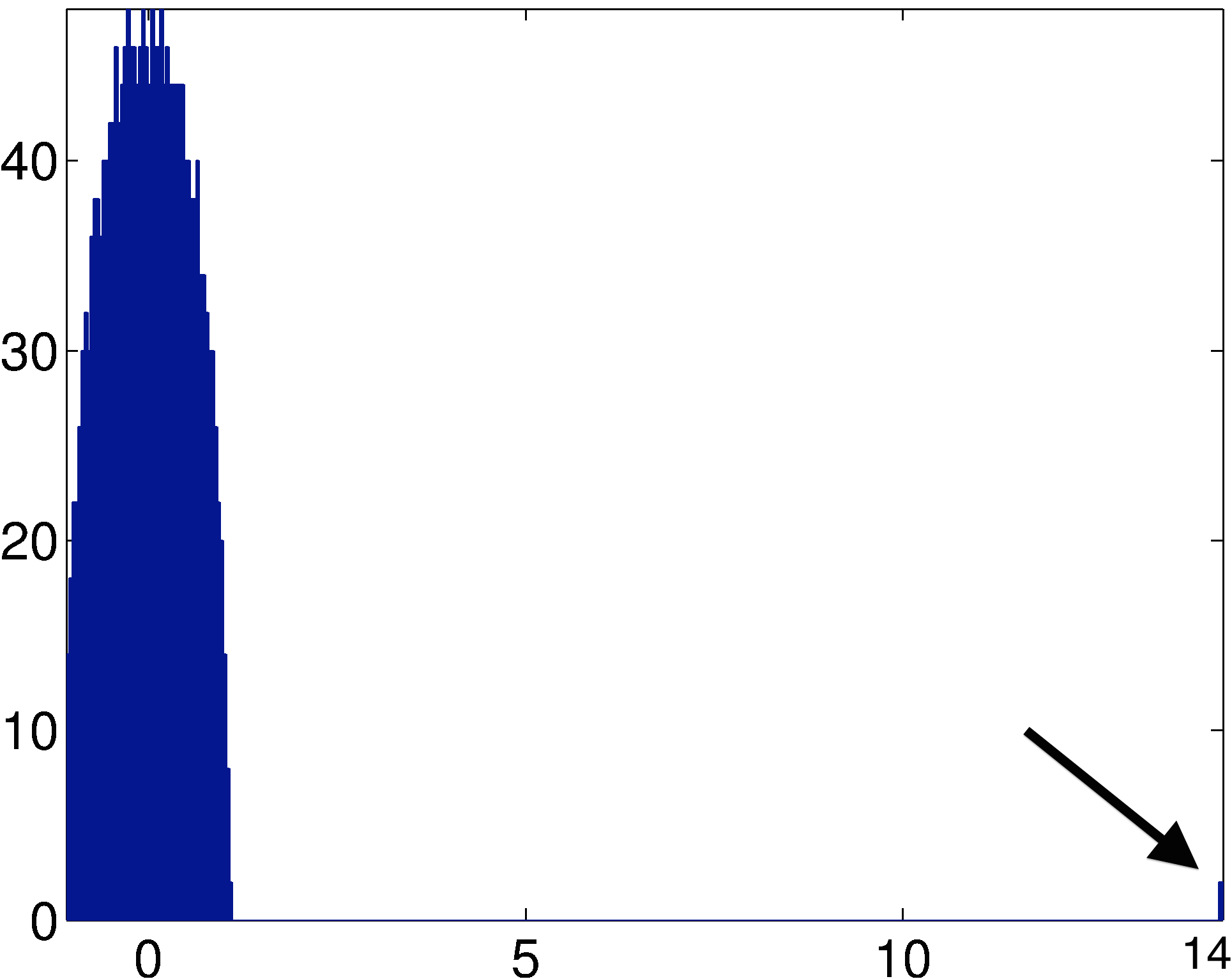}
\includegraphics[width=0.23\textwidth]{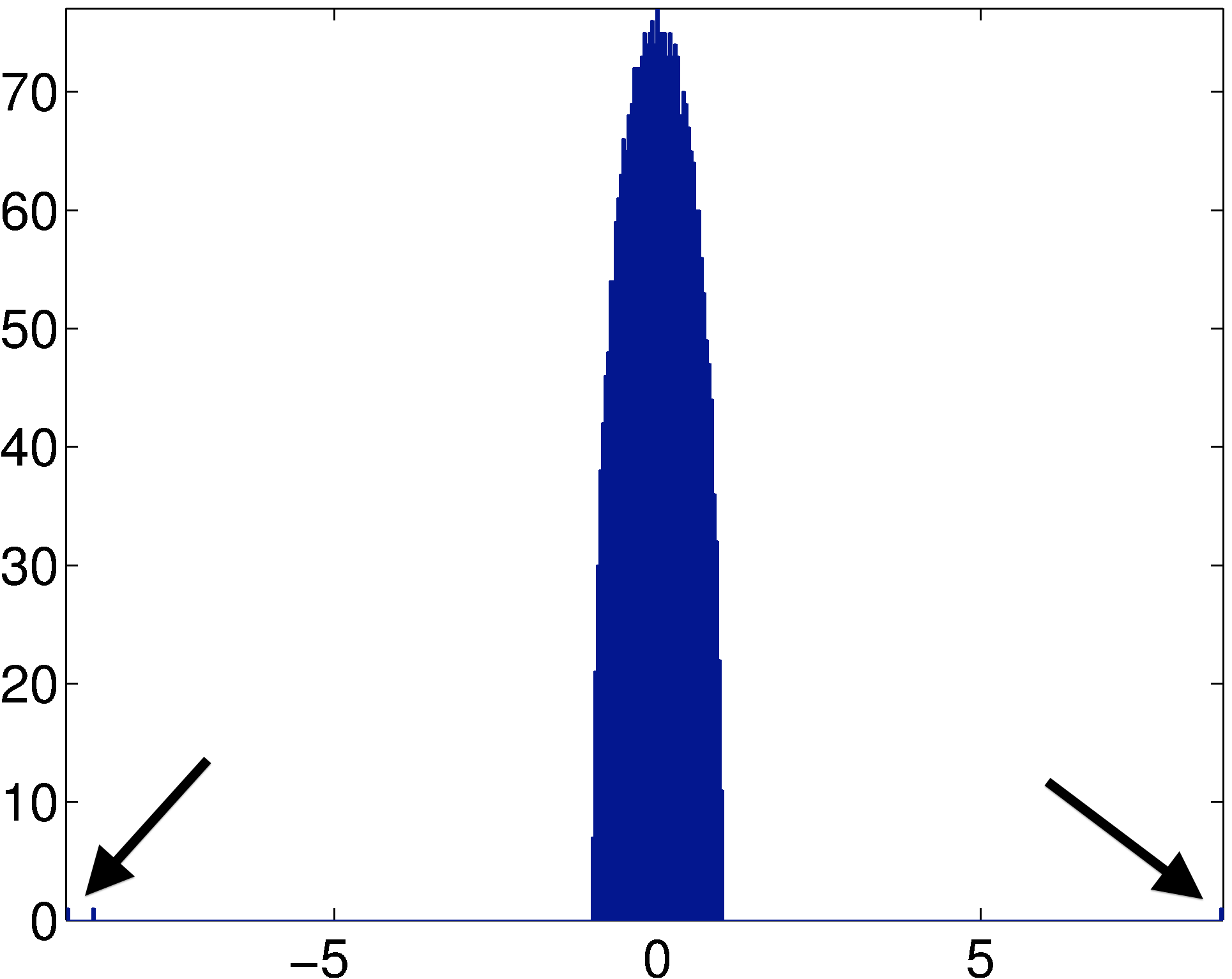}
\includegraphics[width=0.23\textwidth]{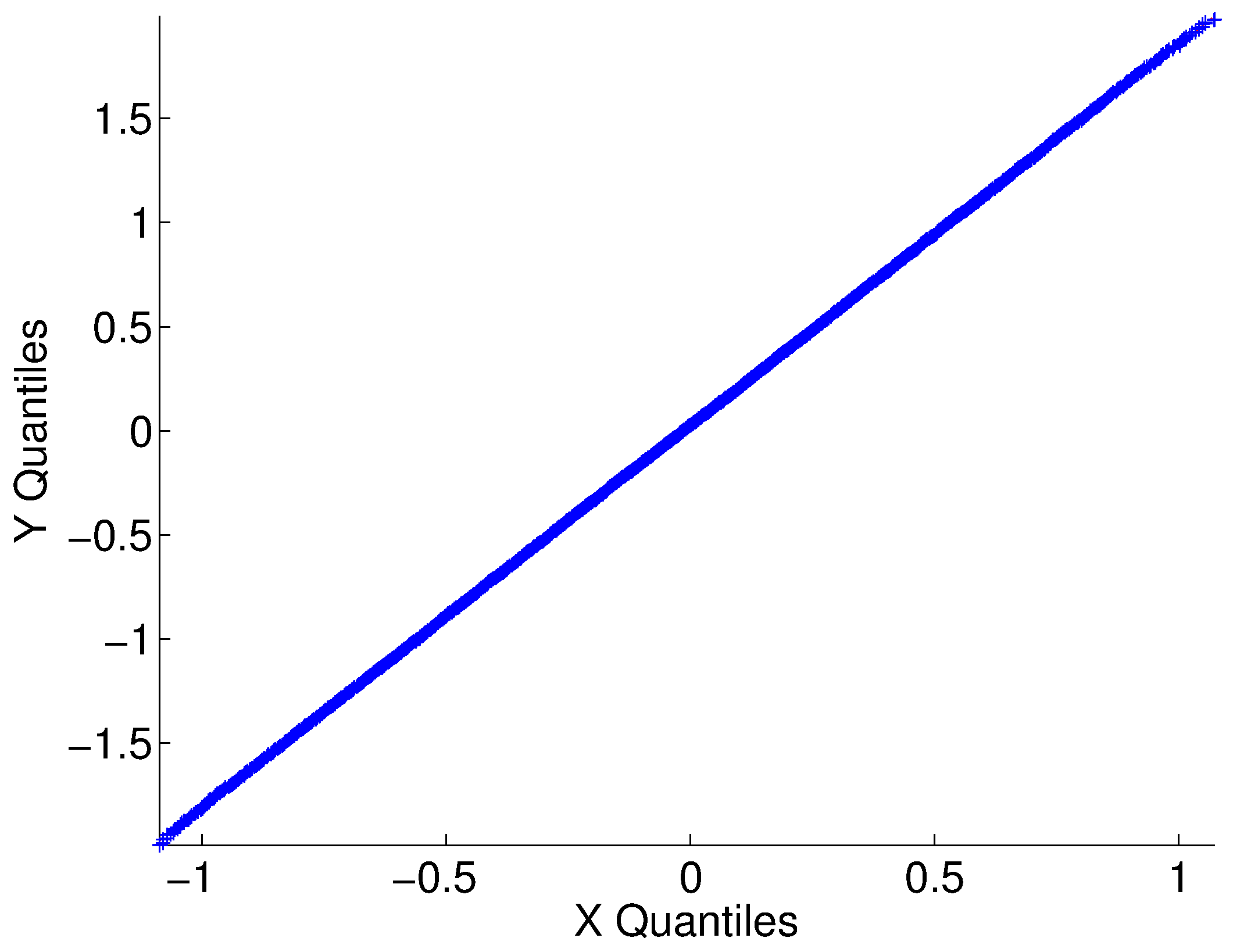} 
\includegraphics[width=0.23\textwidth]{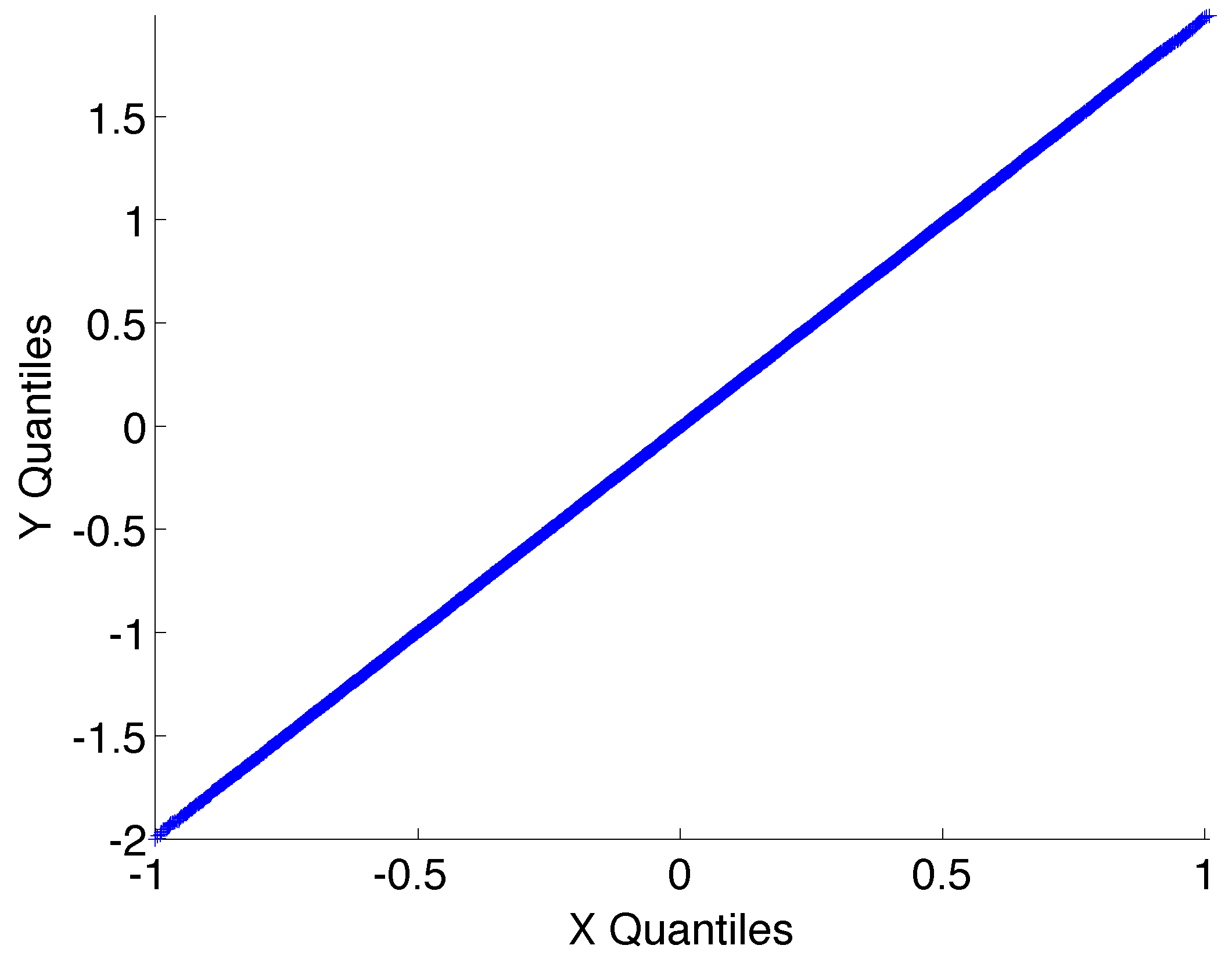}
\end{center}
\caption{\emph{Histogram of the eigenvalues of $R_{O(d),n,\text{nonHaar}}$ with entries sampled from $O(d)$ but not following the Haar measure when $(n,d)=(1000,2)$ (left) and $(1000,3)$ (left middle). We see an outlier at $14$ in the left subfigure, as is indicated by the black arrow, and three outliers in the left middle subfigure, as are indicated by the black arrows. To show that the bulk of the eigenvalues are close to the semi-circle, the QQplot of all the eigenvalues of $R_{O(d),n,\text{nonHaar}}$, which are less than $3$, versus the eigenvalues of symmetric Gaussian random matrix of size $dn\times dn$ is plotted in the right middle (resp. right) subplot when $(n,d)=(1000,2)$ (resp. $(n,d)=(1000,3)$). Note the existence of the outliers which might be falsely interpreted as ``information''. }}
\label{fig:OdnonHaar}
\end{figure}

The histogram of $R_{O(d),n,\text{nonHaar}}$'s spectrum is shown in Figure \ref{fig:OdnonHaar}, when $n=1000$ and $d=2,3$. We also show the QQplot of all the eigenvalues of $R_{O(d),n,\text{nonHaar}}$, which are less than $3$, versus the eigenvalues of symmetric Gaussian random matrix of size $dn\times dn$. According to the QQplot, we may infer that the bulk of the empirical spectral distribution follows the semi-circle law. Note that the outliers might be {\it falsely} interpreted as ``information'', so we should be careful about the sampling scheme on the group matrix. In other words, there exist structures inside the block random matrix that might be misleading.

\subsubsection{Special linear group $Sl(d,\RR)$}\label{Section:NumericalSldR}
Consider a $n\times n$ symmetric block matrix $R_{Sl(d,\RR),n}$ with $d\times d$ entries so that its $(i,j)$-th entry is sampled from $Sl(d,\RR)$ by the following steps.
For each $(i,j)$, $i<j$, get a random $d\times d$ matrix with i.i.d. Gaussian entries, and denote it as $g$. If $|\det(g)|=0$, we resample another matrix until we get $g$ with $|\det(g)|>0$. Then define $R_{Sl(d,\RR),n}(i,j)$ to be $|\det(g)|^{-1/d}g$.
Then, if $\det(R_{Sl(d,\RR),n}(i,j))=1$, we get a component in $Sl(d,\RR)$; otherwise, flip the sign of the first column to ensure we get a component in $Sl(d,\RR)$.

The histogram of spectra of $R_{Sl(d,\RR),n}$ with $d =2,3,4,5$ are shown in Figure \ref{fig:Sld2-1} when $n=1000$. We also show the QQplot of the eigenvalues of $R_{Sl(d,\RR),n}$ versus the eigenvalues of symmetric Gaussian random matrix of size $dn\times dn$ in Figure \ref{fig:SldQQplot}. Note that starting from $Sl(4,\RR)$, the histogram are semi-circle-like.

Note that the histogram of $Sl(2,\RR)$ spreads broadly. The distribution and moments calculation of the sampling scheme on $Sl(d,\RR)$ are detailed in Subsection \ref{subsec:CompsSLdAppendix}, where we see the Cauchy-like behavior of the square of the largest singular value of $R_{Sl(2,\RR),n}$ (see Corollary \ref{coro:SL2svsAreLikeCauchy}). So our theory simply does not apply to this case and the fact that we do not get a semi-circle limit in this case is not surprising. 

On the other hand, the case of $Sl(3,\RR)$ falls under the umbrella of our theory (see Corollary \ref{coro:SLdCase}). The QQplot of the distribution of $R_{Sl(3,\RR),n}$ against the semi-circle law shows a few outliers. This of course is not in contradiction with our theoretical results: we have shown convergence of the LSD of $R_{Sl(3,\RR),n}$ but this naturally does not imply that the extreme eigenvalues of this random matrix convergence to the endpoint of the LSD.

\begin{figure}[h]
\begin{center}
\includegraphics[width=0.23\textwidth]{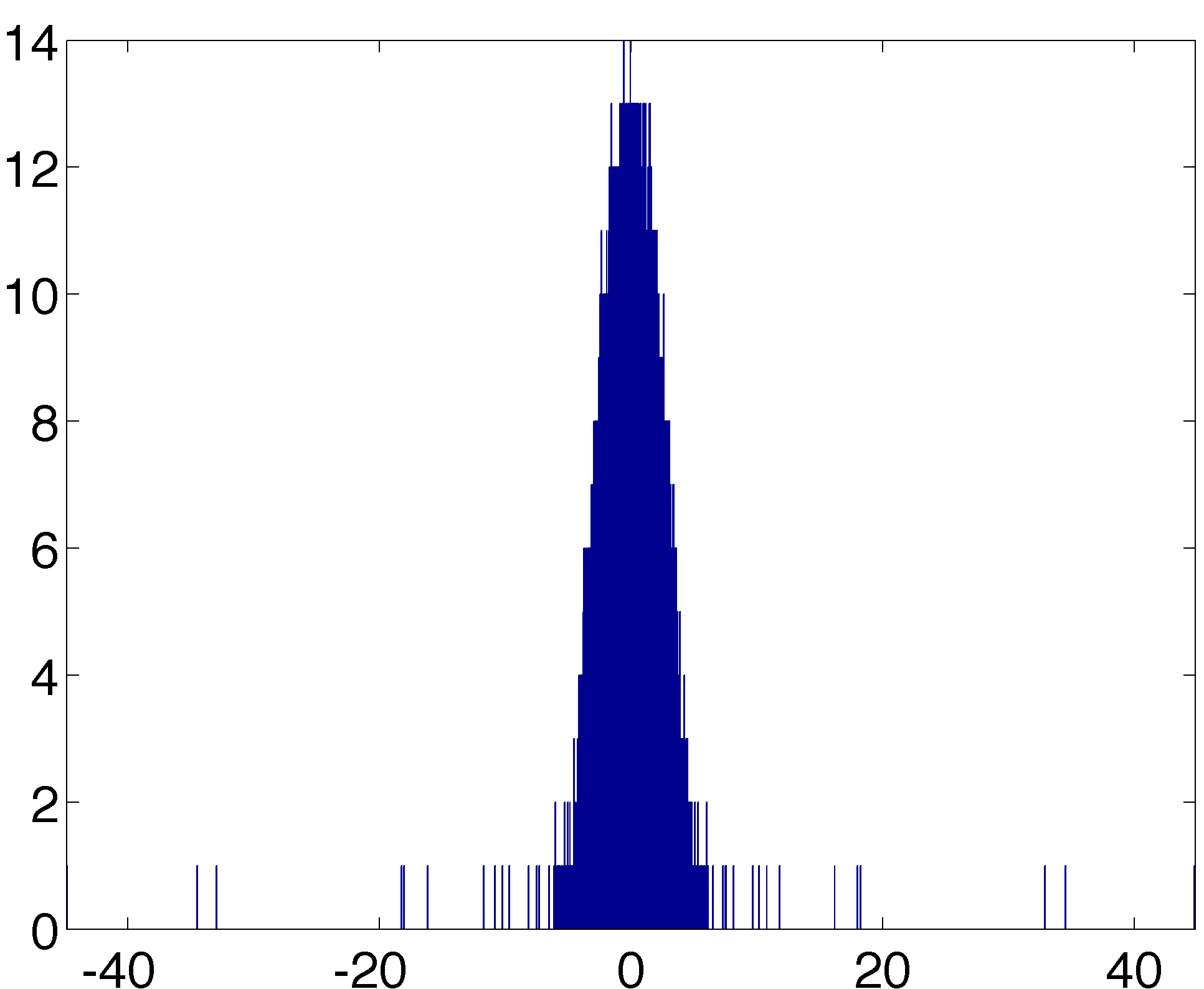} 
\includegraphics[width=0.23\textwidth]{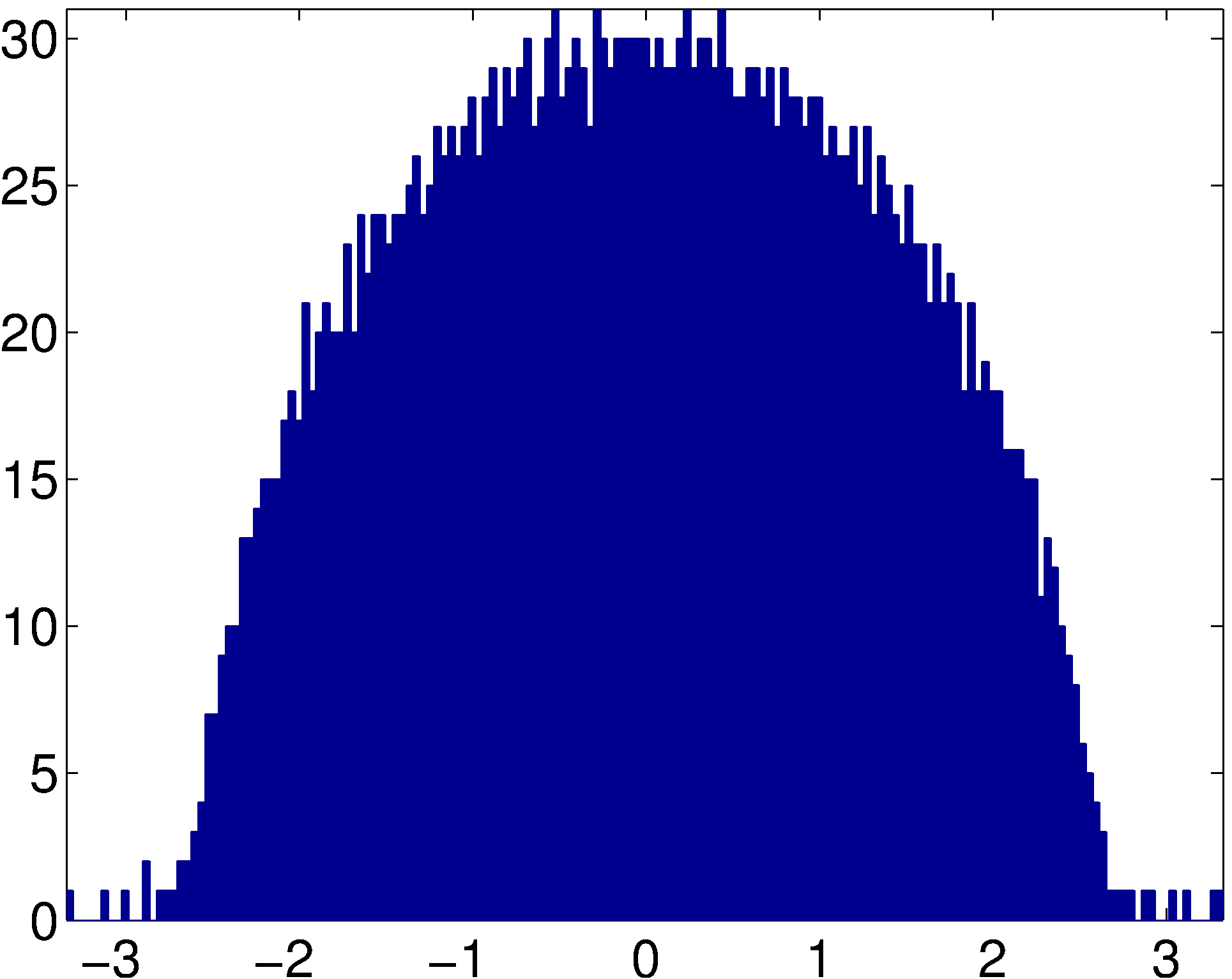}
\includegraphics[width=0.23\textwidth]{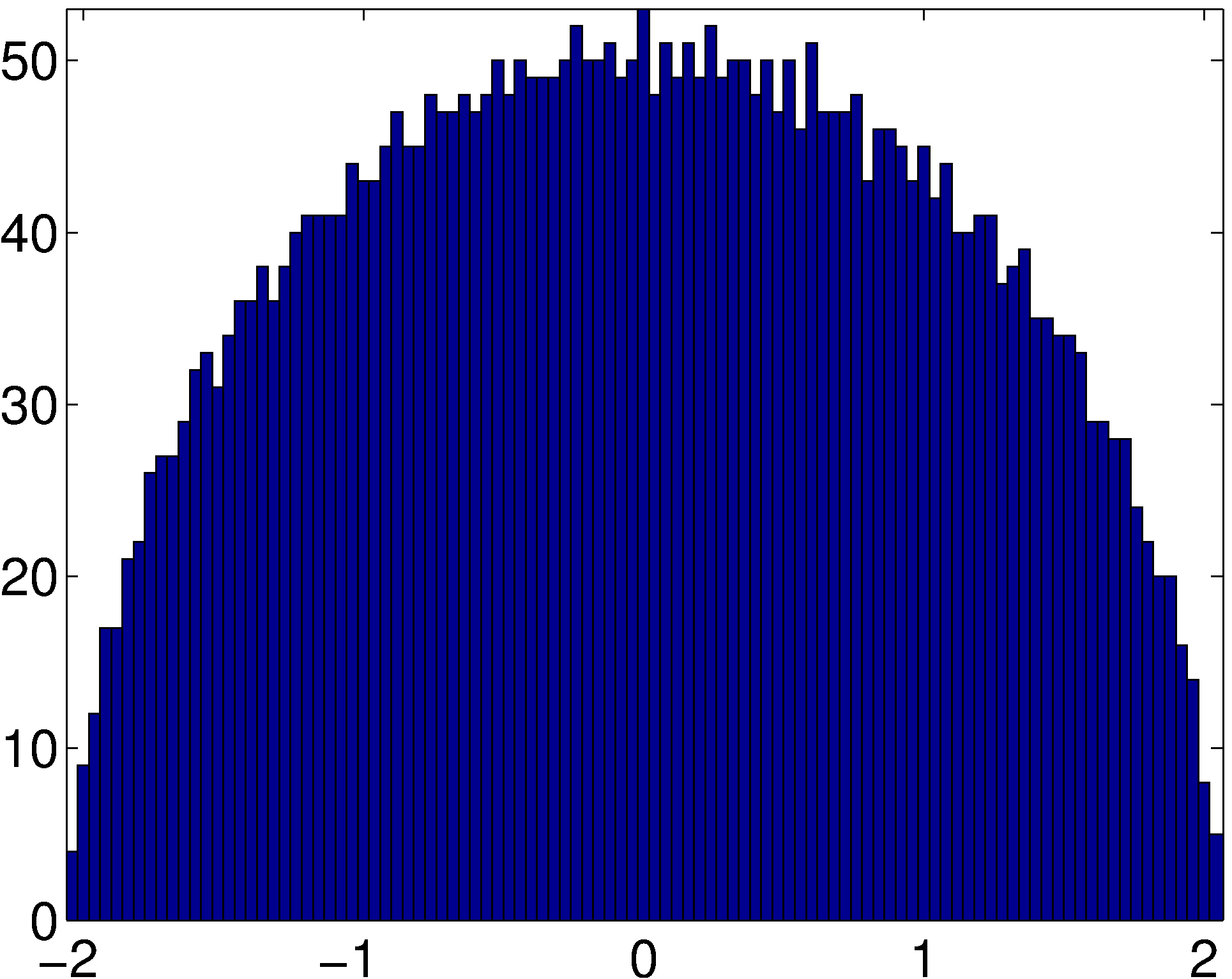} 
\includegraphics[width=0.23\textwidth]{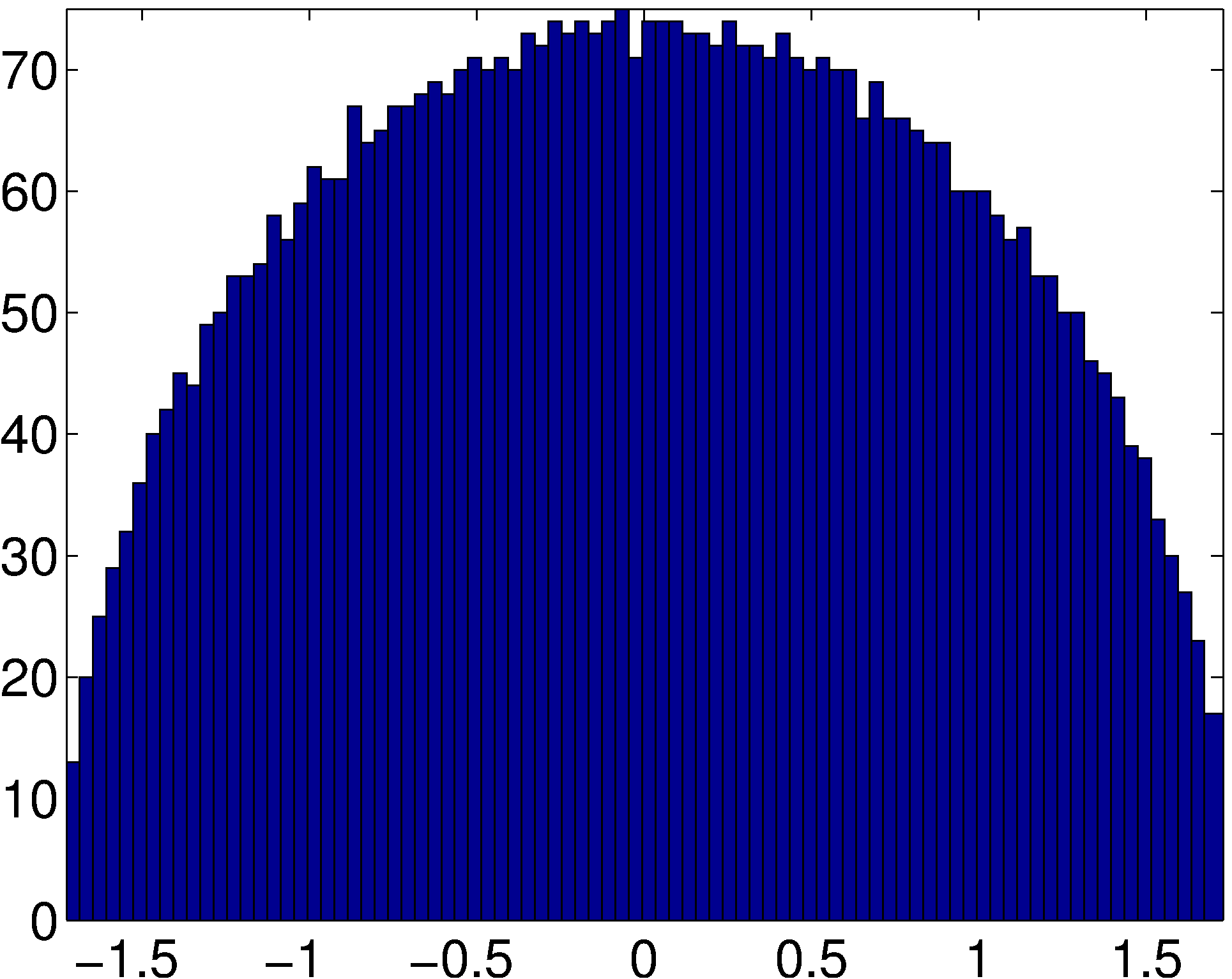}
\end{center}
\caption{\emph{Histogram of the eigenvalues of $R_{Sl(d,\RR),n}$ with $d=2,3,4,5$ (from left to right) and $n=1000$.}}
\label{fig:Sld2-1}
\end{figure}

\begin{figure}[h]	
\begin{center}
\includegraphics[width=0.23\textwidth]{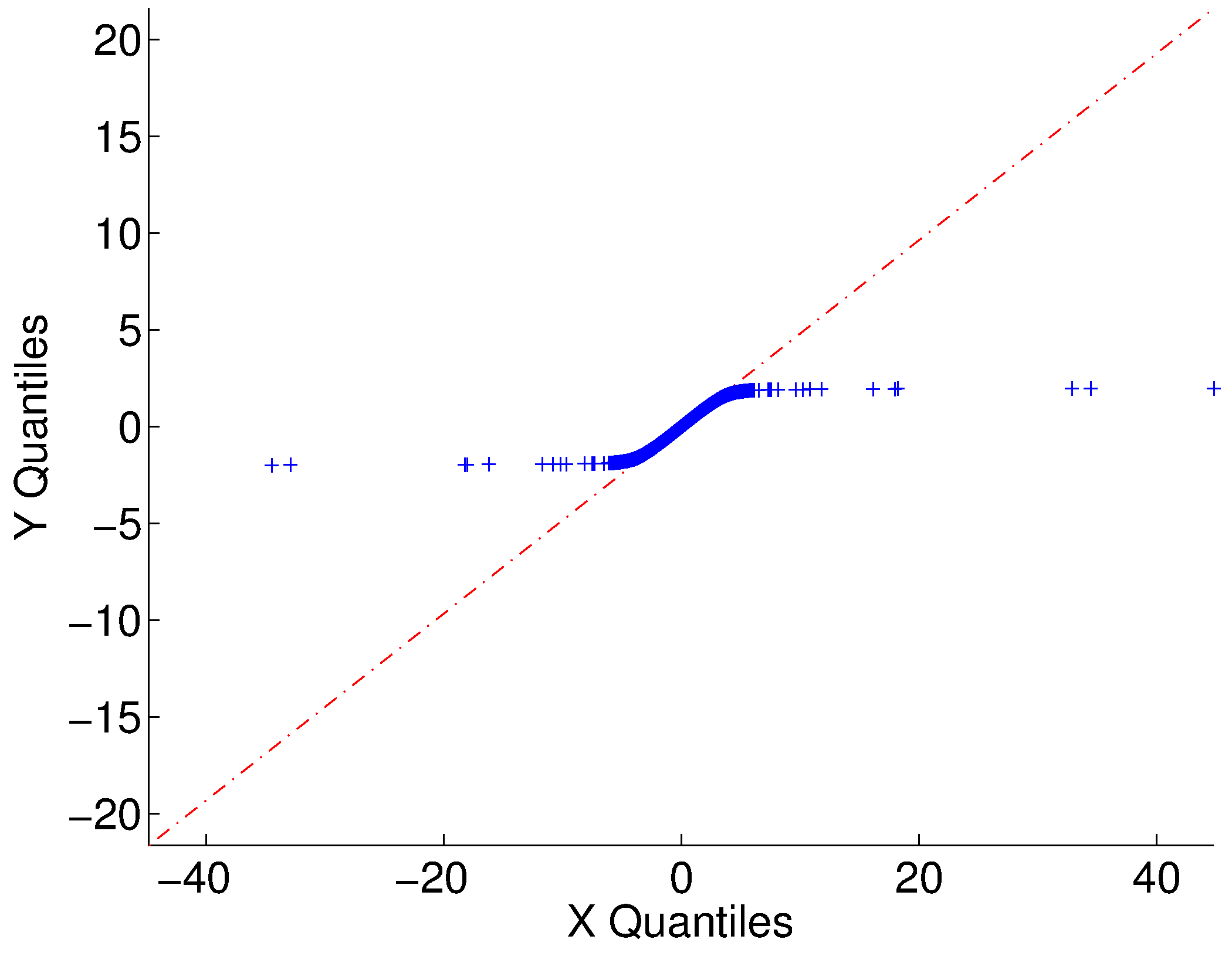} 
\includegraphics[width=0.23\textwidth]{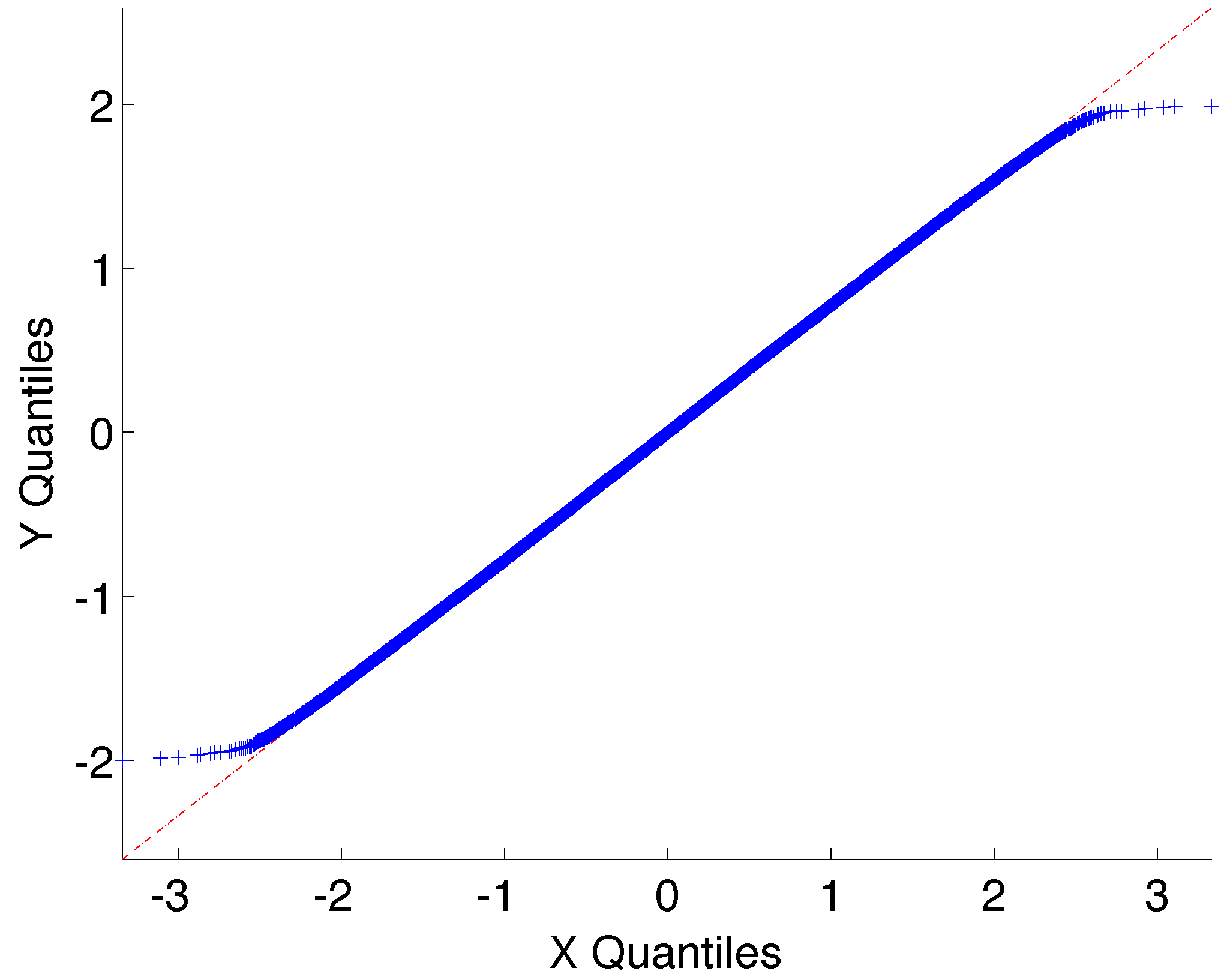}
\includegraphics[width=0.23\textwidth]{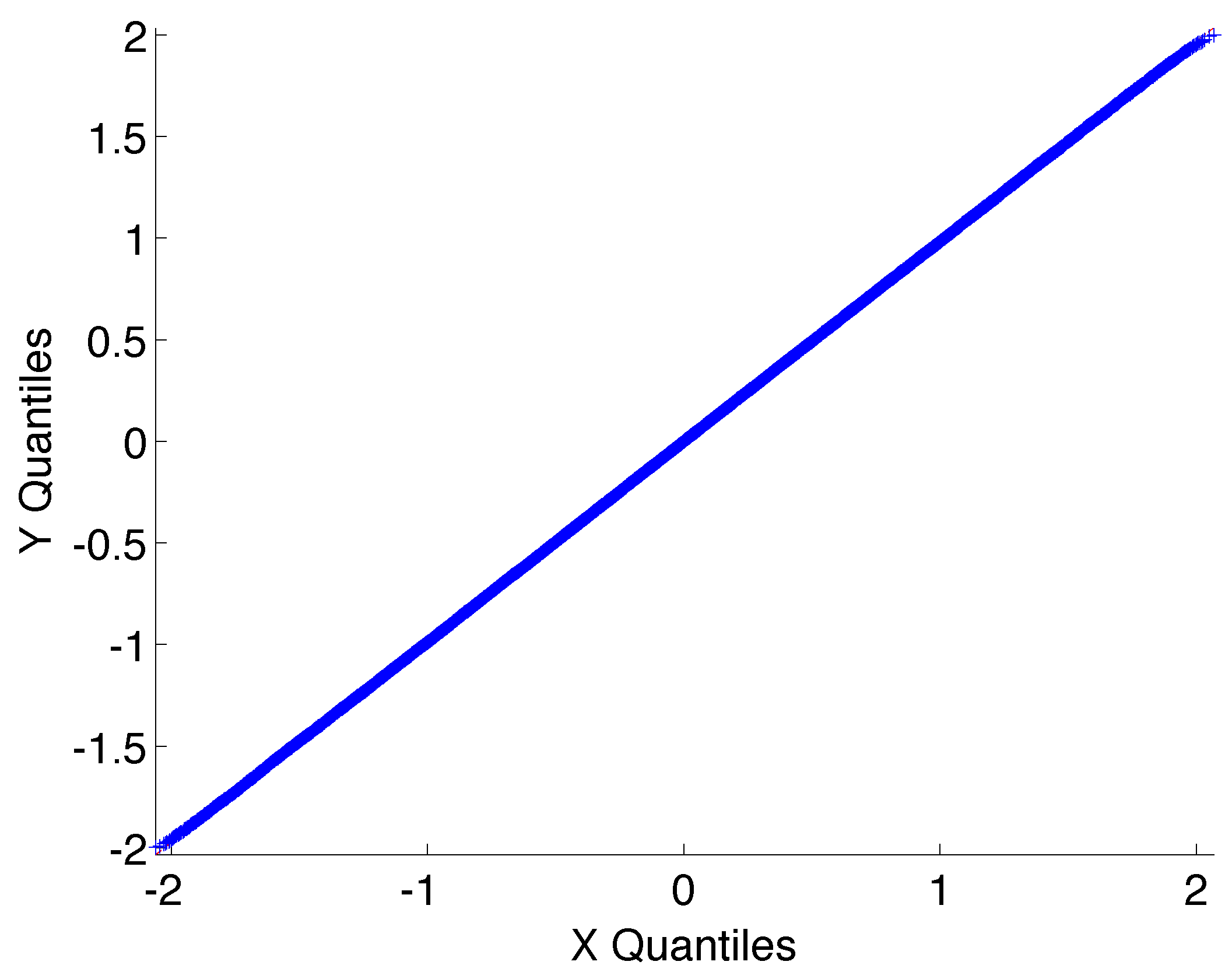} 
\includegraphics[width=0.23\textwidth]{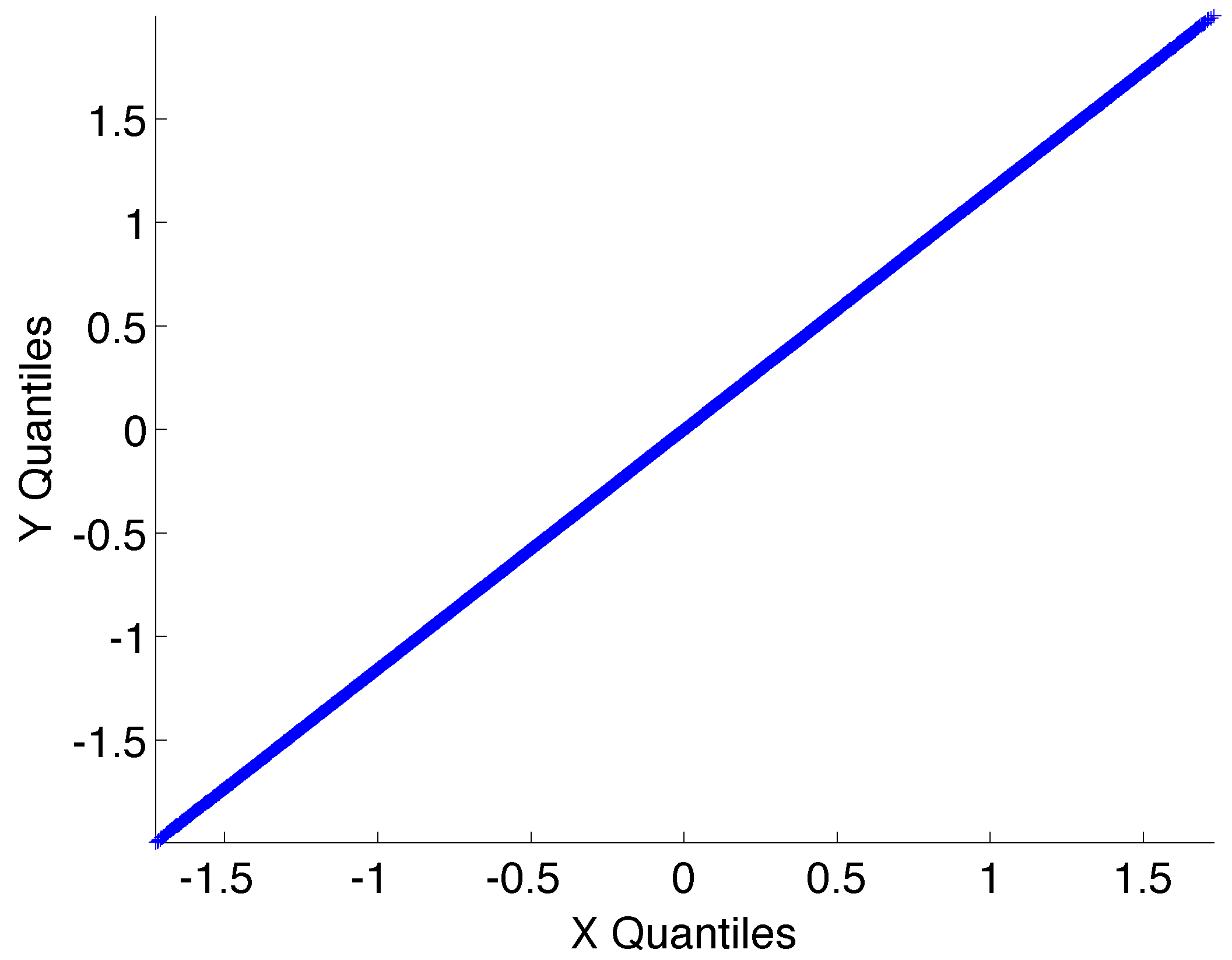}
\end{center}
\caption{\emph{QQplot of the eigenvalues of $R_{Sl(d,\RR),n}$ versus the eigenvalues of symmetric Gaussian random matrix of size $dn\times dn$ with $d=2,3,4,5$ (from left to right) and $n=1000$}}
\label{fig:SldQQplot}
\end{figure}

\subsection{Random block matrix with dependent blocks}
We now show how the GCL behaves in the setup with dependence among the blocks as is discussed in Corollary \ref{coro:controlStieltjesLatentVariablesCaseBlockMatrices}. One practical problem of this kind is the class averaging algorithm. 

The discretized images are simulated as
\[
\mathcal{X}_p:=\{Z^p_i\}_{i=1}^n\subset \RR^p, 
\] 
where $Z^p_i$, $i=1,\ldots,n$ are prepared in the following way. 
Suppose there are $p$ pixels in $\{-L,-L+1,\dots,L-1,L\}\times \{-L,-L+1,\dots,L-1,L\}$, which is the discretization of $[-L,L]\times [-L,L]$ in the Cartesian grid. The image $Z^p_i$ is set by take a Gaussian random vector $Z\sim \mathcal{N}(0,I_p)$. 
Please see Figure \ref{fig:VDMnull} for one of the realization. With $\mathcal{X}_p$, for all $i,j=1,\ldots,n$, evaluate
\[
g_{ij}:=\argmin_{g\in SO(2)}\|Z^p_i-g\circ Z^p_j\|_{L^2},
\]
where $g\circ Z^p_j$ means the numerical rotation of $Z^p_j$ by $g$ in the Cartesian grid. 
If there is more than one minimizer, we choose the first one as $g_{ij}$. Then, find the rotational invariant distance (RID) by
\[
d_{\textup{RID},ij}:=\min_{g\in SO(2)}\|Z^p_i-g\circ Z^p_j\|_{L^2}.
\]
The $SO(2)$ is discretized to $N_r$ equally spaced degrees for the numerical minimization. 
See Figure \ref{fig:VDMnull} for the distribution of the optimal rotation $g_{ij}$ and the distribution of RID.

With $g_{ij}$ and $d_{\textup{RID},ij}$ and a chosen $\epsilon>0$, we build up the $n\times n$ block matrix $\vdmS$ so that the $(i,j)$-th block is
\begin{align}\label{Numerical:S:definition}
\vdmS_{ij}=e^{-d_{\textup{RID},ij}^2/\epsilon}g_{ij},
\end{align}
where $i\neq j$, and the $n\times n$ diagonal block matrix $\vdmD$ so that the $i$-th diagonal block is
\begin{align}
\vdmD_{ii}=\sum_{j=1,j\neq i}^n e^{-d_{\textup{RID},ij}^2/\epsilon}I_2,
\end{align}
where $i=1,\ldots,n$. 
The histogram of the spectra of $\vdmD^{-1}\vdmS$ with $\epsilon$ being the $25\%$ quantile of $d_{\textup{RID},ij}$, $n=700$, $L=31$ and $N_r=240$ is shown in Figure \ref{fig:VDMnull}. Note that in this case, $p=3001\gg n$. 
\begin{figure}[h]	
\begin{center}
\includegraphics[width=0.32\textwidth]{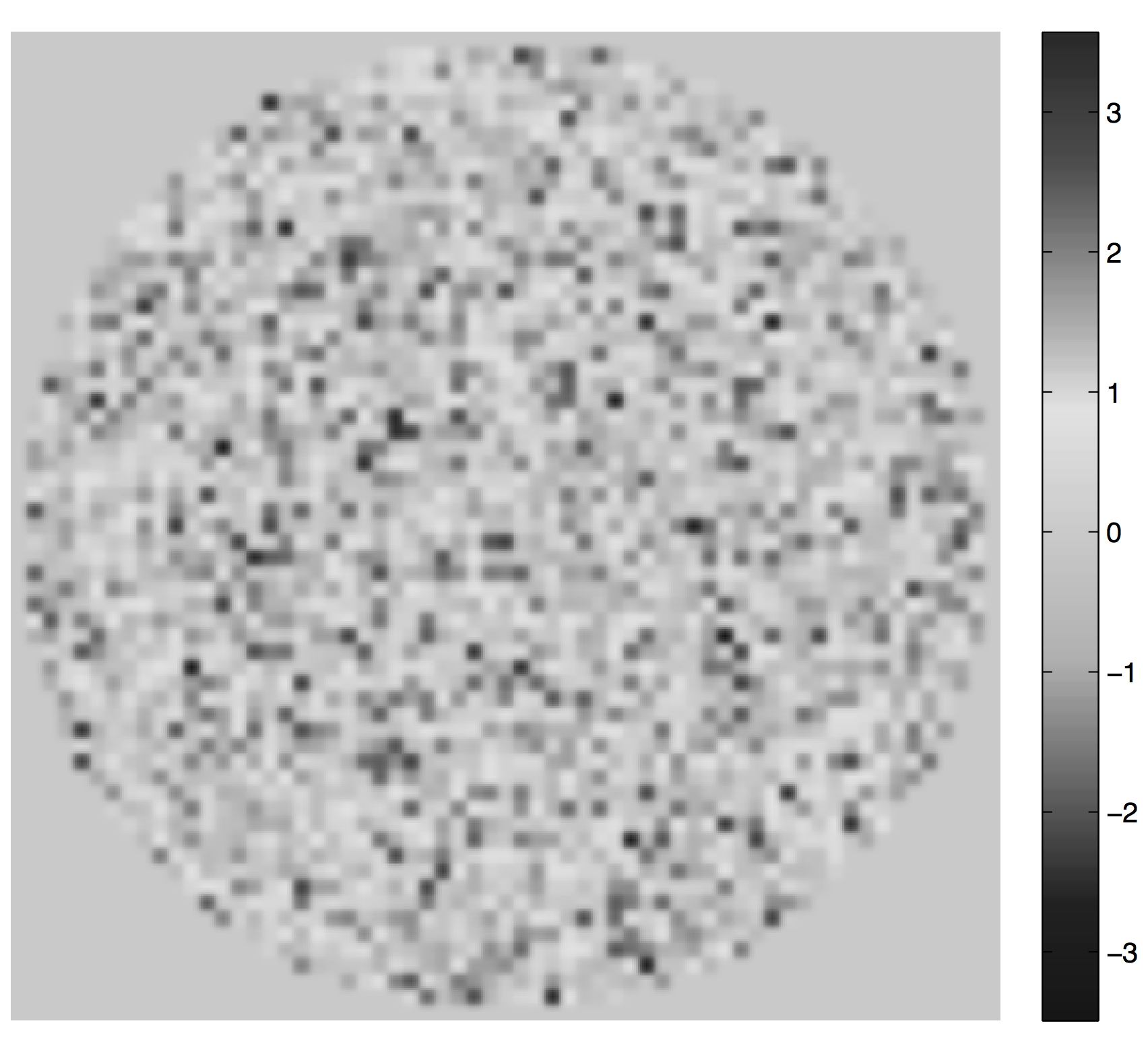} 
\includegraphics[width=0.32\textwidth]{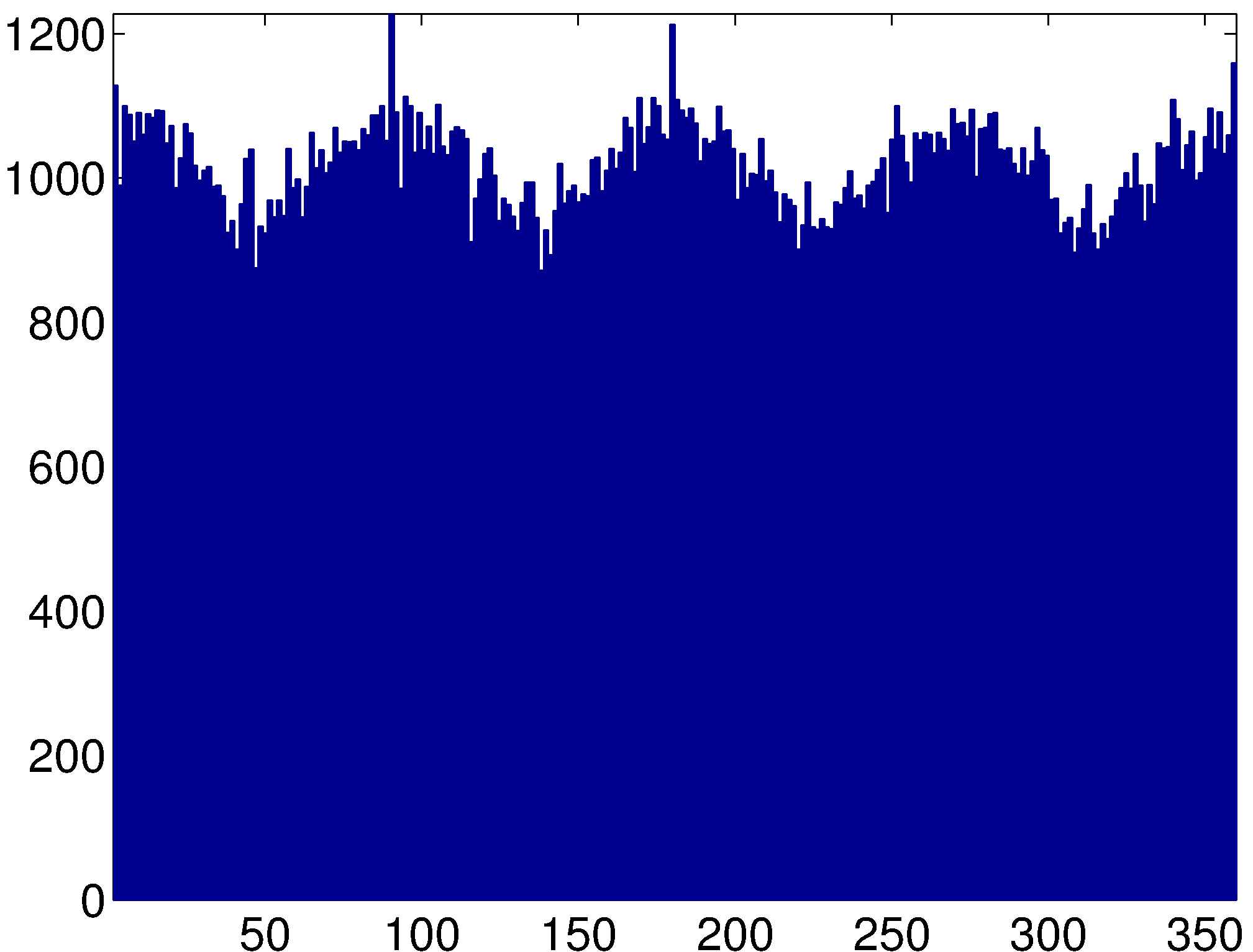}
\includegraphics[width=0.32\textwidth]{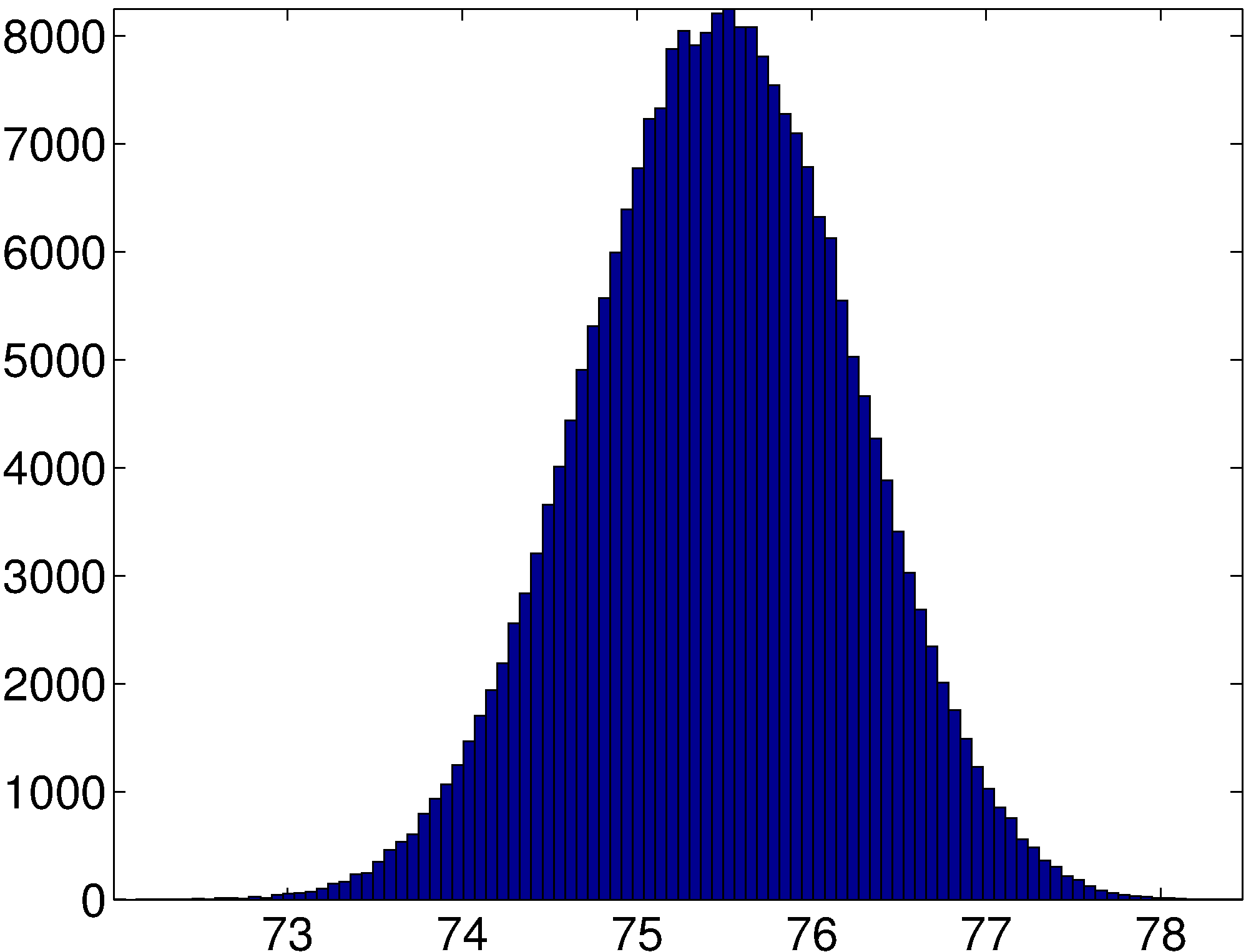} \\
\includegraphics[width=0.32\textwidth]{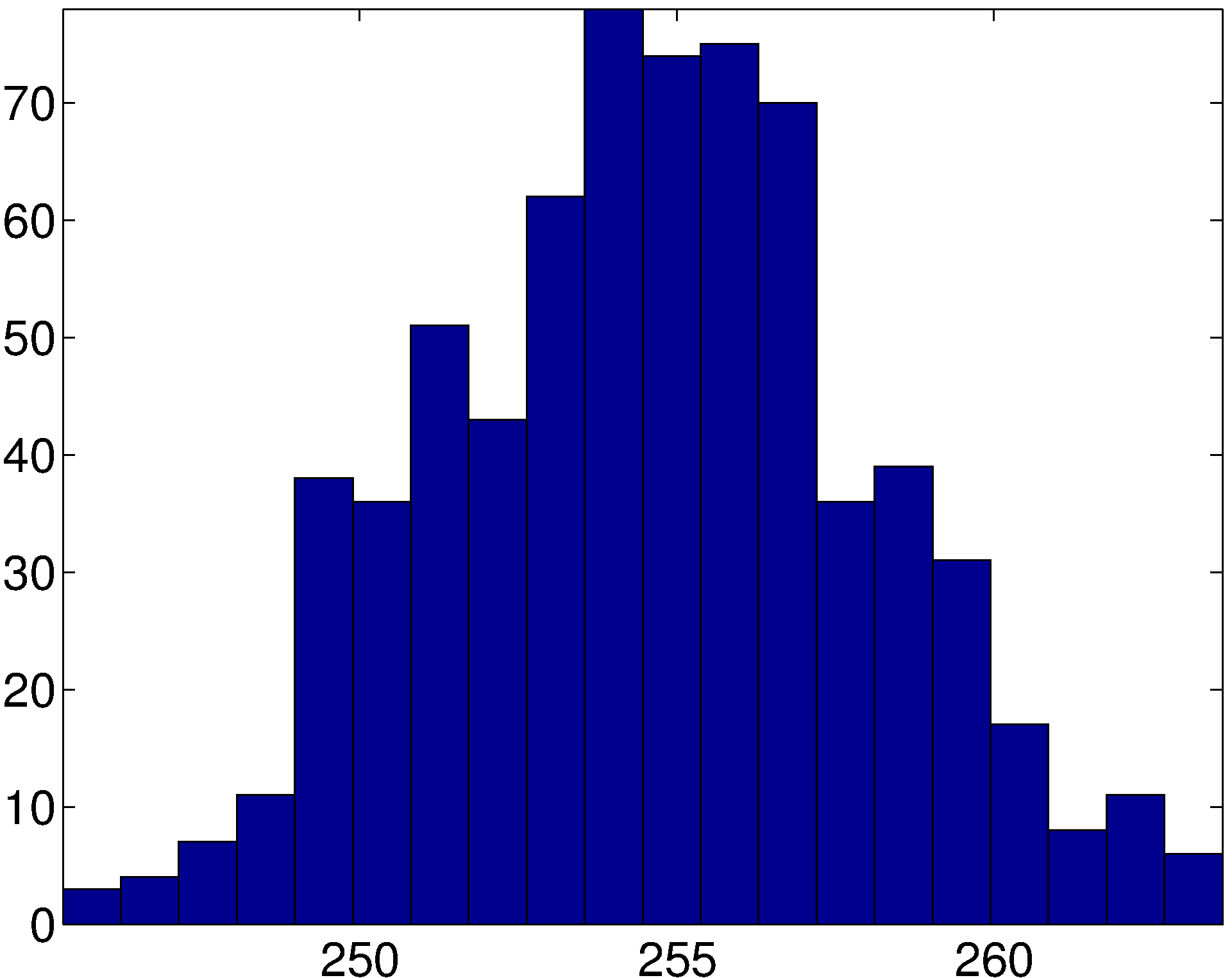} 
\includegraphics[width=0.32\textwidth]{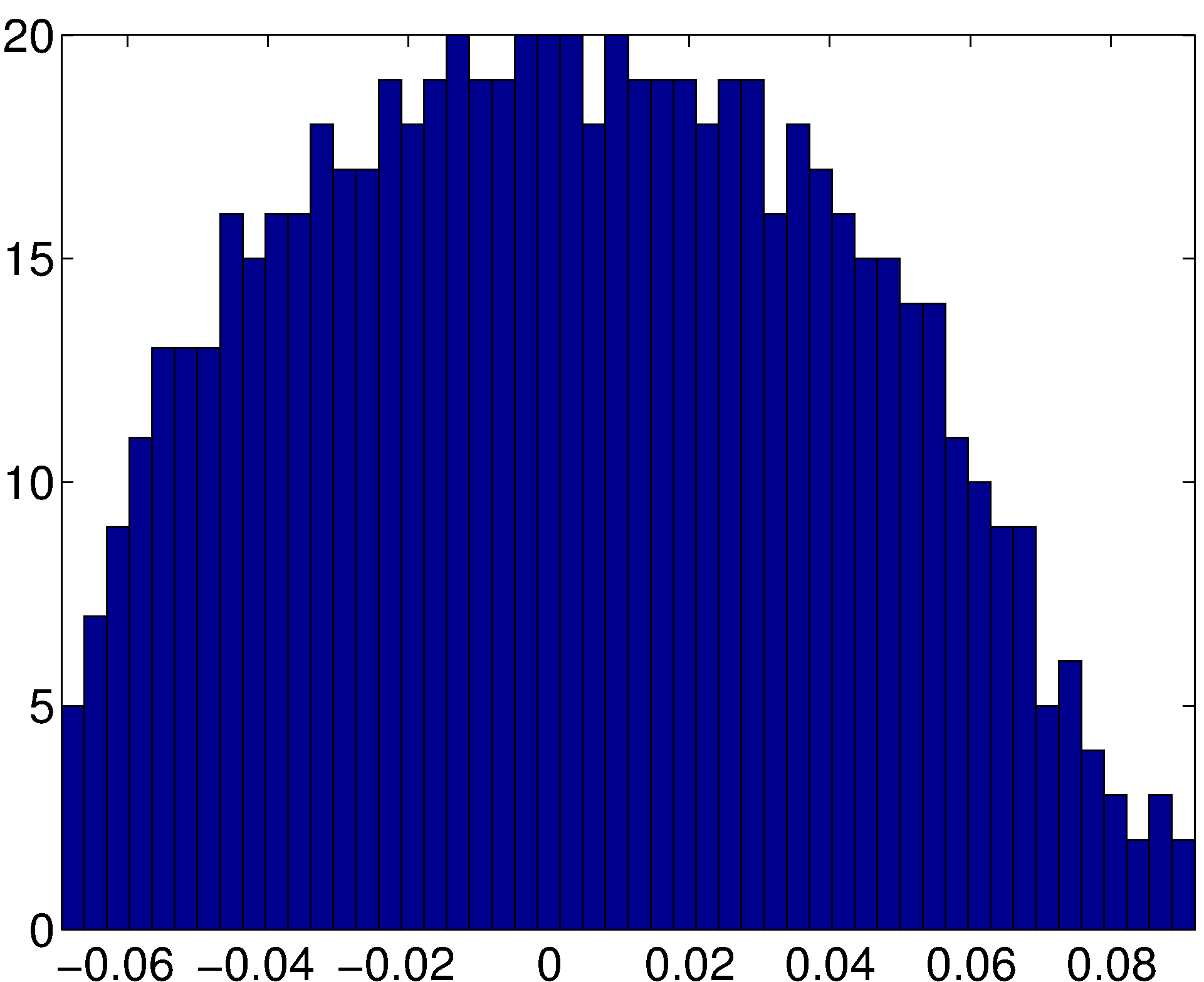}
\includegraphics[width=0.32\textwidth]{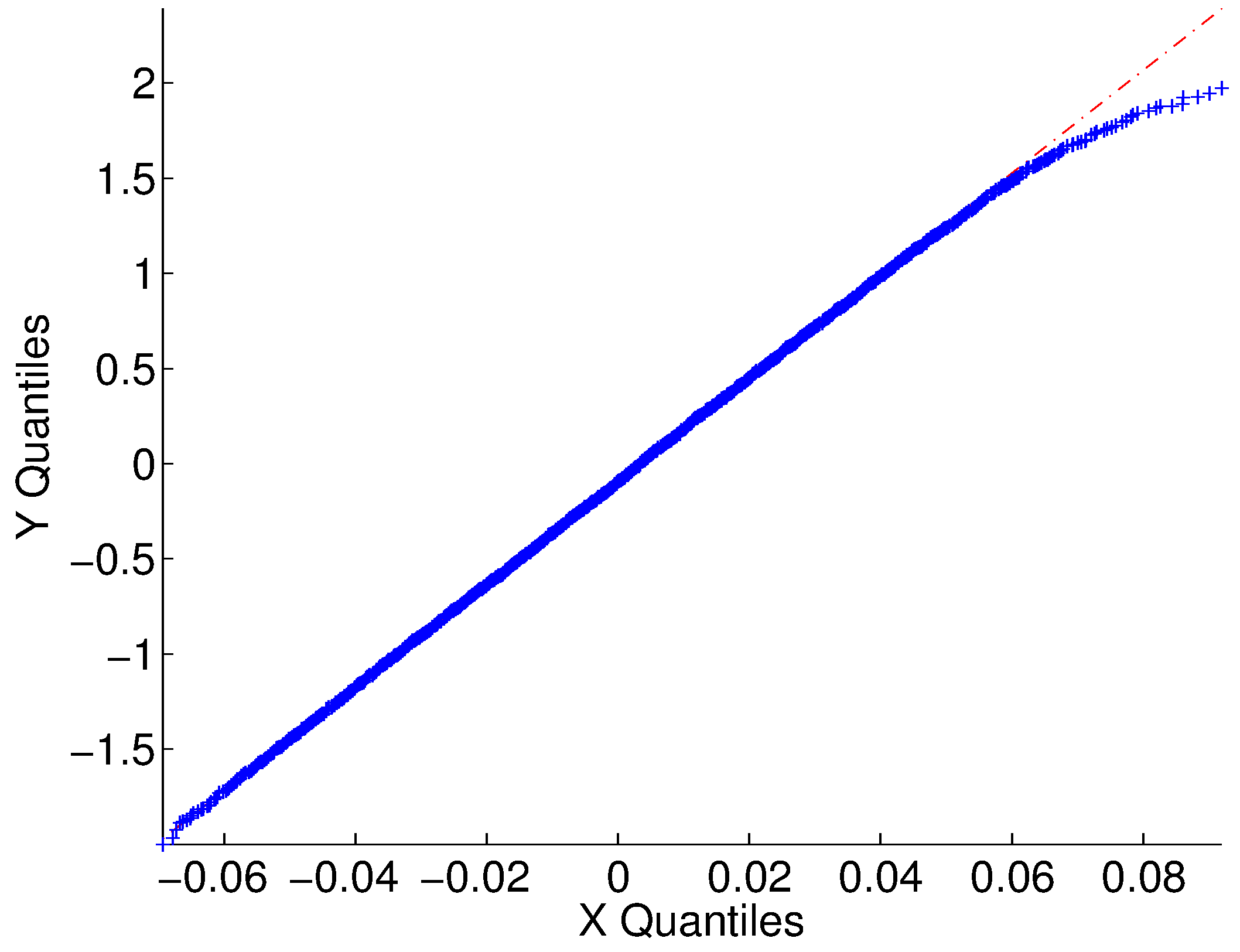}
\end{center}
\caption{The spectrum of GCL of $\mathcal{X}_{3001}$ with $n=700$, $L=31$ and $N_r=240$. Top left: the realization $Z^{3001}_1$; top middle:  the distribution of the optimal rotation $g_{ij}$, where the $x$-axis is the degree of the rotation; top right: the distribution of the RID $d_{\textup{RID},ij}$; bottom left: the histogram of the diagonal entries of $\vdmD$; bottom middle: the histogram of the eigenvalues of $\vdmD^{-1}\vdmS$; bottom right: the QQplot of the eigenvalues of $\vdmD^{-1}\vdmS$ versus the eigenvalues of symmetric Gaussian random matrix of size $700\times 700$.}
\label{fig:VDMnull}
\end{figure}

Note that the numerical rotation of images in the Cartesian grid might deteriorate the statistical property of the images, that is, the statistical property of $Z^p_i$ might be different from that of $R\circ Z^p_i$ for a generic rotation $R$. To eliminate this possibility, next we consider the following model which guarantees the invariance of the statistical property under rotation. Take
\[
\mathcal{X}_{0,p}:=\{Z^p_{0,i}\}_{i=1}^n\subset \RR^p, 
\] 
where $Z^p_{0,i}$, $i=1,\ldots,n$ are i.i.d. sampled from a Gaussian random vector with mean $0$ and covariance matrix $I_p$. The component in $\mathcal{X}_{0,p}$ will serve as a surrogate image which is defined on the uniform discretization of $S^1$ by $p$ grids; that is, $Z^p_{0,\ell}$ can be viewed as a function defined on the grid $\{(\cos(2\pi \ell/p), \sin(2\pi \ell/p))\in S^1\}_{\ell=1}^p$. Under this setup, the rotation is realized by cyclically permuting $Z^p_{0,i}$. Note that these ``surrogate images'' are sufficient for us to model the invariant statistical behavior of a purely noise image under rotation. Similarly, we define the optimal rotation (resp. RID and affinity) between $Z^p_{0,i}$ and $Z^p_{0,j}$ by $g_{0,ij}$ (resp. $d_{\textup{RID},0,ij}$ and $e^{-d_{\textup{RID},ij}^2/\epsilon}$), and hence the GCL. The histogram of eigenvalues of $\vdmD^{-1}\vdmS$ with $n=1000$ and $p=500,1000,2000$ are shown in Figure \ref{fig:VDMnull2}; the QQ plot of eigenvalues of $\vdmD^{-1}\vdmS$ with $n=1000$ and $p=500,1000,2000$ versus the eigenvalues of symmetric Gaussian random matrix of size $n\times n$ are shown in Figure \ref{fig:VDMnull2QQ}. 

\begin{figure}[h]	
\begin{center}
\includegraphics[width=0.32\textwidth]{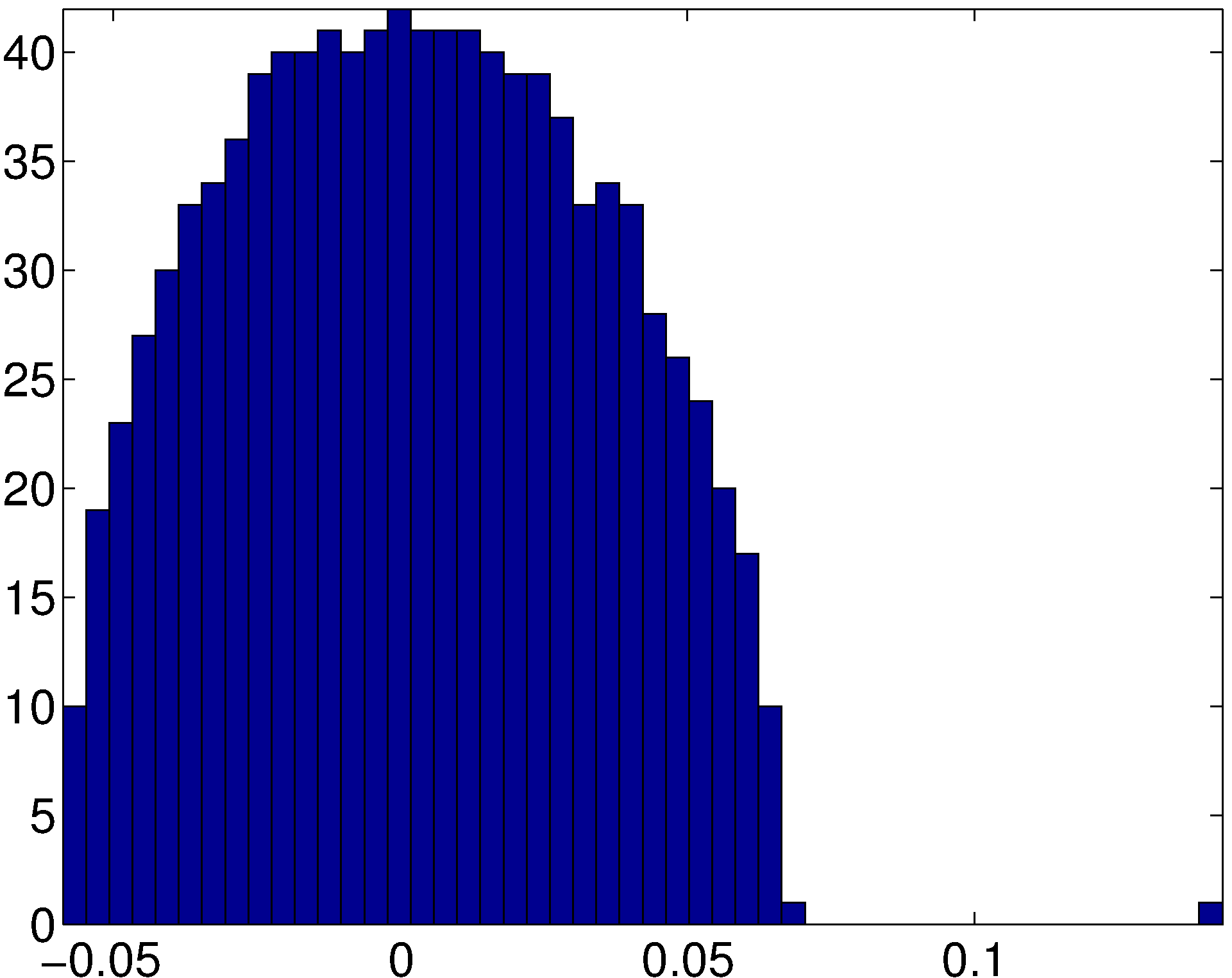}
\includegraphics[width=0.32\textwidth]{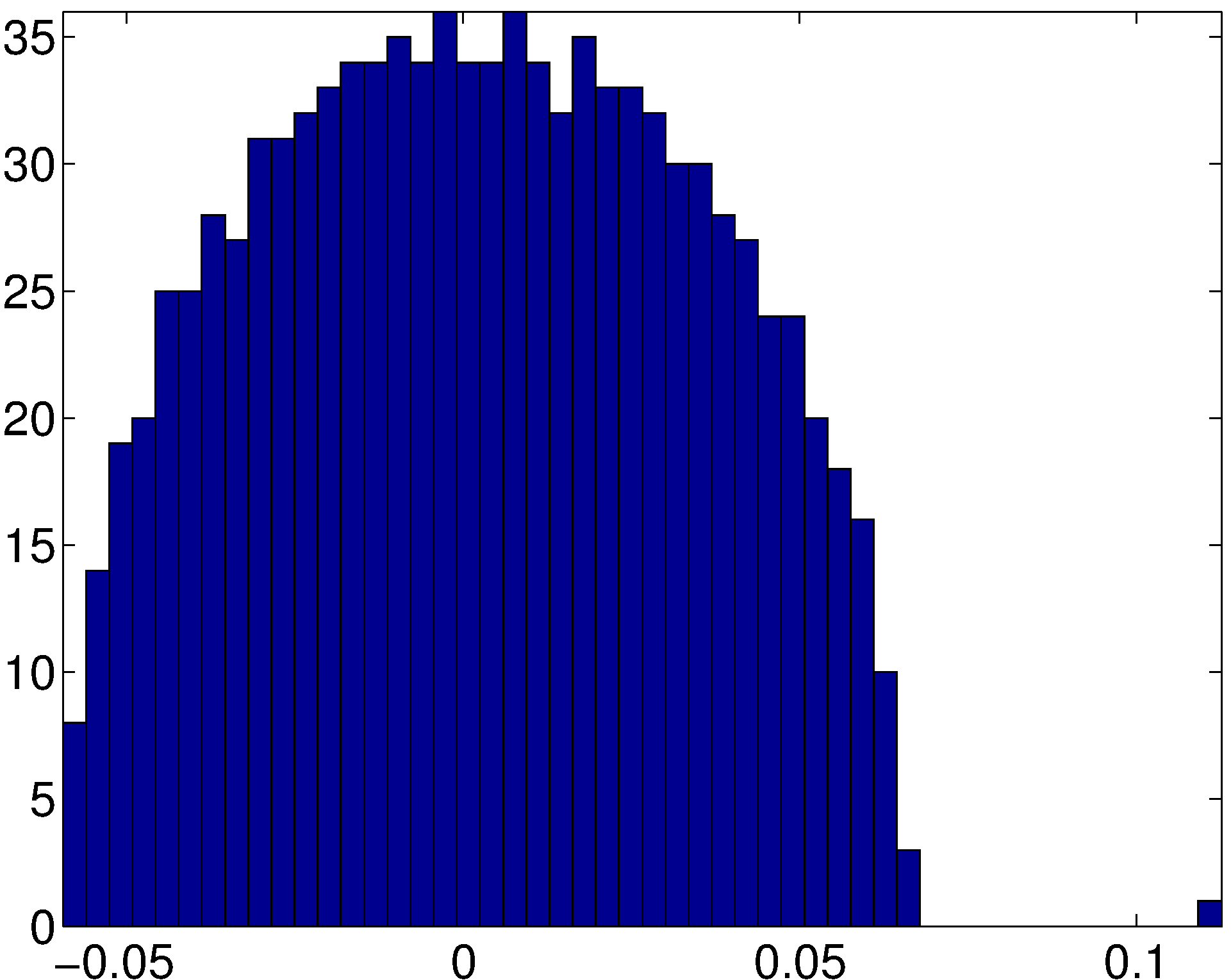}
\includegraphics[width=0.32\textwidth]{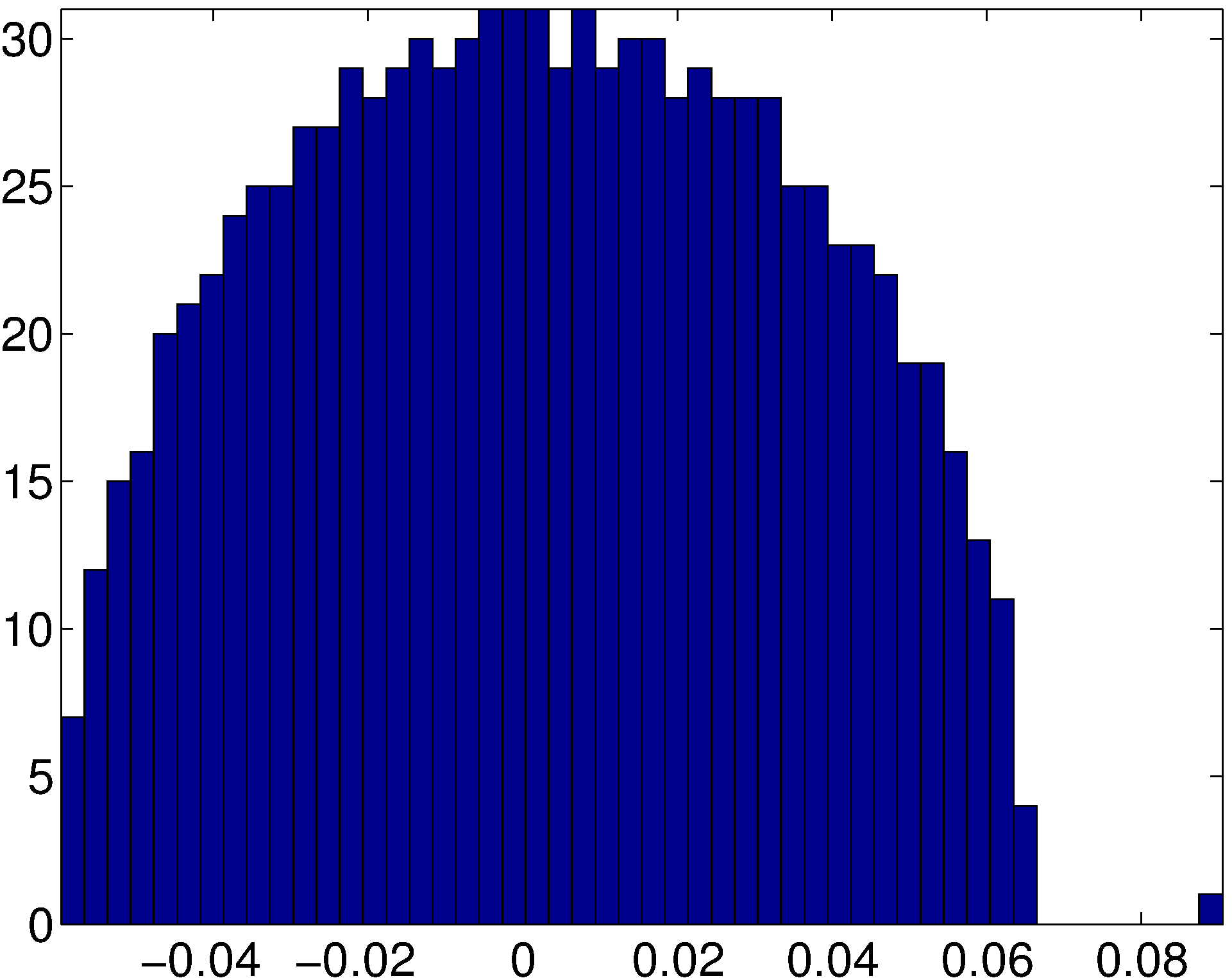}
\end{center}
\caption{
From left to right: The histogram of eigenvalues of $\vdmD^{-1}\vdmS$ defined from $\mathcal{X}_{0,p}$ with $n=1000$, $p=500,1000,2000$ and the affinity $e^{-d_{\textup{RID},ij}^2/\epsilon}$.
}
\label{fig:VDMnull2}
\end{figure}

\begin{figure}[h]	
\begin{center}
\includegraphics[width=0.32\textwidth]{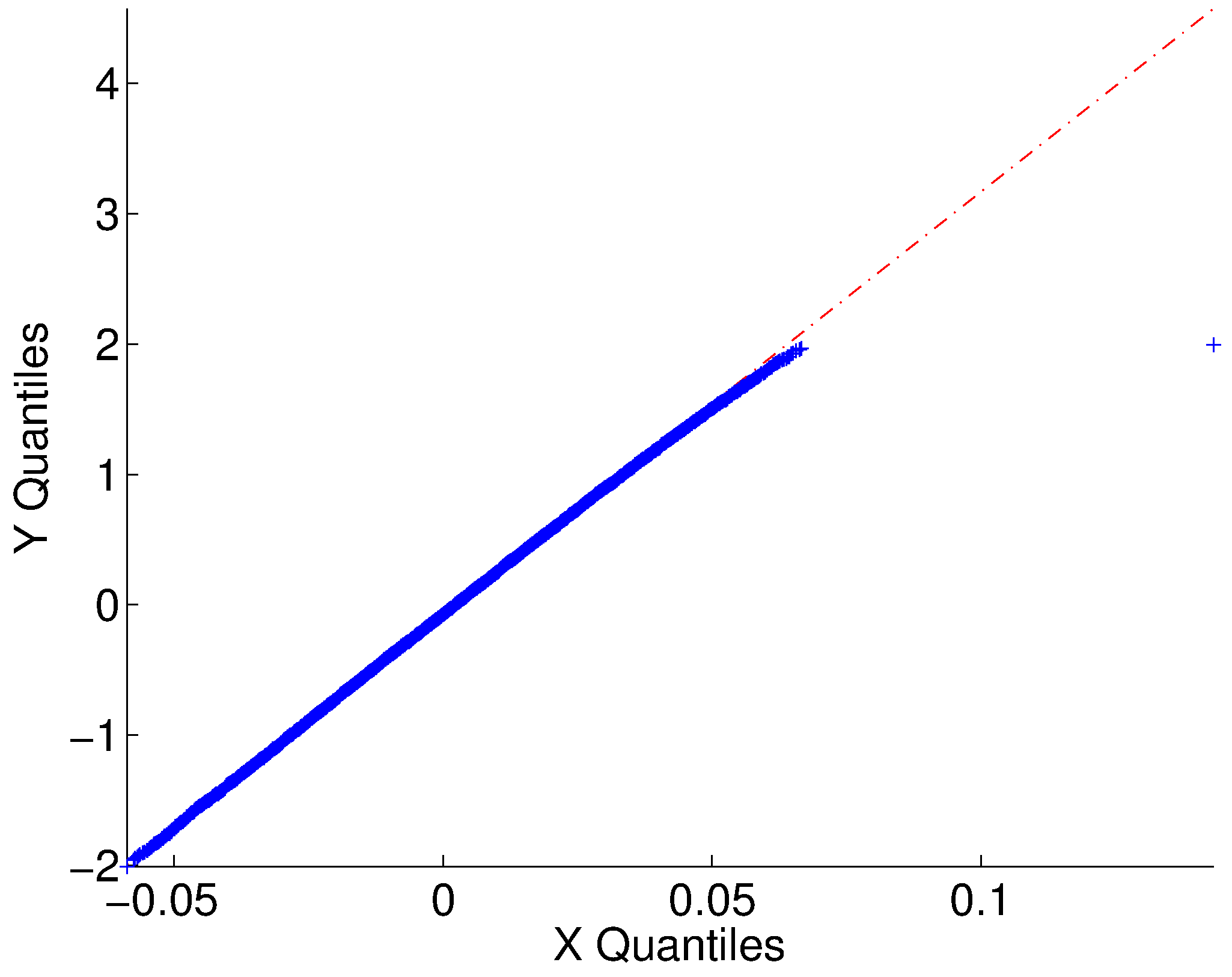}
\includegraphics[width=0.32\textwidth]{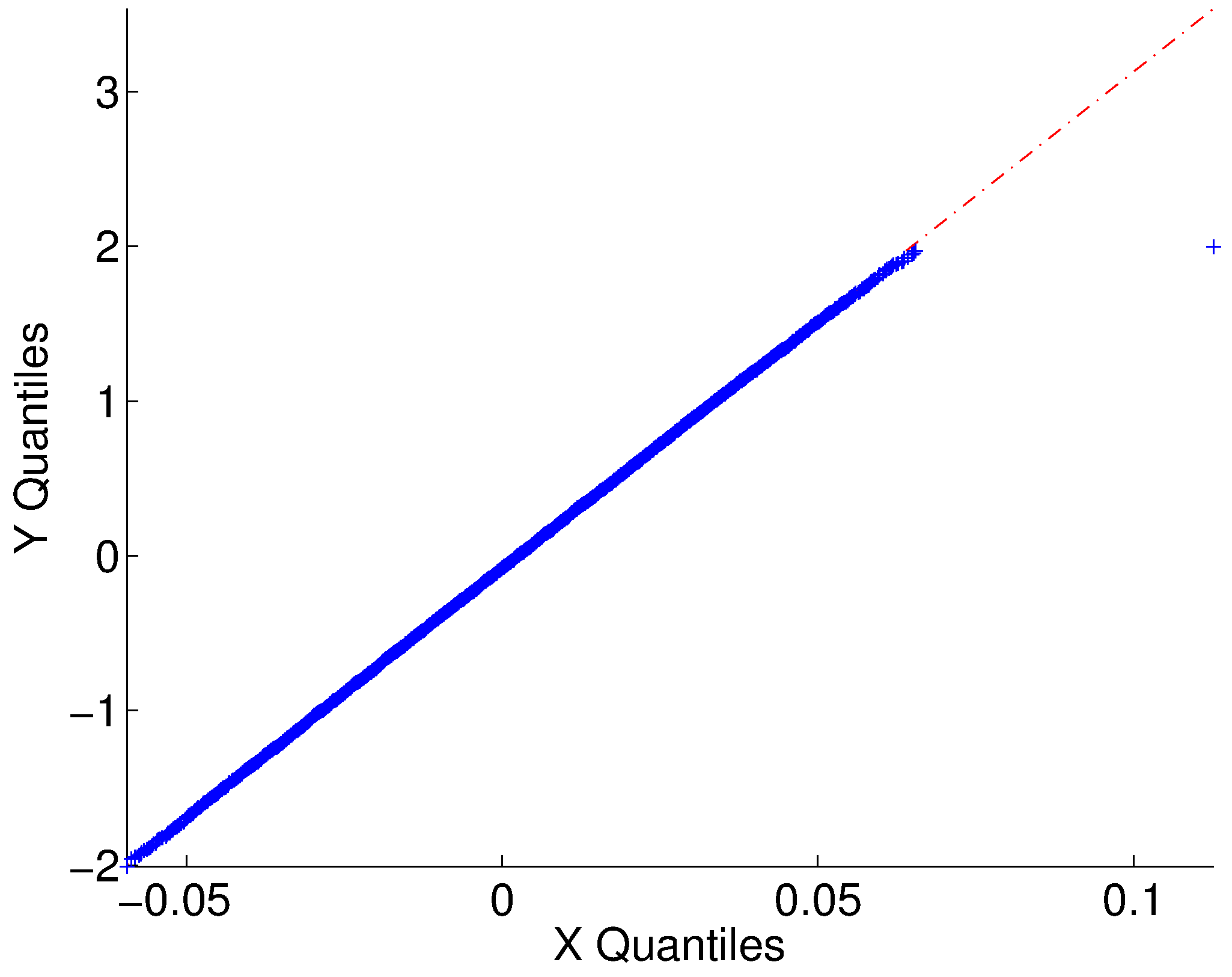}
\includegraphics[width=0.32\textwidth]{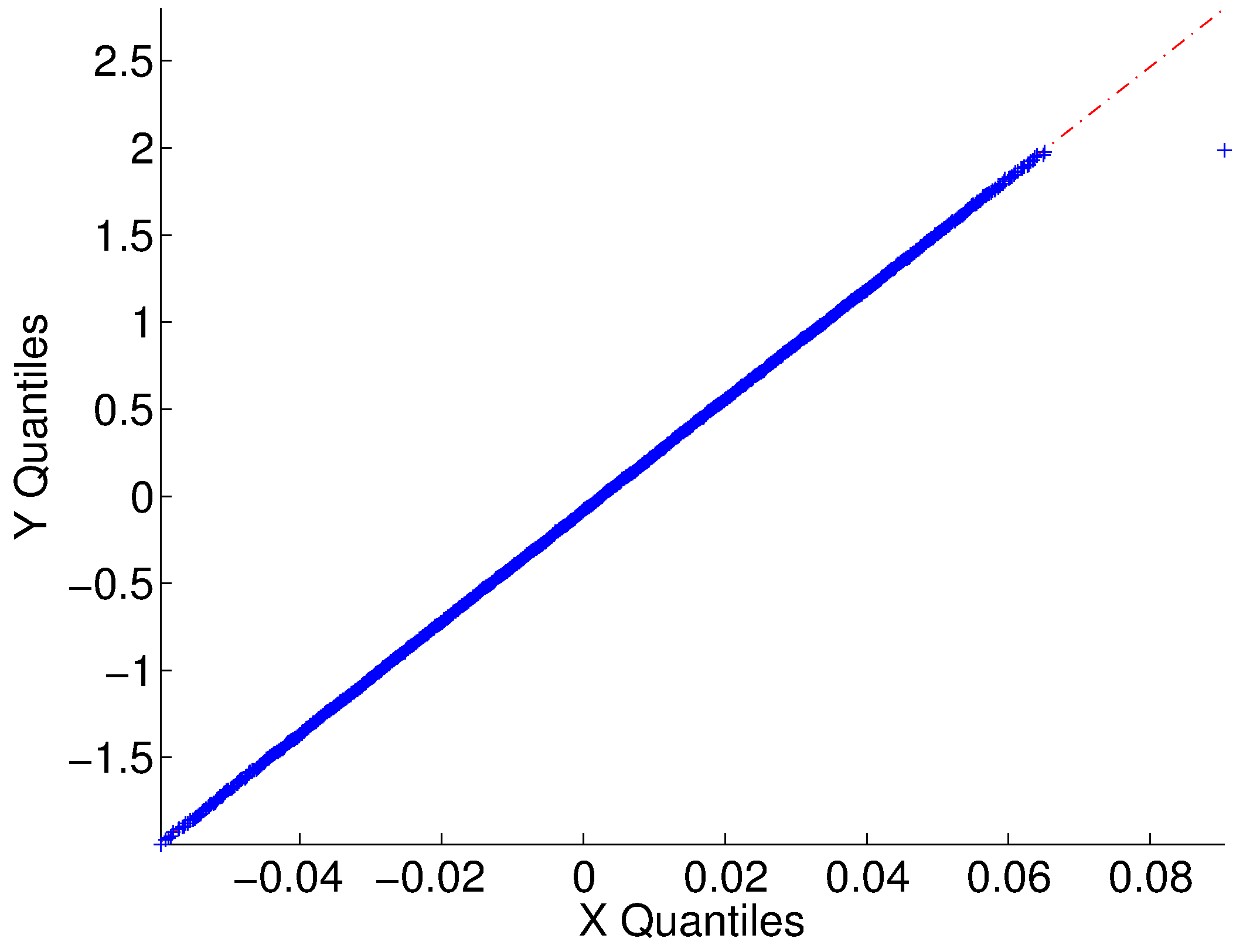}
\end{center}
\caption{
From left to right: the QQplot of eigenvalues of $\vdmD^{-1}\vdmS$ defined from $\mathcal{X}_{0,p}$ with $n=1000$, $p=500,1000,2000$ and the affinity $e^{-d_{\textup{RID},ij}^2/\epsilon}$ versus the eigenvalues of symmetric Gaussian random matrix of size $n\times n$.
}
\label{fig:VDMnull2QQ}
\end{figure}

According to Corollary \ref{coro:controlStieltjesLatentVariablesCaseBlockMatrices}, asymptotically when $n\to \infty$ the distribution of eigenvalues of $\vdmS$ (\ref{Numerical:S:definition}) is deterministic. However, we do not have a quantification of how the LSD looks like.
Furthermore, our subsequent paper \cite{NEKHautieng2014CGLRobust} can be used to shed light on some aspects of the numerical results regarding $\vdmD^{-1}\vdmS$. In particular, using Proposition 2.2 in that paper shows that, asymptotically, the rotationally invariant distance $d_{\textup{RID},ij}/\sqrt{p}$ can be approximated by a constant independent of $i,j$, provided $\log(pn^2)/\sqrt{p}\tendsto 0$. Using Lemma 3.2 in that paper ensures that the spectral distribution of $\vdmD^{-1}\vdmS$ is asymptotically the same as that of the matrix $M$ with block entries $g_{ij}/n$. Corollary \ref{coro:controlStieltjesLatentVariablesCaseBlockMatrices} shows that this matrix has a deterministic spectrum (asymptotically in $n$). In light of the results of Appendix \ref{app:subsec:vdmComps}, it is natural to expect that in the limit where $p\tendsto \infty$, the marginal distribution of $g_{ij}$ is uniform on $SO(2)$. These observations  help explain the shape of the histograms of eigenvalues plotted in Figure \ref{fig:VDMnull2}.

\appendix
\vspace{1cm}
\begin{center}
\textbf{\textsc{APPENDIX}}
\end{center}
\renewcommand{\theequation}{\Alph{section}}
\renewcommand{\theequation}{\Alph{section}-\arabic{equation}}
\renewcommand{\thelemma}{\Alph{section}-\arabic{lemma}}
\renewcommand{\thecorollary}{\Alph{section}-\arabic{corollary}}
\renewcommand{\thesubsection}{\Alph{section}-\arabic{subsection}}
\renewcommand{\thesubsubsection}{\Alph{section}-\arabic{subsection}.\arabic{subsubsection}}
\setcounter{equation}{0}  
\setcounter{lemma}{0}
\setcounter{corollary}{0}
\setcounter{section}{0}

\section{Technical results}
\subsection{Bound on the norm of a subblock of a matrix}
We recall and prove the following simple fact. 
\begin{fact}\label{fact:normOfSubblocksLessThanNormOfMatrix}
Suppose 
$$
T=\begin{pmatrix}
T_{11}& T_{12}\\
T_{21}& T_{22}
\end{pmatrix}\;.
$$
Then
$$
\opnorm{T_{ij}}\leq \opnorm{T}\;.
$$
\end{fact}
\begin{proof}
Recall that if $W$ is a $n\times p$ matrix,
$$
\opnorm{W}=\sup_{u\in \mathbb{C}^n, \norm{u}=1}\sup_{v\in \mathbb{C}^p, \norm{v}=1}|u^* W v|\;.
$$
Let us show that $\opnorm{T_{12}}\leq \opnorm{T}$. Suppose that $T$ is $n\times n$ and $T_{12}$ is $d\times m$, where $d+m=n$. Note that 
$$
u^* T_{12} v=(u^*  0_{n-d}^*) T \begin{pmatrix} 0_{d}\\ v \end{pmatrix}=\tilde{u}^* T \tilde{v}\;.
$$
Of course $\norm{\begin{pmatrix} 0_d\\ v \end{pmatrix}}=\norm{v}$ and similary for $\begin{pmatrix} u\\ 0_{n-d} \end{pmatrix}$. 
By definition, for any $u$ and $v$ with unit norm, 
$$
|\tilde{u}^* T \tilde{v}|\leq \sup_{\alpha\in \mathbb{C}^n, \norm{\alpha}=1}\sup_{\beta\in \mathbb{C}^n, \norm{\beta}=1}|\alpha^* T \beta|=\opnorm{T}\;.
$$
So we have shown that 
$$
\opnorm{T_{12}}\leq \opnorm{T}\;.
$$
The same reasoning applies to the other sub-blocks of $T$.
\end{proof}
\subsection{On finite rank perturbations and Stieltjes transforms}

The following lemma is used repeatedly in our proofs. 

\begin{lemma}\label{lemma:controlDiffStieltjesFiniteRankPerturb}
Suppose $A$ and $B$ are Hermitian $n\times n$ matrices. Let $z\in \mathbb{C}^+$ and call $\imag{z}=v>0$. 
Then 
\begin{equation}\label{eq:boundDiffTraceResolventDeltaRank}
\left|\trace{(A-z\id)^{-1}}-\trace{(B-z\id)^{-1}}\right|\leq \frac{\rank{A-B}}{v}\;.
\end{equation}
\end{lemma}
\begin{proof}
Since $\Delta=A-B$ is Hermitian it is diagonalizable. Therefore, $\Delta=\sum_{k=1}^r \tau_k q_k q_k^*$, where $\tau_k\in \mathbb{R}$, $q_k \in \mathbb{C}^n$ and $r$ is the rank of $\Delta$. Let us call, if $1\leq j\leq r$ $P\Delta_j=\sum_{k=1}^j \tau_k q_k q_k^*$ and $P\Delta_0=0_{n\times n}$. 
We have 
\begin{align*}
(A-z\id)^{-1}-(B-z\id)^{-1}&\,=(B+\Delta-z\id)^{-1}-(B-z\id)^{-1}\\
&\,=\sum_{j=0}^{r-1} \left[(B+P\Delta_{r-j}-z\id)^{-1}-(B+P\Delta_{r-j-1}-z\id)^{-1}\right]\;.
\end{align*}
Of course, $B+P\Delta_{r-j-1}$ is a rank-1 perturbation of $B+P\Delta_{r-j}$. Using Lemma 2.6 of \cite{silversteinbai95}, we therefore have 
$$
\left|\trace{(B+P\Delta_{r-j}-z\id)^{-1}-(B+P\Delta_{r-j-1}-z\id)^{-1}}\right|\leq \frac{1}{v}\;.
$$
Therefore, since 
$$
\left|\trace{(A-z\id)^{-1}-(B-z\id)^{-1}}\right|\leq \sum_{j=0}^{r-1} \left|\trace{(B+P\Delta_{r-j}-z\id)^{-1}-(B+P\Delta_{r-j-1}-z\id)^{-1}}\right|\;,
$$
and the sum on the right hand side contains $r$ (i.e $\rank{A-B}$ terms), the result stated in the lemma follows. 
\end{proof}

\subsection{Invariance and moments}

\begin{lemma}\label{lemma:invarianceAndMomentCsq}
Let $B$ be a random $d\times d$ matrix. Let us call $r_j$, $1\leq j \leq d$ the rows of $B$, and $c_j$'s the columns of $B$. Suppose that 
\begin{enumerate}
\item the rows of $B$ are exchangeable and for any $j\neq k$ $(r_j,r_k)\equalInLaw (r_j,-r_k)$.
\item the columns of $B$ are exchangeable and for any $j\neq k$ $(c_j,c_k)\equalInLaw (c_j,-c_k)$.
\item for any $j$, $\scov{r_j}$ exists.
\end{enumerate}
Then we have 
\begin{enumerate}
\item $\Exp{r_j}=0$, for all $j$.
\item $\Exp{r_j r_k\trsp}=0_{d\times d}$ when $j\neq k$.
\item $\Exp{r_j r_j\trsp}=\gamma \id_d$ for some $\gamma$.
\end{enumerate}
\end{lemma}
\begin{proof}
Since we assume that $\scov{r_j}$ exists for any $j$, it is clear that $\Exp{r_j r_k\trsp}$ exists by the Cauchy-Schwarz inequality. 
Since $(r_j,r_k)\equalInLaw (r_j,-r_k)$ for $j\neq k$, we have 
$$
r_jr_k\trsp \equalInLaw -r_j r_k\trsp \text{ and } r_j\equalInLaw - r_j\;.
$$
Therefore, 
$$
\Exp{r_j r_k\trsp}=-\Exp{r_j r_k\trsp}=0_{d\times d} \text{ and } \Exp{r_j}=0\;.
$$
On the other hand, $\scov{r_j}(k,l)=\Exp{c_k(j)c_l(j)}$. If $k\neq l$, our assumption that for any $j\neq k$ $(c_j,c_k)\equalInLaw (c_j,-c_k)$ guarantees by the same argument as above that $\Exp{c_k(j)c_l(j)}=0$. On the other hand, since the columns are exchangeable, it is clear that $\scov{r_j}(k,k)=\Exp{c_k(j)^2}=\Exp{c_l(j)^2}=\scov{r_j}(l,l)$. So 
$$
\Exp{r_jr_j\trsp}=\gamma_j \id_d\;.
$$ 
Our assumption that the rows are exchangeable guarantees that for all $j\neq k$, $\gamma_j=\gamma_k=\gamma$.
\end{proof}

\section{Preliminaries for Theorem \ref{thm:replacementByGaussian}}

\begin{lemma}\label{lemma:detApproxBlockStieltjes}
Let us call $A_{12}$ be a $d\times (n-1)d$ real random matrix and $A_{21}=A_{12}\trsp$. Let us assume that $\exists R \in \mathbb{R}^+$ such that 
for any symmetric, real, deterministic matrix $\Gamma$,
\begin{align*}
\Exp{\opnorm{A_{12} (\Gamma-z\id)^{-1} A_{12}\trsp-\Exp{A_{12} (\Gamma-z\id)^{-1} A_{12}\trsp}}}&\leq \frac{R}{v}\;,\\
\Exp{\opnorm{A_{12} (\Gamma-z\id)^{-2} A_{12}\trsp-\Exp{A_{12} (\Gamma-z\id)^{-2} A_{12}\trsp}}}&\leq \frac{R}{v^2}\;,\text{ and }\\
\opnorm{\Exp{A_{12}(\Gamma-z\id)^{-2}A_{21}}}\leq K(z)\;,
\end{align*}
for a given function $K$, where $z\in \mathbb{C}^+$ with $\imag{z}=v>0$. 
For $\totalSize=nd$, let $T_n$ be the $N\times N$ matrix
$$
T_n=\begin{pmatrix}
0_{d\times d}&A_{12}\\
A_{21}&Z_{22}
\end{pmatrix}
$$
where $Z_{22}$ is a real, symmetric and deterministic matrix. 
Let 
$$
{\cal W}_n^{11}(z)=\left[-z\id_d-\Exp{A_{12}(Z_{22}-z\id)^{-1}A_{21}}\right]^{-1}\;,
$$
and
$$
{\cal L}(z)=\trace{(Z_{22}-z\id)^{-1}}-
\trace{\left[z\id_d+\Exp{A_{12} (Z_{22}-z\id)^{-1}A_{21}}\right]^{-1}\Exp{A_{12} (Z_{22}-z\id)^{-2}A_{21}}}.
$$ 
Then, under our assumptions,
$$
\left|\Exp{\trace{(T_n-z\id)^{-1}}-\trace{{\cal W}_n^{11}(z)}-{\cal L}(z)}\right|\leq  d(2+K(z))\frac{R}{v^3}\;.
$$
\end{lemma}

\begin{proof}
We call 
$$
(T_n-z\id)^{-1}=\begin{pmatrix}
T_n^{11}(z)&T_n^{12}(z)\\
T_n^{21}(z)&T_n^{22}(z)
\end{pmatrix}\;,
$$	
where $T_n^{11}(z)$ is $d\times d$ and $T_n^{22}(z)$ is $d(n-1)\times d(n-1)$.
Using the standard block inversion formula (see \cite{hj}, p.18), we see that the top-left $d\times d$ block of $(T_n-z\id)^{-1}$ is 
$$
T_n^{11}(z)=(-z\id_d-A_{12}(Z_{22}-z\id)^{-1}A_{21})^{-1}\;.
$$
We note that $(Z_{22}-z\id)^{-1}=S_{1}+iS_{2}$ where $S_{1}$ and $S_{2}$ are real symmetric matrices. Furthermore, after diagonalizing $Z_{22}$ it is clear that $S_2$ is positive semi-definite. So we have
$$
S_3=z\id_d+A_{12}(Z_{22}-z\id)^{-1}A_{21}=(A_{12}S_1 A_{12}\trsp+\myreal{z}\id)+i(A_{12}S_2 A_{12}\trsp+v\id)\;.
$$ 
Of course, the eigenvalues of $A_{12}S_2 A_{12}\trsp+v\id$ are greater than $v$: the matrix $A_{12}S_2 A_{12}\trsp$ is positive semi-definite. Therefore, by applying the Fan-Hoffman Theorem (Proposition III.5.1 in \cite{bhatia97}) to $-iS_3$, we see that the singular values of $S_3$ are all greater than $v$, so that 
$
\opnorm{S_3^{-1}}\leq \frac{1}{v}\;.
$
This shows that $\opnorm{T_n^{11}(z)}\leq 1/v$. The same argument also yields $\opnorm{{\cal W}_n^{11}(z)}\leq 1/v$.

Since $C^{-1}-D^{-1}=C^{-1}(D-C)D^{-1}$,
$$
T_n^{11}(z)-{\cal W}_n^{11}(z)=T_n^{11}(z)\left[A_{12}(Z_{22}-z\id)^{-1}A_{21}-\Exp{A_{12}(Z_{22}-z\id)^{-1}A_{21}}\right]{\cal W}_n^{11}(z)\;.
$$ 
So it is clear that 
$$
\Exp{\opnorm{T_n^{11}(z)-{\cal W}_n^{11}(z)}}\leq\frac{1}{v^2}\frac{R}{v}\;.
$$
By Weyl's majorant theorem (\cite{bhatia97}, Theorem II.3.6), we conclude that 
$$
\Exp{\left|\trace{T_n^{11}(z)-{\cal W}_n^{11}(z)}\right|}\leq \frac{d}{v^2}\frac{R}{v}\;,
$$
which clearly leads to 
$$
\left|\Exp{\trace{T_n^{11}(z)-{\cal W}_n^{11}(z)}}\right|\leq \frac{d}{v^2}\frac{R}{v}\;.
$$

Let us now work on 
$$
T_n^{22}(z)=\left(Z_{22}-z\id+\frac{1}{z}A_{21}A_{12}\right)^{-1}\;,
$$
the bottom-right $(n-1)d\times (n-1)d$ diagonal block of $(T_n-z\id)^{-1}$.
Since $A_{21}=A_{12}\trsp$ is $(n-1)d\times d$, we see that $A_{21}A_{12}$ is a rank-at-most-$d$ matrix of size $(n-1)d\times (n-1)d$. The Sherman-Morrison-Woodbury formula (\cite{hj}, p.19) gives, if $B=(C+\frac{1}{z}X X\trsp)$, with $C$ a $(n-1)d\times (n-1)d$ matrix and $X$ a $(n-1)d\times d$ matrix, 
$$
B^{-1}=C^{-1}-C^{-1}X(z\id_d+X\trsp C^{-1}X)^{-1}X\trsp C^{-1}\; ,
$$
and hence
$$
\trace{B^{-1}}-\trace{C^{-1}}=-\trace{(z\id_d+X\trsp C^{-1}X)^{-1}X\trsp C^{-2}X}\;.
$$
Note that inside the trace on the right-hand side we have two $d\times d$ matrix. For us $C=Z_{22}-z\id$ and $X=A_{21}$. We have seen above that 
$$
\opnorm{(z\id_d+A_{12} C^{-1}A_{21})^{-1}}\leq \frac{1}{v}\;.
$$
Hence, using our assumption on $\Exp{\opnorm{A_{12} (\Gamma-z\id)^{-2} A_{12}\trsp-\Exp{A_{12} (\Gamma-z\id)^{-2} A_{12}\trsp}}}$ as well as Weyl's majorant theorem, we have 
$$
\Exp{\left|\trace{(z\id_d+A_{12} C^{-1}A_{21})^{-1}\left[A_{21} C^{-2}A_{21}-\Exp{A_{12} C^{-2}A_{21}}\right]}}\right|\leq \frac{d}{v}\frac{R}{v^2}\;.
$$
Let us call $\Xi=\Exp{A_{12} C^{-2}A_{21}}$.
Of course, 
\begin{gather*}
\left[(z\id_d+A_{12} C^{-1}A_{21})^{-1}-(z\id_d+\Exp{A_{12} C^{-1}A_{21}})^{-1}\right]\Xi=\\
(z\id_d+A_{12} C^{-1}A_{21})^{-1}\left[\Exp{A_{12} C^{-1}A_{21}}-A_{12} C^{-1}A_{21}\right](z\id_d+\Exp{A_{12} C^{-1}A_{21}})^{-1}\Xi\;.
\end{gather*}
Therefore, since we have assumed that $\opnorm{\Xi}\leq K(z)$, we have 
$$
\opnorm{\left[(z\id_d+A_{12} C^{-1}A_{21})^{-1}-(z\id_d+\Exp{A_{12} C^{-1}A_{21}})^{-1}\right]\Xi}\leq \frac{K(z)}{v^2}\opnorm{\Exp{A_{12} C^{-1}A_{21}}-A_{12} C^{-1}A_{21}}\;.
$$
We have established that 
$$
\Exp{\left|\trace{\left[(z\id_d+A_{12} C^{-1}A_{21})^{-1}-(z\id_d+\Exp{A_{12} C^{-1}A_{21}})^{-1}\right]}\Exp{A_{12} C^{-2}A_{21}}\right|}\leq d\frac{K
(z)}{v^2}\frac{R}{v}\;.
$$
So we finally conclude that, if 
$$
\Delta=(z\id_d+A_{12} C^{-1}A_{21})^{-1}A_{12} C^{-2}A_{21}-(z\id_d+\Exp{A_{12} C^{-1}A_{21}})^{-1}\Exp{A_{12} C^{-2}A_{21}}\;,
$$
we have 
$$
\Exp{\left|\trace{\Delta}\right|}\leq \left(\frac{d}{v^2}+\frac{dK(z)}{v^2}\right)\frac{R}{v}\;.
$$
\end{proof}

We are now in position to state our ``strip-replacement'' theorem.
\begin{theorem}\label{thm:replaceFirstBlockRow}
Let us consider the $N \times N$ real symmetric matrices  
$$
T_n(A)=\begin{pmatrix}
0_{d\times d}&A_{12}\\
A_{21}&Z_{22}
\end{pmatrix}\;, \text{ and }
T_n(B)=\begin{pmatrix}
0_{d\times d}&B_{12}\\
B_{21}&Z_{22}
\end{pmatrix}\;,
$$
where, $A_{12},B_{12}\in\mathbb{R}^{d\times (N-d)}$, $A_{21}=A_{12}\trsp$ and $B_{21}=B_{12}\trsp$, and $Z_{22}$ is a deterministic symmetric matrix. 
Suppose $A_{12}$ and $B_{12}$ satisfy the assumptions of Lemma \ref{lemma:detApproxBlockStieltjes}. 
Suppose further that for any deterministic vector $u\in\mathbb{R}^{N-d}$, 
$$
\Exp{(A_{12}u)(A_{12}u)\trsp}=\Exp{(B_{12}u)(B_{12}u)\trsp}\;.
$$ 
Then 
\begin{equation}\label{eq:keyApproxCoro}
\left|\Exp{\trace{(T_n(A)-z\id)^{-1}}-\trace{(T_n(B)-z\id)^{-1}}}\right|\leq  d(2+K(z))\frac{2 R}{v^3}\;.
\end{equation}
The same is true when $Z_{22}$ is assumed to be random but independent of $A_{12}$ and $B_{12}$. 
\end{theorem}
\begin{proof}
The proof is essentially immediate once we realize that the assumption 
$$
\Exp{(A_{12}u)(A_{12}u)\trsp}=\Exp{(B_{12}u)(B_{12}u)\trsp}\;
$$
implies that, in the notation of the previous Lemma, the deterministic approximating quantity
$$
\trace{{\cal W}_n^{11}(z)}+{\cal L}(z)
$$
takes the same value since the expectations involving $A_{12}$ or $B_{12}$ are the same. 

The case of random $Z_{22}$ is treated by conditioning on $Z_{22}$ and getting an upper bound on the quantities we care about that does not depend on $Z_{22}$. 
\end{proof}

\section{Distributional results for various models}

\subsection{More details on the matrices drawn from $Sl(d,\RR)$}\label{subsec:CompsSLdAppendix}
We give more details about the stochastic properties of the matrices $B$ drawn according to the scheme described in Corollary \ref{coro:SLdCase}. 
\subsubsection{Computing the joint density of the singular values}
We consider the problem of understanding the singular values of the matrix 
$$
B=\frac{G}{|\det(G)|^{1/d}}\;,
$$
where $G$ is $d\times d$ with i.i.d Gaussian entries. 
We write an svd of $G$ as $G=UDV\trsp$. By rotational invariance arguments, it is clear that $(U,V)$ and $D$ are independent. Furthermore, $U$ and $V$ are Haar-distributed on $O(d)$. Note that defining $U$ and $V$ as svd-representatives may induce some mild dependence between them, because of sign issues. It is possible to deal with this dependence issue but we do not discuss it further as our interest here is in singular values.

Let us call $s_i$ the singular values of $B$. Recall that $s_i\geq 0$. For simplicity, we seek to understand not the joint distribution of $s_i$'s but that of $s_i^2$. Note that if $d_i$'s are the singular values of $G$, we have the relationship
$$
s_i=\frac{d_i}{(\prod_{1\leq i \leq d}d_i)^{1/d}}\;,
$$
since $|\det(G)|=\sqrt{\det(G\trsp G)}=\sqrt{\prod_{1\leq i \leq d} d_i^2}$.
In particular, we have, if we denote by $l_i$'s are the eigenvalues of $G\trsp G$, 
$$
s_i^2=\frac{l_i}{(\prod_{1\leq i \leq d}l_i)^{1/d}}\;.
$$

We have the following fact:
\begin{fact}
The joint density of $l_1>l_2>\ldots>l_d$ is
\begin{equation}\label{eq:JointDensityWishart}
f(l_1,\ldots,l_d)=C(d)\exp\left(-\frac{1}{2}\sum_{i=1}^d l_i\right) \prod_{i=1}^d l_i^{-1/2}\prod_{i<j}(l_i-l_j)\;.
\end{equation}
\end{fact}
\begin{proof}
We note that $G\trsp G$ is Wishart-distributed, specifically $\Wishart(d,\id_d)$. The fact we mention is therefore just the content of Corollary 3.2.19 in \cite{muirhead82}. The value of $C(d)$ is known explicitly:
$$
C(d)=\frac{\pi^{d^2/2}}{2^{d^2/2}\left[\Gamma_d(d/2)\right]^2}\;, \quad\text{where}\quad\Gamma_d(x)=\pi^{d(d-1)/4}\prod_{i=1}^d \Gamma[x-(i-1)/2]\;,
$$
provided $\myreal{x}>(m-1)/2$ and $\Gamma$ is the ordinary Gamma function. For details, see \cite{muirhead82}, pp.61-62.
\end{proof}

We now assume that $s_1>s_2>\ldots>s_d$. We call 
\begin{align*}
y_i&=s_i^2\;, 1\leq i \leq d \;,\\
t&=\left(\prod_{1\leq i \leq d} l_i\right)^{1/d} \;.
\end{align*}

We would like to find $g(y_1,\ldots,y_{d-1})$, the density of the $d-1$ largest eigenvalues of $B\trsp B$. Note that $\det(B\trsp B)=1$, so $\prod_{1\leq i \leq d}y_i=1$. 
Note also that if we keep the ordering, we must have $\left(\prod_{1\leq i \leq d-1} y_i\right) y_{d-1}>1$ to guarantee that there exists $y_d< y_{d-1}$ such that $\prod_{1\leq i \leq d} y_i=1$. This defines the subset of $\mathbb{R}^{d-1}$ where $(y_i)_{i=1}^{d-1}$ lives.

\begin{lemma}\label{lemma:jointDensityOrderedSldEigs}
Let us call $\tilde{y}$ the vector $(y_1,\ldots,y_{d-1})$, where $y_1>y_2>\ldots>y_{d-1}>0$ and $\prod_{1\leq i \leq d-1}y_i> 1/y_{d-1}$\;. We call ${\cal R}$ this subset of $\mathbb{R}^{d-1}$. 
Let $\alpha=1/\prod_{1\leq i \leq d-1}y_i$ and 
\begin{align*}
\gamma(\tilde{y})&=\frac{1}{2}\left(\sum_{1\leq i \leq d-1} y_i+\alpha\right)\;,\\
R(\tilde{y})&=\alpha  \prod_{1\leq i<j\leq d-1}(y_i-y_j) \prod_{1\leq i \leq d-1} (y_i-\alpha)\;.
\end{align*}	
Then, the density of $\tilde{y}$ over ${\cal R}$ is 
$$
g(y_1,\ldots,y_{d-1})=\tilde{C}(d) \frac{R(\tilde{y})}{[\gamma(\tilde{y})]^{d^2/2}}\;.
$$
\end{lemma}

\begin{proof}[\textbf{Proof of Lemma \ref{lemma:jointDensityOrderedSldEigs}:}]
To find the density, we will use the following change of variables from $(l_1,\ldots,l_d)$ to $(y_1,\ldots,y_{d-1},t)$:
\begin{align*}
l_i&= t y_i\;, 1\leq i \leq d-1\;,\\
l_d&=\frac{t}{\prod_{1\leq i \leq d-1} y_i}\;.
\end{align*}
We call $\alpha=\frac{1}{\prod_{1\leq i \leq d-1} y_i}$.
Let us call $\tilde{y}$ the $d-1\times 1$ vector with $i$-th entry $y_i$. $1/\tilde{y}$ is the $(d-1)\times 1$ vector with $i$-th entry $1/y_i$. 
The Jacobian matrix for the change of variables we just discussed is 
$$
M=\begin{pmatrix}
t\Id_{d-1}& \tilde{y}\\
-t\alpha/\tilde{y} & \alpha
\end{pmatrix}
$$
By multilinearity of the determinant, we therefore have 
$$
\det(M)=\alpha \det\begin{pmatrix}t\Id_{d-1}& \tilde{y}\\
-t/\tilde{y} & 1
 \end{pmatrix}=\alpha t^{d-1} \det\begin{pmatrix}\Id_{d-1}& \tilde{y}\\
-1/\tilde{y} & 1
 \end{pmatrix}\;.
$$
Now, let us call $\mathsf{y}$ the $d\times 1$ vector such that 
$$
\mathsf{y}=\begin{pmatrix} \tilde{y}\\ 0 \end{pmatrix}\;.
$$
And let $\mathsf{1/y}$ be the vector such that 
$$
\mathsf{1/y}=\begin{pmatrix} 1/\tilde{y}\\ 0 \end{pmatrix}\;.
$$
We have, if $e_d$ denotes the $d$-th canonical basis vector, 
$$
\begin{pmatrix}\Id_{d-1}& \tilde{y}\\
-1/\tilde{y} & 1\end{pmatrix}=\id_d + \mathsf{y}e_d\trsp -e_d \mathsf{1/y}\trsp\;.
$$
From determinant theory (\cite{ggk}, Theorem I.3.2, p.9), we know that 
$$
\det(\id+\sum_{1\leq i \leq m} \phi_i \otimes f_i)=\det(\delta_{i,j}+\langle \phi_i,f_j\rangle)_{1\leq i,j\leq m}\;.
$$
In the circumstances of interest to us, we have 
$$
\left(\langle\phi_i,f_j\rangle\right)_{1\leq i,j\leq 2}=\begin{pmatrix}0&-(d-1)\\1&0\end{pmatrix}\;.
$$
So we conclude that 
$$
\det\begin{pmatrix}\Id_{d-1}& \tilde{y}\\
-1/\tilde{y} & 1
 \end{pmatrix}=1+(d-1)=d\;.
$$
Hence, the Jacobian of our change of variable is 
$$
J=\alpha t^{d-1} d.
$$
We conclude that the density of $(y_1,\ldots,y_{d-1},t)$ is 
$$
h(y_1,\ldots,y_{d-1},t)=\frac{1}{\prod_{1\leq i \leq d-1} y_i} t^{d-1} f\left(ty_1,\ldots,ty_{d-1},\frac{t}{\prod_{1\leq i \leq d-1} y_i}\right)\;.
$$
Now, 
\begin{align*}
f(ty_1,\ldots,ty_{d-1},t\alpha)=&\,C(d)\exp\left(-\frac{t}{2}\left[\sum_{1\leq i \leq d-1} y_i+\alpha\right]\right) t^{-d/2} \\
&\times\left[\prod_{1\leq i \leq d-1} y_i^{-1/2} \right]\alpha^{-1/2} t^{d(d-1)/2}
\prod_{1\leq i<j\leq d-1}(y_i-y_j) \prod_{1\leq i \leq d-1} (y_i-\alpha)\;.
\end{align*}
Therefore,
\begin{align*}
h(y_1,\ldots,y_{d-1},t)=&\,C(d)\alpha t^{d-1-d/2+d(d-1)/2}\exp\left(-\frac{t}{2}\left[\sum_{1\leq i \leq d-1} y_i+\alpha\right]\right)\\
&\times \prod_{1\leq i<j\leq d-1}(y_i-y_j) \prod_{1\leq i \leq d-1} (y_i-\alpha)\\
=&\,C(d)t^{d^2/2-1} \exp(-t\gamma(\tilde{y})) R(\tilde{y})\;,
\end{align*}
where 
\begin{align*}
\gamma(\tilde{y})&=\frac{1}{2}\left(\sum_{1\leq i \leq d-1} y_i+\alpha\right)\;,\\
R(\tilde{y})&=\alpha  \prod_{1\leq i<j\leq d-1}(y_i-y_j) \prod_{1\leq i \leq d-1} (y_i-\alpha)\;.
\end{align*}

The joint density of $(y_1,\ldots,y_{d-1})$ is simply:
$$
g(y_1,\ldots,y_{d-1})=\int_0^\infty h(y_1,\ldots,y_{d-1},t) dt\;.
$$
Note that $\sum_{1\leq i \leq d-1} y_i+\alpha>0$ in the domain we consider, so there are no integrability problems.  
Also, if $K$ is an integer,
$$
\int_0^{\infty} t^{K-1} \exp(-\beta t) dt=\Gamma(K)\beta^{-K}=(K-1)! \beta^{-K}\;.
$$
Therefore, we finally have, for $y_1>y_2>\ldots>y_{d-1}$, 
$$
g(y_1,\ldots,y_{d-1})=\tilde{C}(d) \frac{R(\tilde{y})}{[\gamma(\tilde{y})]^{d^2/2}}\;.
$$
The Lemma is shown. 
\end{proof}

Let us apply the Lemma in the case $d=2$. 

\begin{corollary}[Case $d=2$]\label{coro:SL2svsAreLikeCauchy}
Then, $\gamma(\tilde{y})=\frac{1}{2}(y_1+1/y_1)$ and $R(\tilde{y})=(1-1/y_1^2)$\;. So, for $y_1>1$,  
$$
g(y_1)=C\frac{1-y_1^{-2}}{(y_1+1/y_1)^2}\sim_{\infty} C y_1^{-2}\;.
$$
Therefore, $y_1=s_1^2$ has a $1-\eps$ moment for any $\eps>0$, but not 1 moment. In other words, 
$$
\Exp{s_1^{2-\eps}}<\infty \text{ if } \eps>0\;,
$$
and $\Exp{s_1^2}=\infty$. 
\end{corollary}
This corollary shows that the square of the largest singular value of $G/|\det(G)|^{1/d}$ has Cauchy-like behavior in the tail.

\subsubsection{On the entries of $B\trsp B$}
We recall the famous Bartlett decomposition of a Wishart matrix (see \cite{muirhead82}, p.99).
\begin{theorem}[Bartlett Decomposition]\label{thm:BartlettDecomp}
Let $A$ be $\Wishart_p(n,\id_p)$, with $n\geq p$ and write $A=T\trsp T$, where $T$ is an upper-triangular $p\times p$ matrix with positive diagonal elements. Then the elements of $T$ are all independent, $T_{i,i}^2$ is $\chi^2_{n-i+1}$, for $1\leq i \leq p$, and $T_{i,j}$ is ${\cal N}(0,1)$ for $1\leq i<j \leq p$.
\end{theorem}

From now on, we call $T$ the upper-triangular matrix appearing in the Bartlett decomposition of $G\trsp G$. 
We have the following lemma. 

\begin{lemma}\label{lemma:StochRepEntriesStS}
We have 
$$
B\trsp B=\tilde{T}\trsp \tilde{T}\;,
$$
where 
$$
\tilde{T}=\frac{T}{\det(T)^{1/d}}=\frac{T}{\prod_{i=1}^d T_{i,i}^{1/d}}\;.
$$
$T_{i,i}^2$ are independent and have distribution $\chi^2_{d-i+1}$, for $1\leq i \leq d$. 
Calling $T_i$ the $i$-th column of $T$, we therefore have 
$$
(B\trsp B)_{i,j}=\tilde{T}_i\trsp \tilde{T}_j\;.
$$
\end{lemma}

If $i\leq j$, we have in particular 
$$
(B\trsp B)_{i,j}=\frac{\sum_{k\leq i}T_{k,i}T_{k,j}}{\prod_{l=1}^d T_{l,l}^{2/d}}\;.
$$
More specifically, 
\begin{enumerate}
\item when $i<j$,  
$$
(B\trsp B)_{i,j}=\sum_{k<i}\frac{T_{k,i}T_{k,j}}{\prod_{l=1}^d T_{l,l}^{2/d}}+\frac{T_{i,i}^{1-2/d}T_{i,j}}{\prod_{l\neq i} T_{l,l}^{2/d}}\;.
$$
(Because of independence properties of the $T_{i,j}$'s, computations of moments for $(B\trsp B)_{i,j}$ is relatively simple.) 
\item when $i=j$, 
$$
(B\trsp B)_{i,i}=\sum_{k<i}\frac{T_{k,i}^2}{\prod_{l=1}^d T_{l,l}^{2/d}}+\frac{T_{i,i}^{2-2/d}}{\prod_{l\neq i} T_{l,l}^{2/d}}\;.
$$
\end{enumerate}

In particular, when $d=2$, $(B\trsp B)_{1,1}=\frac{T_{1,1}}{T_{2,2}}=\frac{\chi_2}{\chi_1}$, where the two $\chi$ random variables are independent. Since a Cauchy random variable is the ratio of two independent $\chi_1$ random variables we conclude that $(B\trsp B)_{1,1}$ is stochastically larger than a Cauchy random variable.

\subsection{Properties of the null distribution for the class averaging algorithm}\label{app:subsec:vdmComps}

We now consider the distribution of 
\begin{equation}\label{invariance:distance:group:action}
g_{ij}=\argmin_{g \in SO(2)}\|Z_i-g\circ Z_j\|_2^2\;,
\end{equation}
where $Z_i$ are functions defined on $\RR^2$ and $g\circ Z_i$ is defined as a new function on $\RR^2$ as $(g\circ Z)((x,y)^T):=Z(g(x,y)^T))$, $(x,y)^T\in \RR^2$. 
Intuitively, the dependence among different $g_{ij}$ entries is obvious. However, how the dependence among each entry is not that clear without a careful analysis. In particular, the uniform distribution of $g_{ij}$ needs to be carefully addressed.

\begin{lemma}[Null Case for the Class Averaging Algorithm]\label{lemma:distgijNullCase}
Suppose that $Z_i$ and $Z_j$ are independent. Suppose that each random variable has a distribution that is invariant under the action of $SO(2)$. 
Then $g_{ij}$ and $Z_i$ are independent and so are $g_{ij}$ and $Z_j$. Furthermore, $g_{ij}$ is uniformly distributed on $SO(2)$. 
\end{lemma}

\begin{proof}
Note that conditional on $Z_i$, if $Z_j\rightarrow O^{-1}\circ Z_j$, where $O\in SO(2)$, then $g_{ij}\rightarrow Og_{ij}$ by (\ref{invariance:distance:group:action}). 
Hence, using the assumption that $Z_j\equalInLaw O\circ Z_j$ (i.e the law of $Z_j$ is invariant under the action of $SO(2)$) and $Z_j|Z_i\equalInLaw O\circ Z_j|Z_i$ (this latter equality coming from independence of $Z_i$ and $Z_j$), we see that 
$$
g_{ij}|Z_i\equalInLaw Og_{ij}|Z_i\;.
$$
Since the only distribution on $SO(2)$ that is invariant by left-multiplication by an $SO(2)$ is the uniform distribution on $SO(2)$, we conclude that $g_{ij}|Z_i$ has the uniform distribution on $SO(2)$.

Because the uniform distribution on $SO(2)$ does not depend on $Z_i$, $g_{ij}$ and $Z_i$ are independent. Indeed, let $\Gamma$ be a function of $Z_i$ and $\omega$ be a function of $g_{ij}$. Note that since the distribution of $g_{ij}|Z_i$ does not depend on $Z_i$, we have 
$$
\EE(\omega(g_{ij})|Z_i)\triangleq \Omega=\EE{\omega(g_{ij})}\;.
$$
In other words, $\Omega$ is a constant (in particular, it does not depend on $Z_i$).  Therefore, we have
\begin{align*}
\EE\left[\omega(g_{ij})\Gamma(Z_i)\right]&=\EE\left[\EE(\omega(g_{ij})\Gamma(Z_i)|Z_i)\right]=\EE\left[\Gamma(Z_i)\EE(\omega(g_{ij})|Z_i)\right]\\
&=\EE\left[\Gamma(Z_i)\right]\Omega=\EE\left[\Gamma(Z_i)\right]\EE\left[\omega(g_{ij})\right]\;.
\end{align*}
The same argument shows that $g_{ij}$ is also independent of $Z_j$. So we have established that $g_{ij}$ and $Z_i$ are independent. The same argument shows that $g_{ij}$ and $Z_j$ are independent. However the three random variables $g_{ij}$, $Z_i$ and $Z_j$ are not \emph{jointly} independent.
\end{proof}

The previous lemma has the following useful consequence.

\begin{lemma}\label{lemma:independencegijandgik}
Suppose that $Z_i$, $Z_j$ and $Z_k$ are independent, each random variable having a distribution that is invariant under the action of $SO(2)$. 
Then $g_{ij}$ and $g_{ik}$ are independent, and so are $g_{ij}$ and $g_{jk}$. 
Furthermore, the random variables $\{g_{ij}\}_{j=1}^n$ are jointly independent.
\end{lemma}
\begin{proof}
Let $f_1$ and $f_2$ be two functions. 
We have 
$$
\Exp{f_1(g_{ij})f_2(g_{ik})}=\Exp{\Exp{f_1(g_{ij})f_2(g_{ik})|Z_i}}\;.
$$
Now, it is clear that $g_{ij}|Z_i$ is a function of $Z_j$ only. Similarly, $g_{ik}|Z_i$ is a function of $Z_k$ only. So $g_{ij}|Z_i$ is independent of $g_{ik}|Z_i$. Therefore, 
$$
\Exp{f_1(g_{ij})f_2(g_{ik})|Z_i}=\Exp{f_1(g_{ij})|Z_i}\Exp{f_2(g_{ik})|Z_i}\;.
$$
Now recall that we have shown that $g_{ij}|Z_i \sim U$, where $U$ is a uniformly distributed random variable on $SO(2)$; the same result applies to $g_{ik}$. Therefore, 
$$
\Exp{f_1(g_{ij})|Z_i}\Exp{f_2(g_{ik})|Z_i}=\Exp{f_1(U)}\Exp{f_2(U)}\;.
$$
Of course, our argument above shows that $\Exp{f_1(g_{ij})}=\Exp{f_1(U)}$. 
We conclude that 
\begin{align*}
\Exp{f_1(g_{ij})f_2(g_{ik})}&=\Exp{f_1(U)}\Exp{f_2(U)}\;,\\
&=\Exp{f_1(g_{ij})}\Exp{f_2(g_{ik})}\;.
\end{align*}
This shows that $g_{ij}$ and $g_{ik}$ are independent. The proof of joint independence of $\{g_{ij}\}_{j=1}^n$ follows exactly in the same manner: just start the proof with $f_1,\ldots,f_n$ and apply the same reasoning.

Our statement concerning $g_{ij}$ and $g_{jk}$ is also proven in a similar manner, by writing 
$$
\Exp{f_1(g_{ij})f_2(g_{jk})}=\Exp{\Exp{f_1(g_{ij})f_2(g_{jk})|Z_j}}\;,
$$
and using the fact that $g_{ij}$ and $g_{jk}$ are independent conditionally on $Z_j$. The rest of the argument is similar to the one we gave above. 
\end{proof}

So we have established some pairwise independence results, but we do not have joint independence for the three random variables $(g_{ij},g_{ik},g_{jk})$ or the random variables $\left(g_{ij}\right)_{i< j}$.

\section*{Acknowledgment}
Noureddine El Karoui gratefully acknowledges the support from NSF grant DMS-0847647 (CAREER). Hau-Tieng Wu gratefully acknowledges the support from AFOSR grant FA9550-09-1-0643. The authors also like to thank anonymous referees for their constructive comments that led to a substantial improvement of the paper. 

\bibliographystyle{plain}
\bibliography{noisyManifold}

\end{document}